\documentclass[11pt,reqno]{amsart}
\usepackage{fullpage}

\usepackage{graphicx}
\usepackage[draft]{hyperref}
\usepackage{amsmath,amsopn,amssymb,amsfonts,stmaryrd}
\usepackage{verbatim}
\usepackage{amsthm}
\usepackage{mathtools}
\usepackage{color}
\usepackage{enumitem}
\usepackage[framemethod=TikZ]{mdframed}
\usepackage{bbm}
\usepackage{mathrsfs}
\usepackage{booktabs}
\usepackage{caption}
\usepackage{bm}
\usepackage{tensor}
\usepackage{cancel}
\usepackage{float}

\usepackage{mathrsfs}
\usepackage[scr=boondox]{mathalfa}

\newcommand{\IR}{{\mathbb R}}%Reals
\newcommand{\IC}{{\mathbb C}}%Complex

\newcommand{\IP}{{\mathbb P}}%Projective
\newcommand{\IZ}{{\mathbb Z}}%Integers

\newcommand{\IN}{{\mathbb N}}%Natural numbers
\newcommand{\IH}{{\mathbb H}}%quaternions
\newcommand{\CC}{\mathscr{C}}
\newcommand{\CD}{\mathscr{D}}

\newcommand{\calF}{\mathcal{F}}

\newcommand{\calI}{\mathcal{I}}
\newcommand{\calJ}{\mathcal{J}}
\newcommand{\calK}{\mathcal{K}}
\newcommand{\calL}{\mathcal{L}}

\newcommand{\calN}{\mathcal{N}}

\newcommand{\calT}{\mathcal{T}}

\renewcommand{\Re}{\mathrm{Re}}

\newcommand{\sgn}{\mbox{sgn}}
\newcommand{\vc}[1]{\bm{#1}}

\newcommand{\U}{\mathrm{U}}

\newcommand{\SL}{\mathrm{SL}}

\theoremstyle{plain}
\newtheorem{thm}{Theorem}[section]
\newtheorem{cor}[thm]{Corollary}
\newtheorem{lem}[thm]{Lemma}

\newtheorem{prop}[thm]{Proposition}

\newtheorem{rem}[thm]{Remark}

\theoremstyle{definition}

\numberwithin{equation}{section}

\newcommand{\pmat}[1]{\left( \smallmatrix #1 \endsmallmatrix \right)}
\newcommand{\mat}[1]{\left( \begin{matrix} #1 \end{matrix} \right)}

\renewcommand{\sgn}{\textnormal{sgn}}

\def\lp{\left(}
\def\rp{\right)}

\def\a{\alpha}
\def\b{\beta}
\def\d{\delta}

\def\k{\kappa}
\def\l{\lambda}

\def\w{\omega}

\def\z{\zeta}
\def\th{\theta}
\def\vth{\vartheta}

\def\e{\varepsilon}

\def\s{\sigma}
\def\g{\gamma}
\def\t{\tau}

\def\GG{\Gamma}

\def\LL{\Lambda}
\def\TH{\Theta}

\def\del{ \partial}

\newcommand{\re}{{\rm Re}}
\newcommand{\im}{{\rm Im}}

\renewcommand{\sgn}{{\rm sgn}}

\def\wh{\widehat}
\def\bar{\overline}
\newcommand{\andd}{\quad \mbox{ and } \quad}
\newcommand{\where}{\quad \mbox{ where }}

\makeatletter
\let\ORI@left\left
\protected\def\left#1{\mathopen{}\ORI@left#1}
\makeatother

\setlist[itemize]{leftmargin=*}

\makeatletter
\newcommand{\vast}{\bBigg@{2}}
\newcommand{\Vast}{\bBigg@{5}}
\makeatother

\renewcommand{\pmod}[1]{\ \left( \mathrm{mod} \, #1 \right)}

\newcommand{\ddiv}[2]{\d_{\smash{#1} \equiv 0 \pmod{#2}}}
\newcommand{\Lreg}[1]{\calL^{\mathrm{reg}}_{#1}}
\newcommand{\Lsing}[1]{\calL^{\mathrm{sing}}_{#1}}

\newcommand{\ccc}{\kappa}
\newcommand{\ff}{\mathscr{f}}
\newcommand{\hh}{\mathscr{h}}

\newcommand{\musqmsmall}{\{ - \frac{\mu_2^2}{4m_2} \}}
\newcommand{\musqm}{\scalebox{0.65}{$\left\{  - \frac{\mu_2^2}{4m_2} \right\}$}}
\newcommand{\musqmbig}{\left\{ \! - \frac{\mu_2^2}{4m_2} \right\}}
\newcommand{\psum}{\sideset{}{^*}\sum}

\makeatletter
\newcommand{\nolisttopbreak}{\par\nobreak\@afterheading} 
\makeatother

\makeatletter
\def\l@subsection{\@tocline{2}{0pt}{2.5pc}{5pc}{}}
\makeatother

\allowdisplaybreaks

\author{Kathrin Bringmann and Caner Nazaroglu}
\address{University of Cologne, Department of Mathematics and Computer Science, Weyertal 86-90, 50931 Cologne, Germany}
\email{kbringma@uni-koeln.de}
\email{cnazarog@uni-koeln.de}

\makeatletter
\@namedef{subjclassname@2020}{%
\textup{2020} Mathematics Subject Classification}
\makeatother

\title{Depth Two Mock Modularity by Eisenstein Series Coupling}

\begin{document}

\begin{abstract}
The notion of depth two and higher mock modular forms have found important applications in mathematical physics and enumerative geometry since their inception through indefinite theta functions with general signature. These theta functions generalize Zwegers' work on Lorentzian signature lattices and the framework of mock modular forms that emanated from it. Mock modular forms can also be studied through Eisenstein and Poincar\'e series. The interaction of this second point of view with the indefinite theta function approach yields a wealth of tools to unearth the rich structure behind mock modular forms. For mock modular forms of higher depth, on the other hand, indefinite theta functions and their variants largely remained the only available approach. In this paper, we show that one can indeed get mock modular forms of depth two by ``coupling" a pair of Eisenstein series that yield depth one mock modular forms, thereby providing a new and independent approach to higher depth mock modular forms. We exemplify this new perspective on a depth two object that appeared in the context of Vafa--Witten invariants.

\vspace{-0.4cm}

\end{abstract}

\maketitle

\section{Introduction and Statement of Results}\label{sec:introduction}

The concept of mock modular forms originates from the modern understanding of Ramanujan’s mock theta functions through indefinite theta functions studied by Zwegers \cite{Zwe} and harmonic Maass forms introduced by Bruinier and Funke \cite{BF}. Broadly speaking, a \textit{(mixed) mock modular form} of weight $\k$, as defined by \cite{DMZ}, is a holomorphic function $h$ on $\IH$, which is not modular, but can be completed to a real-analytic modular form by adding products of modular forms and \textit{non-holomorphic Eichler integrals} of modular forms. More specifically, this real-analytic modular completion $\wh{h}$ transforms with weight $\k$ under modular transformations (perhaps under a subgroup of $\SL_2 (\IZ)$ and with some multiplier) and satisfies\footnote{
Here and throughout we write $\t = u + i v$ with $u,v \in \IR$.
We also retain the explicit dependence on $\t$ and $\bar{\t}$.
These modular completions are specializations of bimodular forms, where the two variables are independent.
}
\begin{equation}\label{eq:mock_modular_form_condition}
\frac{\del}{\del \bar{\t}} \wh{h} (\t,\bar{\t})
\in \bigoplus_j v^{r_j} \mathfrak{M}_{\k+r_j} \otimes \bar{\mathfrak{M}_{2+r_j}} .
\end{equation}
Here $\mathfrak{M}_{r}$ denotes the vector space of (weakly) holomorphic modular forms of weight $r$ (for an appropriate subgroup of $\SL_2 (\IZ)$ and multiplier).\footnote{
One chooses the growth conditions for the involved (mock) modular objects depending on the specific context. Our interest, however, is in the general framework and we avoid this specialization.
}
Indefinite theta functions on signature $(n-1,1)$ lattices \cite{Zwe} fit into this structure. The instances where only the case $r_j = - \k$ arises and the contribution from $\mathfrak{M}_0$ is a constant have completions corresponding to harmonic Maass forms. This latter case is where the use of Eisenstein and Poincar\'e series has been fruitful.

In fact, one of the earlier examples of functions that fits into this structure is Zagier's weight~$\frac{3}{2}$ Eisenstein series \cite{HZ, Zag} that yields a mock modular form whose Fourier coefficients are given by Hurwitz class numbers. For the vector-valued adaptation we use in this paper, these Eisenstein series are defined as follows (for $m \in \IN$, $\mu \in \IZ/2m\IZ$, $\t \in \IH$, and $s \in \IC$ with $\re(s) > \frac{1}{4}$)
\begin{equation}\label{eq:depth_one_eisenstein}
\wh{f}_{m,\mu} (\t, \bar{\tau}; s) := 
\frac{v^s}{2} 
\sum_{(c,d) \in \LL}
\frac{\Psi_m (M_{c,d})_{0,\mu}}{|c \t + d|^{2s} (c \t + d)^{\frac{3}{2}}}.
\end{equation}
Here and throughout we use the principal value of the logarithm to define general powers. 
The~Weil multiplier $\Psi_m$ is given in Section~\ref{sec:Weil_multiplier_review}, $M_{c,d}$ denotes an element of $\SL_2(\IZ)$ with lower row $(c,d)$, and 
\begin{equation*}
\LL := \{ (c,d) \in \IZ^2 \setminus \{(0,0)\} : \ \gcd (c,d) = 1 \} .
\end{equation*}
The case $m=1$ corresponds to \cite{Zag} and $s \mapsto \wh{f}_{1,\a} (\t, \bar{\tau}; s)$ analytically continues to $s=0$. There it gives the modular completion $\wh{f}_{1,\a} (\t, \bar{\t})$ of a mock modular form $f_{1,\a} (\t)$ with
\begin{equation}\label{eq:f1_fourier_coef_hurwitz}
f_{1,0} (\t) =
-12 \sum_{n\geq 0} H(4n)  q^{n} 
\andd
f_{1,1} (\t) =
-12 \sum_{n\geq 0} H(4n+3)  q^{n+\frac{3}{4}} .
\end{equation}
Thus these yield the usual generating function of Hurwitz class numbers $H(N)$, rescaled by $-12$ so that its constant term is $1$:
\begin{equation}\label{eq:f1_fourier_coef_hurwitz_examples}
\mat{f_{1,0} (\t) \\ f_{1,1} (\t)}
=
\mat{1 - 6q - 12q^2 - 16q^3 - 18q^4 - 24q^5 - 24q^6 - 24q^7 - 36q^8 + \ldots  
\\ -4q^{\frac34} - 12q^{\frac74} - 12q^{\frac{11}4} - 24q^{\frac{15}4} - 12q^{\frac{19}4} - 36q^{\frac{23}4} - 16q^{\frac{27}4} + \ldots} .
\end{equation}
The modular completion of $f_{1,\a}$ is given by (with $\a \in \IZ/2\IZ$) 
\begin{equation}\label{eq:f1_modular_completion}
\wh{f}_{1,\a} (\t, \bar{\tau}) = f_{1,\a} (\t)
+ \frac{3i}{\pi \sqrt{2}} \int_{-\bar{\t}}^{i \infty} 
\frac{\vth_{1, \a} (w)}{(-i(w+\t))^{\frac{3}{2}}} dw ,
\end{equation}
where we recall \textit{Jacobi theta functions} (for $m\in \IN$, $\mu \in \IZ/2m\IZ$, and $\t \in \IH$)\footnote{
Note that $(\vth_{m, \mu} (\t))_{\mu \in \IZ/2m\IZ}$ forms a vector-valued holomorphic modular form of weight~$\frac{1}{2}$ and multiplier $\Psi_m$:
\begin{equation*}
\vth_{m, \mu} \lp \frac{a \t + b}{c \t + d} \rp
=  
\sqrt{c \t + d} \sum_{\nu \in \IZ/2m\IZ} \hspace{-0.3cm}
\Psi_{m}(M)_{\mu,\nu} 
\vth_{m, \nu} (\t) 
\quad \mbox{for } 
M := \mat{a & b \\ c & d} \in \SL_2(\IZ).
\end{equation*}
\vspace{-0.6cm}
}
\begin{equation}\label{eq:vth_definition}
\vth_{m, \mu} (\t) := \sum_{\ell \in \IZ + \frac{\mu}{2m}} e^{2 \pi i m \ell^2 \t} .
\end{equation}
Specifically, the completion transforms with the dual of the Weil multiplier for the $A_1$-lattice,\footnote{
Shintani \cite{Shin} derived essentially the same modular transformations from the perspective of zeta functions associated to prehomogeneous vector spaces~\cite{SatoShin}.
Note that \cite{Shin} also contains a discussion of the positive discriminant case.
This was studied by Duke, Imamoglu, and T\'oth \cite{DIT} from another perspective, leading to the notion of sesquiharmonic Maass forms \cite{BDR}. These yield a distinct but similar notion of mock modularity of higher depths.
}
\begin{equation*}
\wh{f}_{1,\a} \lp \frac{a\t+b}{c \t+d}, \frac{a\bar{\t}+b}{c\bar{\t}+d} \rp
=
(c \t + d)^{\frac{3}{2}} \sum_{\b \in \IZ/2\IZ} 
\Psi^*_{1}(M)_{\a,\b} 
\wh{f}_{1,\b} \lp \t, \bar{\tau} \rp 
\quad \mbox{for } 
M := \mat{a & b \\ c & d} \in \SL_2(\IZ).
\end{equation*}
It satisfies \eqref{eq:mock_modular_form_condition} with $f_{1,\a}$ having the shadow\footnote{
We call the modular forms appearing on the right-hand side of \eqref{eq:mock_modular_form_condition} the \textit{shadow} of the mock modular form. 
The constant prefactors are largely convention dependent, so we do not fix one and instead simply display the $\bar{\t}$-derivative. 
}
\begin{equation*}
(2 v)^{\frac{3}{2}} \frac{\del}{\del \bar{\t}}
\wh{f}_{1,\a} (\t, \bar{\tau}) 
=
\frac{3i}{\pi \sqrt{2}}
\vth_{1, \a} (-\bar{\t}).
\end{equation*}

This generalizes to other cases of $m$, $\mu$, as we review in Section \ref{sec:weight_3_2_eisenstein}. In particular, restricting to odd, square-free $m$ for technical simplicity, the function $s \mapsto \wh{f}_{m,\mu} (\t, \bar{\tau}; s)$ analytically continues to $s=0$ and its value $\wh{f}_{m,\mu} (\t, \bar{\tau})$ there is the modular completion of a mock modular form $f_{m,\mu} (\t)$~as
\begin{equation*}
\wh{f}_{m,\mu} (\t, \bar{\tau}) = f_{m,\mu} (\t)
+ C_m \int_{-\bar{\t}}^{i \infty} 
\frac{\TH_{m, \mu} (w)}{(-i(w+\t))^{\frac{3}{2}}} dw .
\end{equation*}
Here we define (with $\s$ denoting the sum of divisors function)
\begin{equation}\label{eq:C_m_and_TH_m_mu_definition}
C_m := \frac{3i}{\pi \s(m)} \sqrt{\frac{m}{2}} 
\andd
\TH_{m, \mu} (\t) := \sum_{d \mid m} \vth_{m, \xi_{d,m} \mu} (\t),
\end{equation}
where $\xi_{d,m}$ realize the Atkin--Lehner operators of \cite{EZ}.
The function $(\vth_{m, \xi_{d,m} \mu} (\t))_{\mu \in \IZ/2m\IZ}$, in particular, forms a weight $\frac{1}{2}$ vector-valued holomorphic modular form transforming with the same multiplier as $(\vth_{m,\mu} (\t))_{\mu \in \IZ/2m\IZ}$, namely $\Psi_m$; see Section~\ref{sec:atkin_lehner_operators} for further details.

In this paper we extend these points to a class of functions that generalize mock modular forms. Recall that Zwegers' work \cite{Zwe} was later extended to indefinite theta functions on arbitrary signature lattices starting with Alexandrov, Banerjee, Manschot, and Pioline \cite{ABMP} for lattices of signature $(n-2,2)$. This was further generalized and elaborated in \cite{FK, Kudla, Naz}.~The resulting functions lead to the notion of mock modular forms of higher depth as defined by Zagier and Zwegers in their unpublished work on such theta functions. The vector space $\mathfrak{M}_\k^d$ of weight $\k$ and depth~$d$ mock modular forms, consisting of (vector-valued) holomorphic functions $h$ on $\IH$, and the space of their modular completions $\wh{\mathfrak{M}}_\k^d$ are defined recursively as follows (with $\mathfrak{M}_\k^0 = \wh{\mathfrak{M}}_\k^0 = \mathfrak{M}_\k$): The modular completion $\wh{h} \in \wh{\mathfrak{M}}_\k^d$ of $h \in \mathfrak{M}_\k^d$ is a real-analytic modular form of weight $\k$ such that (for appropriate modular subgroups, multipliers, and growth conditions)
\begin{equation}\label{eq:higher_depth_mock_modular_form_condition}
\frac{\del}{\del \bar{\t}} \wh{h} (\t,\bar{\t})
\in \bigoplus_j v^{r_j} \wh{\mathfrak{M}}^{d-1}_{\k+r_j} \otimes \bar{\mathfrak{M}_{2+r_j}} .
\end{equation}
In particular, the mock modular forms discussed above correspond to $d=1$ by \eqref{eq:mock_modular_form_condition}, whereas the product of two mock modular forms yields (trivially) a depth two mock modular form.

Nontrivial examples of higher depth modular forms can be found via indefinite theta functions. Such objects have since found applications in a wide range of topics in mathematics and physics (e.g.~\cite{ABMP2, AMP, AP, GMN, Man17}). As an example, we consider Vafa--Witten invariants \cite{VW94} for the surface~$\IC\IP^2$ and gauge group $\U(3)$. Manschot \cite{Man17} showed that the generating function of such invariants has the form $\frac{h_{3,\mu}}{\eta^9}$ with $\mu\in\IZ/3\IZ$, where $\eta$ is the Dedekind eta function and\footnote{
In \cite{Man17}, $h_{3,\mu}$ was denoted by $f_{3,\mu}$; see (6.19) there. Also see Appendix A of \cite{Man17} to find how $h_{3,\mu}$ can be expressed in terms of generalized Appell--Lerch functions (which are specific indefinite theta functions).
}
\begin{equation}\label{eq:h_3_0_1_fourier_expansion_vafa_witten}
\mat{h_{3,0} (\t) \\ h_{3,\pm 1} (\t)}
=
\mat{\frac{1}{9} - q + 3 q^2 + 17 q^3 + 41 q^4 + 78 q^5 + 120 q^6 + 193 q^7 + 
240 q^8 + \ldots  
\\ 
 3 q^{\frac53} + 15 q^{\frac83} + 36 q^{\frac{11}3} + 69 q^{\frac{14}3} + 114 q^{\frac{17}3} + 
 165 q^{\frac{20}3}  + \ldots} 
\end{equation}
is a depth two mock modular form.
Its completion is
\begin{equation}\label{eq:depth_two_example}
\wh{h}_{3,\mu} (\t, \bar{\t}) =
h_{3,\mu} (\t) + \frac{3 i \sqrt{3}}{8 \pi \sqrt{2}} 
\sum_{\a \in \IZ/2\IZ} \int_{-\bar{\t}}^{i \infty} 
\frac{\wh{f}_{1,\a} (\t,-w) \vth_{3,2\mu+3\a} (w)}{(-i (w + \t))^{\frac{3}{2}} } dw. 
\end{equation}
It transforms with weight $3$ and the dual of the Weil multiplier for the $A_2$-lattice (see Section~\ref{sec:Weil_multiplier_review})~as
\begin{equation}\label{eq:depth_two_example_modular_transformation}
\wh{h}_{3,\mu} \lp \frac{a\t+b}{c \t+d}, \frac{a\bar{\t}+b}{c\bar{\t}+d} \rp
=
(c \t + d)^{3} \sum_{\nu \in \IZ/3\IZ} 
\Psi^*_{\! A_2} (M)_{\mu,\nu}
\wh{h}_{3,\nu} \lp \t, \bar{\tau} \rp  
\quad \mbox{for } 
M := \mat{a & b \\ c & d} \in \SL_2(\IZ) .
\end{equation}
In particular, the shadow of $h_{3,\mu}$ is 
consistent with \eqref{eq:higher_depth_mock_modular_form_condition} for $d=2$.

By multiplying two copies of the weight $\frac{3}{2}$ Eisenstein series in \eqref{eq:depth_one_eisenstein}, one obtains an object that analytically continues to $s=0$ to yield a trivially depth two mock modular form. The question we address in this paper is whether one can ``couple" these two Eisenstein series to obtain nontrivial depth two mock modular forms such as $h_{3,\mu}$ as well. To this end, for $m_1,m_2 \in \IN$, $\mu_j \in \IZ/2m_j\IZ$ for $j \in \{1,2\}$, $\t \in \IH$, and $s \in \IC$ with $\re (s) > \frac{1}{4}$, we define
\begin{multline}\label{eq:coupled_Eisenstein_definition}
\wh{H}_{\vc{m},\vc{\mu}} (\t, \bar{\t};s)
:=
\frac{isv^{2s}}{2}
\sum_{\substack{(c_1,d_1), (c_2,d_2) \in \LL \\ c_1 d_2 - c_2 d_1 \neq 0}}
\frac{\Psi_{m_1}(M_{c_1,d_1})_{0,\mu_1}}{|c_1\t+d_1|^{2s} \lp c_1\t+d_1 \rp^{\frac{3}{2}}} 
\frac{\Psi_{m_2}(M_{c_2,d_2})_{0,\mu_2}}{|c_2\t+d_2|^{2s} \lp c_2\t+d_2 \rp^{\frac{3}{2}}} 
\\[-2ex]
\times 
\arctan \lp \frac{c_1c_2 |\t|^2 + (c_1 d_2 + c_2 d_1) u + d_1 d_2}{(c_1 d_2 - c_2 d_1)v} \rp .
\end{multline}
Here and throughout we use bold letters to denote vectors and in particular let $\vc{m} := (m_1,m_2)$ and $\vc{\mu} := (\mu_1, \mu_2)$.
We also recall $C_m$ and $\TH_{m, \mu} (\t)$ from \eqref{eq:C_m_and_TH_m_mu_definition} and let (for $m_1,m_2$ odd, square-free)
\begin{equation}\label{eq:depth_two_holomorphic_definition}
H_{\vc{m},\vc{\mu}} (\t) := \sum_{n \in \IZ^{>0}_{\vc{m},\vc{\mu}}} 
\b_{\vc{m},\vc{\mu}} (n) e^{2 \pi i n \t}
\end{equation}
with\footnote{
More generally, we define $\IZ_{m,\mu} := \IZ - \frac{\mu^2}{4m}$ and 
\smash{$\IZ_{\vc{m},\vc{\mu}} := \IZ - \frac{\mu_1^2}{4m_1} - \frac{\mu_2^2}{4m_2}$}, while using superscripts as in $\IZ^{\geq 0}_{\vc{m},\vc{\mu}}$ and $\IZ^{\leq 0}_{\vc{m},\vc{\mu}}$ to restrict to elements that are $\geq 0$ or $\leq 0$.
}
$\IZ^{> 0}_{\vc{m},\vc{\mu}} := \{n \in \IZ - \frac{\mu_1^2}{4m_1} - \frac{\mu_2^2}{4m_2} : \, n>0 \}$
and, for $n \in \IZ^{> 0}_{\vc{m},\vc{\mu}}$,
\begin{align}
\notag 
\b_{\vc{m},\vc{\mu}} (n)
:= 
&
\frac{36 \sqrt{m_1}}{\s (m_1) \s (m_2)} \sum_{d_1 \mid m_1} 
\sum_{\ell \in \IZ + \frac{\xi_{d_1,m_1} \mu_1}{2m_1}} \hspace{-0.5cm}
\frac{\lp \sqrt{n +m_1 \ell^2} - |\ell| \sqrt{m_1} \rp^2}{\sqrt{n +m_1 \ell^2}}
\notag \\[-2.5ex] & \hspace{5.2cm} \times
\sum_{d_2 \mid m_2} d_2 \d_{d_2^2 \mid 4 m_2 (n+m_1\ell^2)}
H\lp \frac{4 m_2 (n+m_1\ell^2)}{d_2^2} \rp 
\notag \\[-1ex]
& -
\frac{36 \sqrt{m_2}}{\s (m_1) \s (m_2)} \sum_{d_2 \mid m_2} 
\sum_{\ell \in \IZ + \frac{\xi_{d_2,m_2} \mu_2}{2m_2}} \hspace{-0.5cm}
\frac{\lp \sqrt{n +m_2 \ell^2} - |\ell| \sqrt{m_2} \rp^2}{\sqrt{n +m_2 \ell^2}}
\notag \\[-2.5ex] & \hspace{5.2cm} \times
\sum_{d_1 \mid m_1} d_1 \d_{d_1^2 \mid 4 m_1 (n+m_2\ell^2)}
H\lp \frac{4 m_1 (n+m_2\ell^2)}{d_1^2} \rp  .
\label{eq:depth_two_holomorphic_fourier_coef_definition}
\end{align}
Here and throughout $\delta_{\mathcal{S}} := 1$ if $\mathcal{S}$ holds and $\delta_{\mathcal{S}} := 0$ otherwise.
We can now state our main result.

\begin{thm}\label{thm:coupled_Eisenstein_depth_two_analytic_continuation}
Let $m_1,m_2 \!\in \IN$ be odd, square-free, $\mu_j \in \IZ/2m_j\IZ$ for $j \in\! \{1,2\}$, $\t \in \IH$, and~${s \in \IC}$.
Then $s \! \mapsto \! \wh{H}_{\vc{m},\vc{\mu}} (\t, \bar{\t};s)$ analytically continues to $\re (s) \!> \! - \frac{1}{20}$ and $\wh{H}_{\vc{m},\vc{\mu}} (\t, \bar{\t}) := \wh{H}_{\vc{m},\vc{\mu}} (\t, \bar{\t};0)$ is
\begin{multline*}
\wh{H}_{\vc{m},\vc{\mu}} (\t, \bar{\t}) 
= 
H_{\vc{m},\vc{\mu}} (\t) 
+ C_{m_2} \int_{-\bar{\t}}^{i \infty} 
\frac{\wh{f}_{m_1,\mu_1} (\t,-w) \TH_{m_2, \mu_2} (w)}{(-i(w+\t))^{\frac{3}{2}}} dw
\\[-1ex]
- C_{m_1} \int_{-\bar{\t}}^{i \infty} 
\frac{\wh{f}_{m_2,\mu_2} (\t,-w) \TH_{m_1, \mu_1} (w)}{(-i(w+\t))^{\frac{3}{2}}} dw .
\end{multline*}
We have, for $M = \pmat{a & b \\ c&d} \in \SL_2(\IZ)$,
\begin{equation*}
\wh{H}_{\vc{m},\vc{\mu}} \lp \frac{a\t+b}{c \t+d}, \frac{a\bar{\t}+b}{c\bar{\t}+d} \rp
= (c \t + d)^3 
\sum_{\substack{\nu_1 \in \IZ/2m_1\IZ \\ \nu_2 \in \IZ/2m_2\IZ}}
\Psi^*_{m_1}(M)_{\mu_1,\nu_1}  \Psi^*_{m_2}(M)_{\mu_2,\nu_2} 
\wh{H}_{\vc{m},\vc{\nu}} (\t, \bar{\t}) .
\end{equation*}
In particular, $H_{\vc{m},\vc{\mu}} (\t)$ is a depth two mock modular form with the shadow
\begin{equation*}
(2 v)^{\frac{3}{2}} \frac{\del}{\del \bar{\t}}
\wh{H}_{\vc{m},\vc{\mu}} (\t, \bar{\t}) 
=
C_{m_2} \wh{f}_{m_1,\mu_1} (\t,\bar{\t}) \TH_{m_2, \mu_2} (-\bar{\t})
- C_{m_1} \wh{f}_{m_2,\mu_2} (\t,\bar{\t}) \TH_{m_1, \mu_1} (-\bar{\t}) .
\end{equation*}
\end{thm}

We can use this result to find the depth two mock modular form $h_{3,\mu}$ as well with such coupled Eisenstein series. To state this, for $\mu \in \IZ/3\IZ$, $\t \in \IH$, and $s \in \IC$ with $\re (s) > \frac{1}{4}$, we define
\begin{multline}\label{eq:h_3_mu_coupled_eisenstein_expression}
\wh{h}_{3,\mu} (\t, \bar{\t};s)
:=
\frac{v^{2s}}{4}
\sum_{\a \in \IZ/2\IZ} 
\sum_{(c_1,d_1), (c_2,d_2) \in \LL}
\frac{\Psi_{1}(M_{c_1,d_1})_{0,\a}}{|c_1\t+d_1|^{2s} \lp c_1\t+d_1 \rp^{\frac{3}{2}}} 
\frac{\Psi_{3}(M_{c_2,d_2})_{0,2\mu+3\a}}{|c_2\t+d_2|^{2s} \lp c_2\t+d_2 \rp^{\frac{3}{2}}} 
\\[-0.5ex]
\times 
\Bigg(
\frac{is}{4} \arctan \lp \frac{c_1c_2 |\t|^2 + (c_1 d_2 + c_2 d_1) u + d_1 d_2}{(c_1 d_2 - c_2 d_1)v} \rp \d_{c_1 d_2 - c_2 d_1 \neq 0} 
+ \frac{1}{8}
- \frac{1}{72} \d_{c_1 d_2 - c_2 d_1 = 0} 
\Bigg) .
\end{multline}
As we show next, this approach also yields an expression for the Fourier coefficients of $h_{3,\mu}$, independent from the one obtained via Appell functions in \cite{Man17}. Here we recall the theta function $\Theta^{[E_6]}_{\mu}$ for the lattice $E_6$ and its cosets in the dual lattice (labeled by $\mu \in \IZ/3\IZ$) with Fourier coefficients
\begin{equation}\label{eq:E6_theta_function}
\Theta^{[E_6]}_{\mu} (\t) =:
\sum_{n \in \IN_0 + \epsilon_\mu} \!\!\!\! r^{[E_6]}_{\mu} (n)  q^n
\quad \mbox{with }
\mu \in \IZ/3\IZ
\mbox{ and }
\epsilon_\mu := 
\begin{cases}
0 \quad &\mbox{if } \mu = 0, \\
\frac{2}{3} \quad &\mbox{if } \mu = \pm 1,
\end{cases}
\end{equation}
given by 
\begin{equation*}
\mat{\Theta^{[E_6]}_{0} (\t) \\\\[-11pt] \Theta^{[E_6]}_{\pm 1} (\t)}
=
\mat{1 + 72 q + 270 q^2 + 720 q^3 + 936 q^4 + 2160 q^5 + 2214 q^6 + 
3600 q^7 + 4590 q^8 + \ldots  
\\
 27 q^{\frac23} + 216 q^{\frac53} + 459 q^{\frac83} + 1080 q^{\frac{11}3} + 
 1350 q^{\frac{14}3} + 2592 q^{\frac{17}3} + 2808 q^{\frac{20}3}  + \ldots} .
\end{equation*}

\begin{cor}\label{cor:h_3_mu_coupled_eisenstein_expression}
For $\mu \in \IZ/3\IZ$ and $\t \in \IH$, the function $s \mapsto \wh{h}_{3,\mu} (\t, \bar{\t};s)$ analytically continues to $\re (s) > -\frac{1}{20}$. We have $\wh{h}_{3,\mu} (\t, \bar{\t};0) = \wh{h}_{3,\mu} (\t, \bar{\t})$, where $\wh{h}_{3,\mu} (\t, \bar{\t})$ is the modular completion of the depth two mock modular form $h_{3,\mu} (\t)$ noted in \eqref{eq:depth_two_example}.
In particular, for 
\begin{equation}\label{eq:h_3_mu_fourier_expansion}
h_{3,\mu} (\t) =: \sum_{n \in \IN_0 + \epsilon_\mu} \!\!\!\! c_\mu (n)  q^n,
\end{equation}
this expression in terms of the coupled Eisenstein series implies the identity, for $n \in \IN_0 + \epsilon_\mu$,
\begin{align}
c_\mu (n)
&= 
\frac{9}{16} \sum_{\ell \in \IZ} 
\frac{\lp \sqrt{4n+\ell^2} - |\ell| \rp^2}{\sqrt{4n+\ell^2}}
\lp H \lp 12 n + 3\ell^2 \rp 
+ 3 \d_{9 \mid  12n + 3\ell^2} H \lp \frac{4 n+\ell^2}{3} \rp \rp
\notag\\[-0.8ex] & \qquad
- \frac{9 \sqrt{3}}{8} 
\sum_{\ell \in \IZ + \frac{2\mu}{3}}
\frac{\lp \sqrt{4n+3\ell^2} - \sqrt{3} |\ell| \rp^2}{\sqrt{4n+3\ell^2}}
H\lp 4n+3\ell^2 \rp 
\notag\\[-0.8ex] & \qquad
+ \frac{9}{2} \sum_{\substack{k \in \IN_0 + \left\{0,\frac{3}{4} \right\} \\ k \leq n}} 
\!\!\!\!\!\!
H(4k) \lp H(12(n-k)) + 3 \d_{9 \mid 12 (n-k)} H \lp \frac{4(n-k)}{3} \rp \rp 
- \frac{1}{72} r^{[E_6]}_{\mu} (n) .
\label{eq:c_mu_n_definition_from_corollary}
\end{align}
\end{cor}

The paper is organized as follows. 
In Section~\ref{sec:preliminaries}, we review preliminaries on analytic tools we use and multiplier systems that appear in our paper. 
In Section~\ref{sec:weight_3_2_eisenstein}, we give an overview of Zagier's weight $\frac{3}{2}$ Eisenstein series to set up notation and certain bounds we employ in later chapters. 
In Section~\ref{sec:non_hol_Eichler_continuation}, we apply these methods to a function yielding non-holomorphic Eichler integrals of modular forms at $s=0$.
In Section~\ref{sec:first_properties_coupled_eisenstein}, we start our investigation of the coupled Eisenstein series defined in \eqref{eq:coupled_Eisenstein_definition} and describe some of its basic properties.
Section~\ref{sec:analytic_atirhmetic_parts_coupled_eisenstein} introduces and examines the respective analytic and arithmetic building blocks appearing in the Fourier expansion of such coupled Eisenstein series.
This is used in Section~\ref{sec:analytic_continuation_coupled_eisenstein} to prove Theorem~\ref{thm:coupled_Eisenstein_depth_two_analytic_continuation} and, in particular, to establish the analytic continuation of the series we introduce.
Finally, we apply this machinery in Section~\ref{sec:vafa_witten_example} to demonstrate its use in the above-mentioned generating function of rank two Vafa--Witten invariants and prove Corollary~\ref{cor:h_3_mu_coupled_eisenstein_expression}. We conclude with two appendices, providing details on the local factors appearing in the arithmetic parts and giving numerical checks, respectively.

\section*{Acknowledgments}
The authors have received funding from the European Research Council (ERC) under the European Union’s Horizon 2020 research and innovation programme (grant agreement No. 101001179).
The second author would like to thank the Isaac Newton Institute for Mathematical Sciences, Cambridge, for support and hospitality during the programme ``Black holes: bridges between number theory and holographic quantum information" where work on this paper was undertaken. This work was supported by EPSRC grant no EP/R014604/1.

\section*{Notation List}
Here we list some of the notation used throughout the paper.

\begin{itemize}
\item $\a_{m,\mu} (n)$: the Fourier coefficients of the mock modular form $f_ {m,\mu} (\t)$,  see Proposition \ref{prop:depth_one_eisenstein_analytic_continuation}.

\item $\b_{\vc{m},\vc{\mu}} (n)$: the Fourier coefficients of $H_{\vc{m},\vc{\mu}} (\t)$ by \eqref{eq:depth_two_holomorphic_definition}, given in \eqref{eq:depth_two_holomorphic_fourier_coef_definition}.

\item $C_m := \frac{3i}{\pi \s(m)} \sqrt{\frac{m}{2}}$.

\item $c_\mu (n)$: the Fourier coefficients of $h_{3,\mu} (\t)$ by \eqref{eq:h_3_mu_fourier_expansion}.

\item $\chi_D (a) := ( \frac{D}{a} )$, where $( \frac{\cdot}{\cdot} )$ is the Kronecker symbol.

\item $D \in \IZ$: denotes a fundamental discriminant.

\item $\epsilon_\mu := 0$ if $\mu = 0$ and
$\epsilon_\mu := \frac{2}{3}$ if $\mu = \pm 1$, see \eqref{eq:E6_theta_function}.

\item  $\wh{f}_{m,\mu} (\t, \bar{\tau}; s)$: the weight $\frac{3}{2}$ Eisenstein series, defined in \eqref{eq:depth_one_eisenstein}. 

\item $\wh{f}_{m,\mu} (\t, \bar{\tau}) := \wh{f}_{m,\mu} (\t, \bar{\tau}; 0)$, see Proposition \ref{prop:depth_one_eisenstein_analytic_continuation}.

\item $f_ {m,\mu} (\t)$: the mock modular form, whose completion is $\wh{f}_{m,\mu} (\t, \bar{\tau})$, see Proposition \ref{prop:depth_one_eisenstein_analytic_continuation}.

\item $\calF_{\vc{m},\vc{\mu}} (\t, \bar{\t};s)$: an ingredient of $\wh{H}_{\vc{m},\vc{\mu}} (\t, \bar{\t};s)$, defined in \eqref{eq:dimension_two_part_Eisenstein}.

\item $g_{m,\mu} (\t, \bar{\tau}; s)$: an ingredient of $\wh{H}_{\vc{m},\vc{\mu}} (\t, \bar{\t};s)$, defined in \eqref{eq:depth_one_Eichler_eisenstein_pluswriting}.

\item $\wh{h}_{3,\mu} (\t, \bar{\t}; s)$: the function that is defined in \eqref{eq:h_3_mu_coupled_eisenstein_expression} and continues to $\wh{h}_{3,\mu} (\t, \bar{\t})$, see Corollary \ref{cor:h_3_mu_coupled_eisenstein_expression}.

\item $\wh{h}_{3,\mu} (\t, \bar{\t})$: the modular completion of $h_{3,\mu} (\t)$, given in \eqref{eq:depth_two_example_modular_transformation}.

\item $h_{3,\mu} (\t)$: the depth two mock modular form given in \cite{Man17}, see around \eqref{eq:h_3_0_1_fourier_expansion_vafa_witten}.

\item $\wh{H}_{\vc{m},\vc{\mu}} (\t, \bar{\t};s)$: the coupled Eisenstein series defined in \eqref{eq:coupled_Eisenstein_definition}.

\item $\wh{H}_{\vc{m},\vc{\mu}} (\t, \bar{\t}) := \wh{H}_{\vc{m},\vc{\mu}} (\t, \bar{\t};0)$, see Theorem \ref{thm:coupled_Eisenstein_depth_two_analytic_continuation}.

\item $H_{\vc{m},\vc{\mu}} (\t)$: the holomorphic part of $\wh{H}_{\vc{m},\vc{\mu}} (\t, \bar{\t})$, given in \eqref{eq:depth_two_holomorphic_definition}.

\item $\calI_\k(v,t;s)$: the ``analytic part" of the Fourier expansion of $\wh{f}_{m,\mu} (\t, \bar{\tau}; s)$, defined in \eqref{eq:Ik_definition}. 

\item $\calJ_\k(v,t;s)$: the analytic part of the Fourier expansion of $g_{m,\mu} (\t, \bar{\tau}; s)$, defined in \eqref{eq:Jk_definition}.

\item $\calK_{\vc{\k}} (v,t,\w;s)$: the analytic part of the Fourier expansion of $\calF_{\vc{m},\vc{\mu}} (\t, \bar{\t};s)$, defined in \eqref{eq:Kk_definition}.

\item $[k] := \{1,2,\ldots,k\}$.

\item $\LL := \{ (c,d) \in \IZ^2 \setminus \{\vc{0}\} : \ \gcd (c,d) = 1 \}$.

\item $\calL_{m,\mu} (n;s)$: the arithmetic parts of the Fourier expansion, defined in \eqref{eq:arithemtic_part_defn}.

\item $\Lreg{m,\mu} (n;s)$: a component of the decomposition \eqref{eq:Lreg_Lsing_decomp_definition} for the pole at $s=0$, defined in \eqref{eq:Lreg_definition}.

\item $\Lsing{m,\mu} (n;s)$: a component of the decomposition \eqref{eq:Lreg_Lsing_decomp_definition} for the pole at $s=0$, defined in \eqref{eq:Lsing_definition}.

\item $M_{c,d}$: an element of $\SL_2(\IZ)$ with lower row $(c,d)$.

\item $\vc{m} := (m_1,m_2)$.

\item $\vc{\mu} := (\mu_1,\mu_2)$.

\item $N_{m,\mu} (n, k) := | \{\nu \pmod{k} : \ 
m \nu^2 - \mu \nu + n + \frac{\mu^2}{4m}  \equiv 0 \pmod{k} \} |$.

\item $\Psi_m$: Weil multiplier for the lattice $L = \IZ$ with the symmetric bilinear form $B_m (x,y) = 2 m x y$.

\item $\Psi_{\! A_2}$: Weil multiplier for the $A_2$-lattice.

\item $r^{[E_6]}_{\mu} (n)$: the Fourier coefficients of the the theta function $\Theta^{[E_6]}_{\mu} (\t)$, given in \eqref{eq:E6_theta_function}.

\item $\square$: shorthand for the phrase ``a square (of an integer)", see the proof of Lemma \ref{lem:arithmetic_part_at_origin}.

\item $\s_{D,s} (f)$: arithmetic function defined in \eqref{eq:sigma_chiD_definition}.

\item $\vth_{m, \mu} (\t)$: Jacobi theta function, defined in \eqref{eq:vth_definition}.

\item $\TH_{m, \mu} (\t)$: combination of Jacobi theta functions defined in \eqref{eq:C_m_and_TH_m_mu_definition} for $m$ odd, square-free.

\item  $\Theta^{[E_6]}_{\mu} (\t)$: the theta function of the lattice $E_6$ and its cosets in the dual lattice for $\mu \in \IZ/3\IZ$.

\item $\xi_{d,m}$: factors that realize Atkin--Lehner operators, defined in \eqref{eq:atkin_lehner_definition}.

\item $\IZ_{m,\mu} := \IZ - \frac{\mu^2}{4m}$.

\item $\IZ_{\vc{m},\vc{\mu}} := \IZ - \frac{\mu_1^2}{4m_1} - \frac{\mu_2^2}{4m_2}$.

\item $\IZ^{>0}_{\vc{m},\vc{\mu}}$: the subset of $\IZ_{\vc{m},\vc{\mu}}$ with elements that are $>0$. Objects like \smash{$\IZ^{\geqslant 0}_{\vc{m},\vc{\mu}}$} are similarly defined.
\end{itemize}

\section{Preliminaries}\label{sec:preliminaries}
We first recall some preliminary facts that we use below.

\subsection{Analytic preliminaries}\label{sec:analytic_preliminaries}

We first recall Poisson summation in the generality we require (see e.g.~\cite[Proposition 5.1.29~(ii)]{BN}).

\begin{thm}\label{thm:poisson_summation}
Let $f: \IR \to \IC$ be in $L^1 (\IR)$, have bounded variation, and, for all $a \in \IR$, satisfy $f(a) = \frac{1}{2} \lp \lim_{x \to a^-} f(x) + \lim_{x \to a^+} f(x) \rp$. Then, with $\hat{f} (x) := \int_{\IR} f(y) e^{-2 \pi i x y} dy$,
\begin{equation*}
\sum_{k \in \IZ} f(k + \a) = \psum_{n \in \IZ} \hat{f} (n) e^{2 \pi i \a n}
\quad \mbox{for all } \a \in \IR.
\end{equation*}
Here $\sum^*$ means that we take the limit of symmetric partial sums.\footnote{
More specifically, $\displaystyle\psum_{n \in \IZ-\b} \ldots := \displaystyle\lim_{M \to \infty} \displaystyle\sum_{\substack{n \in \IZ-\b \\ |n| \leq M}} \ldots$ for any $\b \in \IR$.
}
\end{thm}

We also require a version of the Phragm\'en--Lindel\"of principle given by Rademacher~\cite{Rad}.

\begin{thm}[\cite{Rad}, Theorem 2]\label{thm:phragmen_lindelof}
Let $a < b$ and $f$ be a holomorphic function in the open strip $a < \re (s) < b$, extending continuously to its closure. Suppose there are constants $A,B > 0$ with
\begin{equation*}
|f(s)| \leq A  e^{|\im (s)|^B} \quad \mbox{for } a \leq \re (s) \leq b .
\end{equation*}
Also assume that there are real constants $M_a, M_b > 0$, $Q > -a$, and $\a \geq \b$ such that
\begin{equation*}
|f(a+it)| \leq M_a |Q+a+ it|^\a \andd
|f(b+it)| \leq M_b |Q+b+ it|^\b.
\end{equation*}
Then, with $\ell (x) := \frac{b-x}{b-a}$, we have
\begin{equation*}
|f(s)| \leq M_a^{\ell (\re (s))} M_b^{1-\ell (\re (s))}
|Q+s|^{\a \ell (\re (s)) + \b (1-\ell (\re (s)))} 
\quad \mbox{for } a \leq \re (s) \leq b.
\end{equation*}
\end{thm}

\subsection{Weil multiplier systems}\label{sec:Weil_multiplier_review}

Next we review Weil representations (see e.g.~\cite{Bru2, Shin2} for details). We specialize to the one-dimensional lattice $L = \IZ$ with quadratic form $Q_m(x) = m x^2$ for $m \in \IN$. The corresponding symmetric bilinear form is $B_m (x,y) = 2 m x y$.
We use $\mu \in \IZ/2m\IZ$ to denote the element $\frac{\mu}{2m} \in L^*/L$ and write 
$\Psi_m (M)_{\mu, \nu}$ for the $(\mu, \nu)$-th entry of the Weil multiplier matrix  $\Psi_m (M)$, where $M \in \SL_2 (\IZ)$.
It satisfies the following properties:

\begin{enumerate}[label=(\arabic*),ref=(\arabic*), wide, labelindent=0pt]

\item\label{item:weil_T_S_explicit} For $T = \pmat{1 & 1 \\ 0 & 1}$ and $S = \pmat{0 & -1 \\ 1 & 0}$, we have the following, with $\mu, \nu \in \IZ/2m\IZ$,
\begin{equation}\label{eq:weil_T_S_explicit}
\Psi_m (T)_{\mu, \nu} = e^{2\pi i \frac{\mu^2}{4m}} \d_{\mu, \nu}
\andd
\Psi_m (S)_{\mu, \nu} = \frac{e^{-\frac{\pi i}{4}}}{\sqrt{2m}}
e^{- 2 \pi i \frac{\mu \nu}{2m}}.
\end{equation}
Here $\d_{\mu, \nu}:=1$ if $\mu=\nu$ (as elements of $\IZ/2m\IZ$) and zero otherwise.

\item\label{item:weil_projective_representation} For $M_1 = \pmat{a_1 & b_1 \\ c_1 & d_1}$, $M_2 = \pmat{a_2 & b_2 \\ c_2 & d_3}$, and $M_3 = M_1 M_2 =: \pmat{a_3 & b_3 \\ c_3 & d_3}\in\SL_2(\IZ)$, we have\footnote{
In particular, $\Psi_m$ (like any Weil multiplier for an odd-rank lattice) forms a projective representation of $\SL_2 (\IZ)$. Alternatively, one may extend $\Psi_m$ to the metaplectic double cover of $\SL_2 (\IZ)$ and consider ordinary representations.
}
\begin{equation}\label{eq:weil_projective_representation}
\sqrt{c_1 M_2 \t +d_1}
\sqrt{ c_2 \t +d_2 }
\sum_{\g \pmod{2m}} \Psi_m (M_1)_{\mu, \g} \Psi_m (M_2)_{\g, \nu}
=
\sqrt{c_3 \t +d_3} \Psi_m (M_3)_{\mu, \nu},
\end{equation}
where $M\t= \frac{a \t +b}{c \t +d}$ for $M = \pmat{a & b \\ c & d} \in \SL_2 (\IZ)$.

\item\label{item:minus_infinity_weil} We have
\begin{equation}\label{eq:minus_infinity_weil}
\Psi_m (-I)_{\mu,\nu} = -i \d_{\mu,-\nu} .
\end{equation}

\item\label{item:Weil_multiplier_unitarity} The representation $\Psi_m$ is unitary (here and throughout $*$ denotes the complex conjugate)
\begin{equation*}
\sum_{\g \pmod{2m}} \Psi_m^*(M)_{\mu, \g} \Psi_m (M)_{\nu, \g}
= \d_{\mu, \nu} 
\quad \mbox{for any }
M  \in \SL_2 (\IZ) .
\end{equation*}
\end{enumerate}

\begin{rem}\label{rem:psi_zero_mu_lower_row_determining}
Note that \eqref{eq:weil_T_S_explicit} and \eqref{eq:weil_projective_representation}  imply that $\Psi_m (TM)_{\mu, \nu} = e^{2 \pi i \frac{\mu^2}{4m}} \Psi_m (M)_{\mu,\nu}$. Thus we~can specify the modular matrix in the factor $\Psi_m (M_{c,d})_{0,\mu}$ of the Eisenstein series \eqref{eq:depth_one_eisenstein} with its lower row~$(c,d)$. The same applies to the coupled Eisenstein series \eqref{eq:coupled_Eisenstein_definition}. Also note that by \eqref{eq:weil_projective_representation},~\eqref{eq:minus_infinity_weil},
\begin{equation}\label{eq:weil_multiplier_minus_cd_property}
\Psi_m (M_{c,d})_{0, \mu} = \Psi_m (M_{c,d})_{0, -\mu}
\andd
\frac{\Psi_m (M_{c,d})_{0, \mu}}{\lp c \t + d \rp^{\frac{3}{2}}}
= \frac{\Psi_m (M_{-c,-d})_{0, \mu}}{\lp - c \t - d \rp^{\frac{3}{2}}}.
\end{equation}
By \eqref{eq:depth_one_eisenstein}, this implies $\wh{f}_{m,\mu} (\t, \bar{\tau}; s)  = \wh{f}_{m,-\mu} (\t, \bar{\tau}; s)$.
\end{rem}

Below we also encounter $\Psi_{\! A_2} (M)_{\mu,\nu}$, the Weil multiplier for the $A_2$-lattice ($\mu, \nu \in \IZ/3\IZ \cong \! A_2^*/A_2$). This forms an ordinary representation of $\SL_2 (\IZ)$, whose values on generators are
\begin{equation*}
\Psi_{\! A_2} (T)_{\mu,\nu} = e^{2\pi i \frac{\mu^2}{3}} \d_{\mu,\nu}
\andd
\Psi_{\! A_2} (S)_{\mu,\nu} = -\frac{i}{\sqrt{3}} e^{2 \pi i \frac{\mu \nu}{3}}.
\end{equation*}
Note the decomposition of $\Psi_3$ to  $\Psi_{\! A_2}$ and dual of the Weil multiplier for the $A_1$-lattice $\Psi_1^*$ as
\begin{equation}\label{eq:Weil_three_decomposition}
\Psi_{3}(M)_{2 \mu + 3 \a, 2 \nu + 3 \b} 
=
\Psi^*_{1}(M)_{\a,\b} \, \Psi_{\! A_2} (M)_{\mu,\nu}
\where
\a,\b \in \IZ/2\IZ \mbox{ and } \mu,\nu \in \IZ/3\IZ.
\end{equation}

\subsection{Atkin--Lehner operators}\label{sec:atkin_lehner_operators}

We finally review Atkin--Lehner operators introduced by Eichler and Zagier in \cite{EZ} (see \cite{DMZ} for the terminology). Restricting to $m \in \IN$ odd, square-free, note that for each $d \mid m$, there is a unique $\xi_{d,m}$ modulo $2m$ (with $d \mapsto \xi_{d,m}$ one-to-one) such that
\begin{equation}\label{eq:atkin_lehner_definition}
\xi_{d,m} \equiv 1 \pmod{\frac{2m}{d}} 
\andd
\xi_{d,m} \equiv -1 \pmod{2 d} .
\end{equation}
Also note that $\xi_{d,m}^2 \equiv 1 \pmod{4m}$ and in particular $\xi_{d,m}$ are all invertible modulo $2m$. These facts together with \eqref{eq:weil_T_S_explicit} and \eqref{eq:weil_projective_representation} imply that  $\Psi_m (M)_{\mu, \nu} = \Psi_m (M)_{\xi_{d,m} \mu, \xi_{d,m} \nu}$ for $M \in \SL_2 (\IZ)$ and~$d \mid m$.
Consequently, if $(h_\mu)_{\mu \in \IZ/2m\IZ}$ transforms like a vector-valued modular form under $\SL_2 (\IZ)$ with some half-integral weight and multiplier $\Psi_m$ (or $\Psi_m^*$), then so does the vector $(h_{\xi_{d,m} \mu})_{\mu \in \IZ/2m\IZ}$.

For example, the combination \smash{$(\TH_{m, \mu})_{\mu \in \IZ/2m\IZ}$} defined in \eqref{eq:C_m_and_TH_m_mu_definition} for odd, square-free $m$, forms a weight $\frac{1}{2}$ holomorphic modular form with multiplier $\Psi_m$ (just like the functions $\vth_{m, \mu}$). Moreover, it is an eigenfunction for all Atkin--Lehner operators with eigenvalue $1$.
Also, note that because $\xi_{1,m} \equiv 1 \pmod{2m}$ and $\xi_{m,m} \equiv -1 \pmod{2m}$ along with $\vth_{m, \mu} = \vth_{m, -\mu}$, we have 
\begin{equation*}%\label{eq:TH_vth_comparison}
\TH_{1, \mu} = \vth_{1, \mu} 
\andd
\TH_{p, \mu} = 2 \vth_{p, \mu} 
\quad \mbox{for any odd prime } p.
\end{equation*}

\section{Weight~$\frac{3}{2}$ Eisenstein Series}\label{sec:weight_3_2_eisenstein}

We next review Zagier's \cite{Zag} Eisenstein series $\wh{f}_{m,\mu} (\t, \bar{\tau}; s)$, defined in \eqref{eq:depth_one_eisenstein} 
for $m \in \IN$, $\mu \in \IZ/2m\IZ$, $\t \in \IH$, and $s \in \IC$ with $\re(s) \!>\! \frac{1}{4}$.
We sketch the involved details of its continuation~for the reader's convenience and to fix notation for the vector-valued setting (see also~\cite{BK, GZ, HZ}).\footnote{See also the SAGE package WeilRep \cite{Wil3} and details in \cite{Wil}.}

Note that \eqref{eq:depth_one_eisenstein} is absolutely and locally uniformly convergent in $\tau$ and $s$ in its domain of definition. Thus, it yields a holomorphic function in $s$ for $\re(s) > \frac{1}{4}$ and there satisfies 
\begin{equation}\label{eq:hatf_modular_transformation}
\wh{f}_{m,\mu} \lp \frac{a\t+b}{c \t+d}, \frac{a\bar{\t}+b}{c\bar{\t}+d}; s \rp
=
(c \t + d)^{\frac{3}{2}} \sum_{\nu \in \IZ/2m\IZ} 
\Psi^*_{m}(M)_{\mu,\nu} \,
\wh{f}_{m,\nu} \lp \t, \bar{\tau}; s \rp 
\end{equation}
for all $M = \pmat{a & b \\ c & d} \in \SL_2(\IZ)$.
One way to prove the analytic continuation of $\wh{f}_{m,\mu} (\t, \bar{\tau}; s)$ (and discuss the resulting modular object at $s=0$) is through its Fourier expansion. We next recall the properties of the ingredients that appear in this expansion.

\subsection{The analytic part}\label{sec:analytic_part_dim1_1}
One of the ingredients in the Fourier expansion of $\wh{f}_{m,\mu} (\t, \bar{\tau}; s)$ is its \textit{analytic part}
(for $\k \in \frac{1}{2} \IZ$, $s \in \IC$ with $2 \re (s) + \k > 1$, $t \in \IR$, $v \in \IR^+$):
\begin{equation}\label{eq:Ik_definition}
\calI_\k(v,t;s) := v^{1-\k-s} e^{2 \pi t v}
\int_{-\infty}^\infty \frac{e^{-2\pi i t v x}}{(1-ix)^\k \lp 1+x^2\rp^s} dx.
\end{equation}
We let $\calI(v,t;s) := \calI_{\frac{3}{2}}(v,t;s)$.
This yields a continuous function of $(v,t,s)$ that is holomorphic in $s$ in its region of definition.
We now record the properties of $\calI_\k(v,t;s)$.\footnote{It is known that $\calI_\k(v,t;s)$ can be expressed in terms of Whittaker functions (see e.g.~\cite{Maass}~ by Maass). The properties of $\calI_\k(v,t;s)$ along with its continuation to the complex plane can then be studied from this angle. For our goals, the integral representation in \eqref{eq:Ik_definition} and its generalizations below suffice, so we restrict our discussion~accordingly.}

\begin{lem}\label{lem:analytic_part_calI_properties}
Let $0 < \e < \frac{1}{8}$ be fixed, $\k \in \frac{1}{2} \IN$ with $\k \geq \frac{3}{2}$, $t\in\IR$, and $v\in\IR^+$.
\begin{enumerate}[leftmargin=*]
\item[\rm(1)] For $s \in \IC$ with $\frac{1-\k}{2} + \e \leq \re (s) \leq \frac{1}{\e}$ and $|\im (s)| \leq \frac{1}{\e}$, we have
\begin{equation*}
\calI_\k (v,t;s)  v^{s+\k -1} e^{-2\pi t v} \ll_{\k, \e} e^{- \pi |t| v} .
\end{equation*}

\item[\rm(2)] We have
\begin{equation*}
\calI_\k (v,t;0) =
\begin{cases}
\frac{(2\pi)^\k}{\GG (\k)} t^{\k - 1} \quad &\mbox{if } t>0, \\
0 \quad &\mbox{if } t \leq 0.
\end{cases}
\end{equation*}

\item[\rm(3)] For $t \leq 0$ and $s \in \IC$ with $\frac{1-\k}{2} + \e \leq \re (s) \leq \frac{1}{\e}$ and $|\im (s)| \leq \frac{1}{\e}$ we have\footnote{Here and in similar settings, we denote the derivative in $s$ as $\calI'_\k (v,t;s) := \frac{\del}{\del s} \calI_\k (v,t;s)$.}
\begin{equation*}
\frac{1}{s} \calI_\k (v,t;s)  v^{s+\k -1} e^{-2\pi t v} \ll_{\k, \e} e^{- \pi |t| v} 
\ \ \mbox{and} \ \ 
\calI'_\k (v,t;0) = 2 \pi v^{1-\k} 
\int_2^\infty \frac{e^{2 \pi t v x}}{x^\k} dx.
\end{equation*}
\end{enumerate}
\end{lem}
\begin{proof}
Part (1) follows by shifting the integration path by $-\frac{i}{2}$ if $t > 0$ and by $\frac{i}{2}$ if $t \leq 0$ (away from the branch points at $\pm i$).
For (2), note that the integrand in \eqref{eq:Ik_definition} has only one branch point at $-i$ for $s=0$. So for $t \leq 0$, we deform the path upwards to find $\calI_\k(v,t;0) = 0$. For~$t >0$, we similarly deform the path downwards to surround the branch cut from $-i$ to $- i \infty$ in clockwise orientation. The result then follows from the fact that \smash{$\frac{2 \pi i}{\Gamma (\k)} = \int_{\mathcal{C}} e^z z^{-\k} dz$}, where $\mathcal{C}$ is the Hankel-type contour starting at $-\infty$ below the real line, circling around the origin, and then going back to $- \infty$ above the real line (see e.g.~\cite[(5.9.2)]{NIST}).

Finally, for (3), away from $s=0$, say $|s| > \frac{1}{8}$, the claimed bound follows from~(1). So assume $|s| \leq \frac{1}{8}$ from now on, which is contained in the region assumed by the lemma.
As $t \leq 0$, we deform the integration path in \eqref{eq:Ik_definition} upwards to surround the branch-cut from $i$ to $i \infty$.

\begin{figure}[h!]
\vspace{-10pt}
\centering
\includegraphics[scale=0.33]{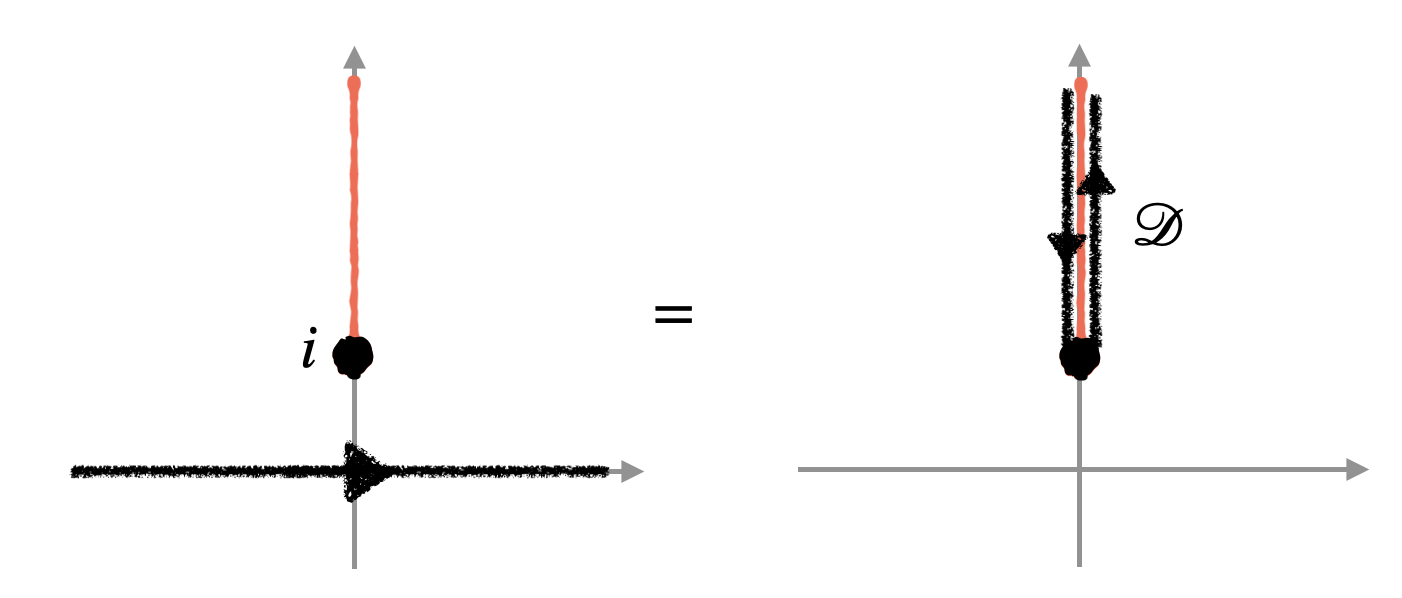}
\vspace{-20pt}
\caption{}
\label{fig:main_int_path_D}
\vspace{-10pt}
\end{figure}

\noindent Moreover, because $|s| \leq \frac{1}{8}$, the growth of the integrand around $z=i$ is at most like $|1+iz|^{-\frac{1}{8}}$ and the integrals can be taken right on the ray from $i$ to $i \infty$.
This yields (see Figure~\ref{fig:main_int_path_D} for $\CD$)
\begin{equation}\label{eq:Ik_expression_D_integration_path}
\calI_\k(v,t;s) 
\! = \!
v^{1-\k-s} e^{2 \pi t v}\!\!\!
\int_{\CD}\!\frac{e^{-2\pi i t v w}}{(1-iw)^{\k+s} (1+iw)^s} dw 
\! = \!
2 \sin (\pi s) v^{1-\k-s} e^{4 \pi t v} \!\!\int_0^\infty \!\! \frac{e^{2\pi t v y}}{(2+y)^{\k+s} y^s}  d y .
\end{equation}
Both the claimed bound on $\frac{1}{s} \calI_\k (v,t;s)$ and the value of $\calI'_\k (v,t;0)$ follows from this expression.
\end{proof}

\subsection{The arithmetic part}
The second ingredient in the Fourier expansion of $\wh{f}_{m,\mu} (\t, \bar{\tau}; s)$ is its \textit{arithmetic part}, defined as the Dirichlet series
\begin{equation}\label{eq:arithemtic_part_defn}
\calL_{m,\mu} (n;s):= 
e^{\frac{\pi i}{4}}
\sum_{c\geq 1} c^{-2s-\frac{3}{2}} \sum_{d \pmod{c}^*}  
\Psi_{m} (M_{c,d})_{0,\mu} e^{2 \pi i n \frac{d}{c}} 
\ \ \  \mbox{for }
m \in \IN, \, \mu \in \IZ/2m\IZ, \, n \in \IZ_{m,\mu} ,
\end{equation}
which converges absolutely in $(c,d)$ for $\re (s) > \frac{1}{4}$ and is holomorphic there.
Here the sum over $d \pmod{c}^*$ is over residue classes modulo $c$ with $\gcd (c,d) = 1$.
In this section and the associated Appendix~\ref{sec:solving_quadratic_equations}, we review its properties (giving details on the results of \cite{BK, EZ, HZ, Maass2, Zag2} adapted to our setting).
We first note that a standard calculation (see e.g.~\cite{BK}) yields 
\begin{equation}\label{eq:arithmetic_part_rewriting_with_qaudratic_equation}
\calL_{m,\mu} (n;s)= 
\frac{1}{\sqrt{2m}}
\sum_{c\geq 1} c^{-2s-1} \sum_{k \mid c} 
\mu \lp \frac{c}{k} \rp N_{m,\mu} (n, k) 
\quad 
\mbox{for } \re (s) > \frac{1}{4},
\end{equation}
where $\mu(\cdot)$ denotes the M\"obius function and
\begin{equation}\label{eq:number_congruence_solutions}
N_{m,\mu} (n, k) := \left| \left\{\nu \pmod{k} : \ 
m \nu^2 - \mu \nu + n + \frac{\mu^2}{4m}  \equiv 0 \pmod{k} \right\} \right|.
\end{equation}
One can evaluate $N_{m,\mu} (n, k)$ with elementary tools. We give the result for $m \in \IN$ with our notation in Appendix~\ref{sec:solving_quadratic_equations}.
From now on, we restrict to $m$ odd, square-free for technical simplicity. We also~let
\begin{equation}\label{eq:sigma_chiD_definition}
\s_{D,s} (f) := 
\sum_{\substack{a,c > 0 \\ ac \mid f}} \frac{\mu (a) \chi_D (a) }{a^{s} c^{2s-1}} ,
\end{equation}
with $D \!\in\! \IZ$ a fundamental discriminant, $\chi_D (a) := ( \frac{D}{a} )$, $f \in \IN$, and $s \in \IC$.
This defines a multiplicative function of $f$.
We can use this to evaluate the Bell series associated with $N_{m,\mu} (n,\cdot)$.

\begin{lem}\label{lem:bell_series_quadratic}
Let $m \in \IN$ be odd, square-free, $\mu \in \IZ/2m\IZ$, $n \in \IZ_{m,\mu}$ with $n \neq 0$, $p$ prime, and~$s \in \IC$ with $\re (s) > 0$. Writing $-4mn = D f^2$, where $D \in \IZ$ is a fundamental discriminant and~$f \in \IN$ satisfying $p^\l \parallel f$ with $\l \in \IN_0$, we have
\begin{equation*}
\sum_{r\geq 0} \frac{N_{m,\mu} (n,p^r)}{p^{rs}}
= 
\frac{1}{1-\chi_D (p) p^{-s}} 
\begin{cases}
\lp 1+p^{-s} \rp \s_{D,s} \lp p^\l \rp
\quad &\mbox{if } p \nmid m, \\[1ex]
\s_{D,s} \lp p^\l \rp +\d_{\l \geq 1}  p^{-(s-1)} \s_{D,s} \lp p^{\l-1} \rp
\quad &\mbox{if } p \mid m .
\end{cases}
\end{equation*}
\end{lem}
\begin{proof}
Write $-4mn = p^\ell n_0$ with $\ell \in \IN_0$, $p \nmid n_0 \in \IZ$ and $f = p^\l f_0$ with $p \nmid f_0 \in \IN$. 
We first assume $p \nmid m$, $p$ odd.
If $\ell$ is even, $\ell = 2 \l$ and $p \nmid D$ so $n_0 = D f_0^2$ and~\smash{$( \frac{n_0}{p} ) = ( \frac{D}{p} )$}.
By Lemma~\ref{lem:quadratic_solution_count3} for $\a=0$,
\begin{equation}\label{eq:Np_general_even_case}
N_{m,\mu} (n, p^r) 
= 
\begin{cases}
p^{\left\lfloor \frac{r}{2} \right\rfloor} 
\quad & \mbox{if } r \leq 2 \l, \\
\lp 1 +\chi_D (p) \rp p^\l 
\quad & \mbox{if } r > 2 \l.
\end{cases}
\end{equation}
The claimed expression for the Bell series follows for $\re (s) > 0$ noting that $\chi_D (p) \in \{ \pm 1 \}$.
If $\ell$ is odd, on the other hand, then $p \parallel D$, $\ell = 2 \l +1$, and Lemma~\ref{lem:quadratic_solution_count3} (for $\a=0$) implies
\begin{equation}\label{eq:Np_general_odd_case}
N_{m,\mu} (n, p^r) 
= 
\begin{cases}
p^{\left\lfloor \frac{r}{2} \right\rfloor} 
\quad & \mbox{if } r \leq 2 \l + 1, \\
0
\quad & \mbox{if } r > 2 \l + 1.
\end{cases}
\end{equation}
The claim follows once we note that $\chi_D (p) = 0$ in this case.

Next we examine $p=2$ and first assume $2 \mid \mu$, as in Lemma \ref{lem:quadratic_solution_count2}~(1) ($\a=0$). Then $\ell \geq 2$ and we distinguish three cases: (1) $\ell$ even, $n_0 \equiv 1 \pmod{4}$, (2) $\ell$ even, $n_0 \equiv 3 \pmod{4}$, (3) $\ell$ odd. 
In case (1), where $\chi_D(2) \in \{ \pm 1 \}$, \eqref{eq:Np_general_even_case} holds by Lemma \ref{lem:quadratic_solution_count2}~(1).
Similarly, in cases (2) and (3), where $\chi_D (2) = 0$, \eqref{eq:Np_general_odd_case} holds. So $p=2$ and $2 \mid \mu$ gives the same result as the other cases of $p \nmid m$ above. 
The case $2 \nmid \mu$ follows from Lemma \ref{lem:quadratic_solution_count2}~(2), noting $2\d_{\smash{n+\frac{\mu^2}{4m} \equiv 0 \pmod{2}}} = 1 + \chi_D (2)$ and $\l = 0$.

Finally, a similar computation as above with Lemma~\ref{lem:quadratic_solution_count3} and $\a=1$ yields the case $p \mid m$.
\end{proof}

Combining Lemma~\ref{lem:bell_series_quadratic} with \eqref{eq:arithmetic_part_rewriting_with_qaudratic_equation} and using Lemma~\ref{lem:quadratic_solution_count_n_zero} for $n=0$, we express the arithmetic part $\calL_{m,\mu} (n;s)$ in terms of Dirichlet L-functions.

\begin{lem}\label{lem:arithmetic_part_dirichlet_writing}
Let $m \in \IN$ be odd, square-free, $\mu \in \IZ/2m\IZ$, $n \in \IZ_{m,\mu}$, and $s \in \IC$ with $\re (s) > \frac{1}{4}$.
\begin{enumerate}[leftmargin=*]
\item[\rm(1)] If $n \neq 0$, writing $-4mn = D f^2$ with $D \in \IZ$ a fundamental discriminant and $f \in \IN$, we have
\begin{equation*}
\calL_{m,\mu} (n;s)
=
\frac{L(2s+1,\chi_D)}{\sqrt{2m} \, \s_{-2s-1} (m) \z (4s+2)}
\sum_{d \mid m, f} d^{-2s} \s_{D,2s+1} \lp \frac{f}{d} \rp .
\end{equation*}

\item[\rm(2)] If $n=0$, and hence $2m \mid \mu$, then we have\footnote{
The result here for $n=0$ has the same form as the one for $n \neq 0$, if we decompose $-4mn = Df^2$ with $D=1$ and $f=0$ for $n=0$. The expression for $\s_{1,2s+1}$ in \eqref{eq:sigma_chiD_definition} evaluates as $\s_{1,2s+1} (0) = \frac{\z (4s+1)}{\z (2s+1)}$ for $f=0$.
}
\begin{equation*}
\calL_{m,0} (0;s)
=
\frac{ \s_{-2s} (m)\z(4s+1) }{\sqrt{2m} \, \s_{-2s-1} (m) \z (4s+2) }.
\end{equation*}
\end{enumerate}
\end{lem}

These expressions yield meromorphic continuation of $\calL_{m,\mu} (n;s)$ to the whole complex $s$-plane.\footnote{
Note that \eqref{eq:arithmetic_part_rewriting_with_qaudratic_equation} is absolutely convergent in $(c,k)$ for $\re (s) > 0$ and already extends the original definition.
}
In particular, $\calL_{m,\mu} (n;s)$ has a pole at $s=0$ if $-4mn$ is a square. In preparation for the analytic continuation here and below, we make this more explicit. For $m \in \IN$ odd, square-free, $\mu \in \IZ/2m\IZ$, $n \in \IZ_{m,\mu}$, and $s \in \IC$, we decompose
\begin{equation}\label{eq:Lreg_Lsing_decomp_definition}
\calL_{m, \mu} (n;s) = \frac{1}{\sqrt{2m} \, \s_{-2s-1} (m) \z (4s+2) }
\lp \Lreg{m,\mu} (n;s) + \frac{1}{4s} \Lsing{m,\mu} (n;s) \rp,
\end{equation}
where we define (with $D \in \IZ$ a fundamental discriminant and $f \in \IN$)
\begin{align}
\Lreg{m,\mu} (n;s) &:= 
\begin{cases}
L(2s+1, \chi_D) \sum_{d \mid m,f} d^{-2s} \s_{D,2s+1} \lp \frac{f}{d} \rp
\quad &\mbox{if } -4mn = Df^2 \mbox{ with } D \neq 1,
\\[.45em]
\lp \z (2s+1) - \frac{1}{2s} \rp 
\sum_{d \mid m,f} d^{-2s} \s_{1,2s+1} \lp \frac{f}{d} \rp
\quad &\mbox{if } -4mn = f^2,
\\[.45em]
\lp \z (4s+1) - \frac{1}{4s} \rp \s_{-2s} (m)
\quad &\mbox{if } n = 0,
\end{cases}
\label{eq:Lreg_definition}
\\
\Lsing{m,\mu} (n;s) &:= 
\begin{cases}
2 \sum_{d \mid m,f} d^{-2s} \s_{1,2s+1} \lp \frac{f}{d} \rp
\quad &\mbox{if } -4mn = f^2,
\\
\s_{-2s} (m) \quad &\mbox{if } n = 0,
\\
0 \quad &\mbox{otherwise}.
\end{cases}
\label{eq:Lsing_definition}
\end{align}
Both $\Lreg{m,\mu} (n;s)$ and $\Lsing{m,\mu} (n;s)$ are entire functions of $s$.
We next give their values~at~$s=0$.

\begin{lem}\label{lem:arithmetic_part_at_origin}
Let $m \in \IN$ be odd, square-free and $\mu \in \IZ/2m\IZ$.
Then we have
\begin{align*}
\Lreg{m,\mu} (n;0) 
&= \frac{\pi}{\sqrt{4mn}} \sum_{\substack{d\mid m}} d\, \d_{d^2 \mid 4mn} H \lp \frac{4mn}{d^2} \rp
\quad \mbox{for } n \in \IZ_{m,\mu}^{>0},
\\
\Lsing{m,\mu} (n;0) 
&= \sum_{d \mid m} \sum_{k \in \IZ + \frac{\xi_{d,m} \mu}{2m}} \d_{-mk^2 = n} 
\quad \mbox{for } n \in \IZ_{m,\mu} .
\end{align*}
\end{lem}
\begin{proof}
The result for $\Lreg{m,\mu}$ follows from \cite[p.~95]{HZ},
the one on $\Lsing{m,\mu}$ from the fact that $\s_{1,1} (f) = 1$ for $f \in \IN$ and the identity (for $n \in \IZ_{m,\mu}$ and with $-4mn = \square$ standing for ``$-4mn$ is a square") 
\begin{equation*}
\sum_{d \mid m} \sum_{k \in \IZ + \frac{\xi_{d,m} \mu}{2m}} \d_{-mk^2=n}
=
\begin{cases}
2 d(\gcd (m,\mu)) \quad &\mbox{if } n \neq 0 \mbox{ and } -4mn = \square, \\
d(m) \quad &\mbox{if } n=0, \\
0 \quad &\mbox{otherwise}.
\end{cases}
\end{equation*}
To establish this in the first case, note that $-4mn = f^2$ for $f \in \IN$ implies $f \equiv \xi_{\d,m} \mu \pmod{2m}$ for some $\d \mid m$. Using that the $\xi_{d,m}$ form a group under multiplication, one can rewrite the left-hand side above as \smash{$2 \sum_{d \mid m} \d_{\mu \equiv \xi_{d,m} \mu \pmod{2m}}$}. The identity in this case then follows by noting that $\mu \equiv \xi_{d,m} \mu \pmod{2m}$ if and only if $\mu \equiv 0 \pmod{d}$.
\end{proof}

To bound \smash{$\Lreg{m,\mu}$} and \smash{$\Lsing{m,\mu}$}, we require estimates for the Riemann zeta function and Dirichlet L-functions appearing in $\mathcal{L}^{\mathrm{reg}}$.
Here we can employ the convexity bound (for $\e > 0$),
\begin{equation}\label{eq:zeta_function_convexity_bound}
\z (s) - \frac{1}{s-1} 
\ll_\e (1+|\im (s)| )^{\max \left\{ \frac{1-\re (s)}{2}, 0 \right\} + \e}
\quad\mbox{for }\re (s) \geq 0 .
\end{equation}
For the weight $\frac{3}{2}$ Eisenstein series, we can also use the convexity bound for $L(s, \chi_D)$, obtained by replacing $(1+|\im (s)|)$ with the analytic conductor $(1+|\im (s)|) |D|$.
However, our arguments for the coupled Eisenstein series require moving to the subconvexity range in the modulus-aspect (estimates in the $t$-aspect play little role and when convenient we relax the exponents on $1+|\im (s)|$).
We use the subconvex bound with the Weyl-exponent, recently proved in full generality by Petrow and Young \cite{PY1, PY2} giving, for $\e >0$, $D \in \IZ$ a fundamental discriminant, and $t \in \IR$,\footnote{
A weaker subconvexity bound with a larger exponent on the modulus, such as the Burgess bound \cite{Burgess, HeathBrown}, can also be used with our arguments to continue the coupled Eisenstein series to a (smaller) half-plane including $s=0$.
}
\begin{equation}\label{eq:Petrow_Young_bound}
L \lp \frac{1}{2} + i t, \chi_D \! \rp \ll_\e \lp (1+|t|) |D| \rp^{\frac{1}{6} + \e}.
\end{equation}

\begin{lem}\label{lem:Lreg_Lsing_bound}
Let $\e > 0$ and let $m \in \IN$ be odd, square-free, $\mu \!\in\! \IZ/2m\IZ$, $n \in \IZ_{m,\mu}$, and $s \in \IC$ with \smash{$\re (s) \geq -\frac{1}{4}$}. Then we have
\begin{equation*}
\Lreg{m,\mu} (n;s) \ll_{m, \e} \!
\lp 1+ |\im (s)| \rp \lp 1+|n|^{\max \left\{ -\frac{2\re (s)}{3}, 0\right\} + \e} \rp 
\, \ \mbox{and} \ \,
\Lsing{m,\mu} (n;s) \ll_{m, \e} \! \lp 1+|n|^\e \rp \d_{-4mn = \square}. 
\end{equation*}
\end{lem}
\begin{proof}
This follows from the bound 
$|\s_{D,2s+1} (f)| \leq d(f^2)$ for $\re (s) \geq -\frac{1}{4}$ and using convexity arguments (Theorem~\ref{thm:phragmen_lindelof}) with \eqref{eq:zeta_function_convexity_bound} and \eqref{eq:Petrow_Young_bound} (while relaxing the exponent on $1+|\im(s)|$).
\end{proof}

\subsection{The analytic continuation}
We now review the analytic continuation of the weight \smash{$\frac{3}{2}$} Eisenstein series.
Equipped with the discussions of \eqref{eq:Ik_definition} and \eqref{eq:arithemtic_part_defn}, a standard argument using Poisson summation yields the Fourier expansion of 
$\wh{f}_{m,\mu} (\t, \bar{\tau}; s)$ as (for $\re (s) > \frac{1}{4}$)
\begin{equation}\label{eq:depth_one_eisenstein_fourier_expansion}
\wh{f}_{m,\mu} (\t, \bar{\tau}; s) 
=  \d_{0,\mu} v^s
- \sum_{n \in \IZ_{m, \mu}} \calL_{m,\mu} \lp n;s \rp \calI \lp v,n;s \rp e^{2 \pi i n \t}.
\end{equation}
This expansion then yields the analytic continuation\footnote{One can continue meromorphically to the whole $s$-plane, but we limit our discussion to what we need below.} of $\wh{f}_{m,\mu} (\t, \bar{\tau}; s) $ to $\re (s) > - \frac{1}{4}$.

\begin{prop}\label{prop:depth_one_eisenstein_analytic_continuation}
Let $m \in \IN$ be odd, square-free, $\mu \in \IZ/2m\IZ$, $\t \in \IH$, $s \in \IC$.
Then $s \mapsto \wh{f}_{m,\mu} (\t, \bar{\tau}; s)$ analytically continues to $\re (s) > - \frac{1}{4}$ and $\wh{f}_{m,\mu} (\t, \bar{\tau}) := \wh{f}_{m,\mu} (\t, \bar{\tau}; 0)$ satisfies
\begin{equation*}%\label{eq:f_m_mu_alpha_m_mu_completion_Eichler_integral}
\wh{f}_{m,\mu} (\t, \bar{\tau}) = f_{m,\mu} (\t)
+ C_m \int_{-\bar{\t}}^{i \infty} 
\frac{\TH_{m, \mu} (w)}{(-i(w+\t))^{\frac{3}{2}}} dw .
\end{equation*}
Here
\begin{equation*}%\label{eq:f_m_mu_alpha_m_mu_definition}
f_{m, \mu}(\t) := \sum_{n \in \IZ_{m,\mu}^{\geqslant 0}} \a_{m,\mu} (n) e^{2 \pi i n \t}
\quad \mbox{with }
\a_{m,\mu} (n) := 
- \frac{12}{\s (m)} \sum_{d \mid m} d \d_{d^2 \mid 4 mn} H \lp \frac{4mn}{d^2} \rp .
\end{equation*}
\end{prop}
\begin{proof}
Using the decomposition \eqref{eq:Lreg_Lsing_decomp_definition} in the expansion \eqref{eq:depth_one_eisenstein_fourier_expansion}, for $\re (s) > \frac{1}{4}$ we rewrite
\begin{multline*}
\wh{f}_{m,\mu} (\t, \bar{\tau}; s) 
=  \d_{0,\mu} v^s  
- \frac{v^{-s-\frac{1}{2}}}{\sqrt{2m} \, \s_{-2s-1} (m) \z (4s+2) }
\Bigg(
\sum_{n \in \IZ_{m, \mu}} \Lreg{m,\mu} \lp n;s \rp 
v^{s+\frac{1}{2}} \calI \lp v,n;s \rp e^{2 \pi i n \t}
\\ 
+\frac{1}{4} \sum_{n \in \IZ_{m, \mu}^{\leqslant 0}} \Lsing{m,\mu} \lp n;s \rp 
v^{s+\frac{1}{2}} \frac{\calI \lp v,n;s \rp}{s} e^{2 \pi i n \t}
\Bigg) .
\end{multline*}
By Lemmas \ref{lem:analytic_part_calI_properties} and \ref{lem:Lreg_Lsing_bound}, both sums are absolutely and locally uniformly convergent for $\re (s) > -\frac{1}{4}$, justifying the splitting of the two sums and showing that they give holomorphic functions of $s$~there.
This proves the analytic continuation. The value at $s=0$ follows from Lemmas \ref{lem:analytic_part_calI_properties} and \ref{lem:arithmetic_part_at_origin}.
\end{proof}

\section{Continuing to one-dimensional non-holomorphic Eichler integrals}
\label{sec:non_hol_Eichler_continuation}

In this section we consider ($m \in \IN$, $\mu \in \IZ/2m\IZ$, $\t \in \IH$, and $s \in \IC$ with $\re(s) > \frac{1}{4}$)
\begin{equation}\label{eq:depth_one_Eichler_eisenstein_pluswriting}
g_{m,\mu} (\t, \bar{\tau}; s) := 2 is v^s 
\sum_{c\geq 1} \sum_{\substack{d \in \IZ \\ \gcd (c,d)=1}}
\frac{\Psi_{m}(M_{c,d})_{0,\mu}}{|c\t+d|^{2s} \lp c\t+d \rp^{\frac{3}{2}}}
\arctan \lp \frac{u+\frac{d}{c}}{v} \rp .
\end{equation}
This series can be treated with the methods of Section~\ref{sec:weight_3_2_eisenstein}.
Since $|\arctan (x)| \leq \frac{\pi}{2}$ for $x \in \IR$, just like \eqref{eq:depth_one_eisenstein}, it is absolutely and locally uniformly convergent in $(\tau, s)$ on its domain of definition. In particular, it gives a holomorphic function in~$s$ for $\re(s) > \frac{1}{4}$. Its analytic continuation can again be obtained from its Fourier expansion, which involves the same arithmetic part $\calL$ from \eqref{eq:arithemtic_part_defn} and a slightly modified analytic part, which we introduce next.

For $\k \in \frac{1}{2} \IZ$, $s \in \IC$ with $2 \re (s) + \k > 1$, $t \in \IR$, and $v \in \IR^+$ we define (writing $\calJ(v,t;s)$ if $\k = \frac{3}{2}$)
\begin{equation}\label{eq:Jk_definition}
\calJ_\k(v,t;s) := -2i v^{1-\k-s} e^{2 \pi t v}
\int_{-\infty}^\infty \frac{e^{-2\pi i t v x} \arctan (x)}{(1-ix)^\k \lp 1+x^2\rp^s} dx .
\end{equation}
Like $\calI_k$ from \eqref{eq:Ik_definition}, for the ranges assumed above, the function $\calJ_\k(v,t;s)$ is continuous in $(v,t,s)$ and holomorphic in $s$ because $| \arctan (x) | \leq \frac{\pi}{2}$. Using that $\arctan (z) = \frac{i}{2} ( \mathrm{Log} (1-iz) - \mathrm{Log} (1+iz) )$ and repeating the arguments from the proof of Lemma~\ref{lem:analytic_part_calI_properties}, we obtain the following analogue.

\begin{lem}\label{lem:analytic_part_calJ_properties}
Let $0 < \e < \frac{1}{8}$ be fixed, $\k \in \frac{1}{2} \IN$ with $\k \geq \frac{3}{2}$, $t\in\IR$, and $v\in\IR^+$.
\begin{enumerate}[leftmargin=*]
\item[\rm(1)] For $s \in \IC$ with $\frac{1-\k}{2} + \e \leq \re (s) \leq \frac{1}{\e}$ and $|\im (s)| \leq \frac{1}{\e}$, we have
\begin{equation*}
\calJ_\k (v,t;s)  v^{s+\k -1} e^{-2\pi t v} \ll_{\k, \e} e^{- \pi |t| v} .
\end{equation*}

\item[\rm(2)] For $t \leq 0$, we have 
\begin{equation*}
\calJ_\k (v,t;0) = 2 \pi v^{1-\k} 
\int_2^\infty \frac{e^{2 \pi t v x}}{x^\k} dx.
\end{equation*}
\end{enumerate}
\end{lem}

With $\calL$ from \eqref{eq:arithemtic_part_defn} and $\calJ$ from \eqref{eq:Jk_definition}, we again apply Poisson summation to obtain (for $\re (s) > \frac{1}{4}$)
\begin{equation*}%\label{eq:depth_one_Eichler_Eisenstein_fourier_expansion}
g_{m,\mu} (\t, \bar{\tau}; s) 
=  s \sum_{n \in \IZ_{m, \mu}} \calL_{m,\mu} \lp n;s \rp 
\calJ \lp v,n;s \rp e^{2 \pi i n \t}.
\end{equation*}
As in Proposition~\ref{prop:depth_one_eisenstein_analytic_continuation}, we use this expansion along with Lemmas \ref{lem:arithmetic_part_at_origin}, \ref{lem:Lreg_Lsing_bound}, and \ref{lem:analytic_part_calJ_properties} to obtain the analytic continuation of $g_{m,\mu} (\t, \bar{\tau}; s)$, restricting $m$ to be odd and square-free.

\begin{prop}\label{prop:depth_one_Eichler_eisenstein_analytic_continuation}
Let $m \in \IN$ be odd, square-free, $\mu \in \IZ/2m\IZ$, $\t \in \IH$, and $s \in \IC$.
Then $s \mapsto g_{m,\mu} (\t, \bar{\tau}; s)$ analytically continues to $\re (s) > - \frac{1}{4}$ and $g_{m,\mu} (\t, \bar{\tau}) := g_{m,\mu} (\t, \bar{\tau}; 0)$ satisfies
\begin{equation*}
g_{m,\mu} (\t, \bar{\tau}) = 
- C_m \int_{-\bar{\t}}^{i \infty} 
\frac{\TH_{m, \mu} (w)}{(-i(w+\t))^{\frac{3}{2}}} dw .
\end{equation*}
\end{prop}

\section{First Observations on the Coupled Eisenstein Series}\label{sec:first_properties_coupled_eisenstein}

In this section, we investigate the coupled Eisenstein series defined in \eqref{eq:coupled_Eisenstein_definition} for $m_1,m_2 \in \IN$, $\mu_j \in \IZ/2m_j\IZ$ for $j \in \{1,2\}$, $\t \in \IH$, and $s \in \IC$ with $\re (s) > \frac{1}{4}$.
Since $|\arctan (x)| \leq \frac{\pi}{2}$ for $x \in \IR$, this series is absolutely and locally uniformly convergent in $(\tau, s)$ in its domain of definition and thus yields a holomorphic function in~$s$ for $\re(s) > \frac{1}{4}$. Note that
\begin{equation*}
\frac{c_1c_2|\t|^2+(c_1d_2+d_1c_2)u+d_1d_2}{(c_1d_2-d_1c_2)v}
=
i \frac{(c_1\t+d_1)(c_2\bar{\t}+d_2)+(c_1\bar{\t}+d_1)(c_2\t+d_2)}{(c_1\t+d_1)(c_2\bar{\t}+d_2)-(c_1\bar{\t}+d_1)(c_2\t+d_2)} .
\end{equation*}
This relation together with the fact that the symplectic form $c_1 d_2 - c_2 d_1$ in $(c_1,d_1)$ and $(c_2,d_2)$ is $\SL_2 (\IZ)$-invariant implies the modular properties of $\wh{H}_{\vc{m},\vc{\mu}} (\t, \bar{\t};s)$ as (for $M = \pmat{a & b \\ c&d} \in \SL_2(\IZ)$)
\begin{equation}\label{eq:whH_modular_transformation}
\wh{H}_{\vc{m},\vc{\mu}} \lp \frac{a\t+b}{c \t+d}, \frac{a\bar{\t}+b}{c\bar{\t}+d}; s \rp
= (c \t + d)^3 
\sum_{\substack{\nu_1 \in \IZ/2m_1\IZ \\ \nu_2 \in \IZ/2m_2\IZ}} \!\!
\Psi^*_{m_1}(M)_{\mu_1,\nu_1}  \Psi^*_{m_2}(M)_{\mu_2,\nu_2} 
\wh{H}_{\vc{m},\vc{\nu}} (\t, \bar{\t}; s) .
\end{equation}
By \eqref{eq:weil_multiplier_minus_cd_property}, we find
\begin{equation*}
\wh{H}_{\vc{m},\vc{\mu}} (\t, \bar{\t};s)
= \wh{H}_{\vc{m},(-\mu_1,\mu_2)} (\t, \bar{\t};s)
= \wh{H}_{\vc{m},-\vc{\mu}} (\t, \bar{\t};s) 
\end{equation*}
and that the contributions of $(c_1,d_1)$ and $(-c_1,-d_1)$ to \eqref{eq:coupled_Eisenstein_definition} are the same (similarly for $(c_2,d_2)$ and $(-c_2,-d_2)$).
Using this fact and splitting off the contributions where $c_1c_2=0$,
\begin{align}
\wh{H}_{\vc{m},\vc{\mu}} (\t, \bar{\t};s)
&= 
2isv^{2s} \d_{0,\mu_2} \sum_{c_1\geq 1} 
\sum_{\substack{d_1 \in \IZ \\ \gcd (c_1,d_1)=1}}
\frac{\Psi_{m_1}(M_{c_1,d_1})_{0,\mu_1}}{|c_1\t+d_1|^{2s}(c_1\t+d_1)^{\frac{3}{2}}}
\arctan \lp \frac{u+\frac{d_1}{c_1}}{v} \rp
\notag \\
&\hspace{1cm}
- 2isv^{2s} \d_{0,\mu_1} \sum_{c_2\geq 1} 
\sum_{\substack{d_2 \in \IZ \\ \gcd (c_2,d_2)=1}}
\frac{\Psi_{m_2}(M_{c_2,d_2})_{0,\mu_2}}{|c_2\t+d_2|^{2s}(c_2\t+d_2)^{\frac{3}{2}}}
\arctan \lp \frac{u+\frac{d_2}{c_2}}{v} \rp
\notag \\
&\hspace{1cm}
+2isv^{2s}
\sum_{c_1,c_2\geq 1} 
\sum_{\substack{d_1,d_2 \in \IZ \\ c_1d_2-d_1c_2 \neq 0 \\ \gcd (c_j,d_j)=1}}
\frac{\Psi_{m_1}(M_{c_1,d_1})_{0,\mu_1}}{|c_1\t+d_1|^{2s}(c_1\t+d_1)^{\frac{3}{2}}} \frac{\Psi_{m_2}(M_{c_2,d_2})_{0,\mu_2}}{|c_2\t+d_2|^{2s}(c_2\t+d_2)^{\frac{3}{2}}}
\notag \\
&\hspace{5cm} \times  
\arctan \lp \frac{c_1c_2|\t|^2+(c_1d_2+d_1c_2)u+d_1d_2}{(c_1d_2-d_1c_2)v} \rp .
\label{eq:hatH_rewriting1} 
\end{align}
The variable inside the final $\arctan$ can be rewritten as $\frac{x_1 x_2 + 1}{x_2 - x_1}$ with 
\smash{$x_j \!:=\! \frac{1}{v} (u\!+\!\frac{d_j}{c_j})$}.
Note that for $x_1,x_2 \in \IR$ with $x_1 \neq x_2$ (equivalent to $c_1d_2-d_1c_2 \neq 0$ as $c_1,c_2 \in \IN$), we have
\begin{equation}\label{eq:arctan_summation_identity}
\arctan \lp  \frac{x_1 x_2 + 1}{x_2 - x_1} \rp
=
\arctan (x_1) - \arctan (x_2) + \frac{\pi}{2} \sgn (x_2-x_1) .
\end{equation}
Using this identity in the third term of \eqref{eq:hatH_rewriting1}, we can remove the condition $c_1 d_2 - d_1c_2 \neq 0$ there, since the right-hand side of \eqref{eq:arctan_summation_identity} vanishes at $x_1 = x_2$.
This allows us to rewrite \eqref{eq:hatH_rewriting1} as (for~$\re (s) > \frac{1}{4}$) 
\begin{equation}\label{eq:hatH_rewriting2_g_f_F}
\wh{H}_{\vc{m},\vc{\mu}} (\t, \bar{\t};s)
=
g_{m_1,\mu_1} (\t, \bar{\t};s)
\wh{f}_{m_2, \mu_2} (\t, \bar{\t};s)
-
g_{m_2,\mu_2} (\t, \bar{\t};s)
\wh{f}_{m_1, \mu_1} (\t, \bar{\t};s)
+
\calF_{\vc{m},\vc{\mu}} (\t, \bar{\t};s) ,
\end{equation}
where $g$ is defined in \eqref{eq:depth_one_Eichler_eisenstein_pluswriting}, the Eisenstein series $\wh{f}$ in \eqref{eq:depth_one_eisenstein}, and we let ($m_1,m_2 \in \IN$, $\mu_j \in \IZ/2m_j\IZ$ for $j \in \{1,2\}$, $\t \in \IH$, and $s \in \IC$ with $\re (s) > \frac{1}{4}$)
\begin{equation}\label{eq:dimension_two_part_Eisenstein}
\calF_{\vc{m},\vc{\mu}} (\t, \bar{\t};s)
:=
\pi isv^{2s}
\sum_{c_1,c_2 \geq 1} 
\sum_{\substack{d_1,d_2 \in \IZ \\ \gcd (c_j,d_j)=1}} 
\sgn \lp \frac{d_2}{c_2} - \frac{d_1}{c_1} \rp 
\prod_{j \in \{1,2\}}
\frac{\Psi_{m_j}(M_{c_j,d_j})_{0,\mu_j}}{|c_j\t+d_j|^{2s} \lp c_j\t+d_j \rp^{\frac{3}{2}}} .
\end{equation}
Since $\sgn (x) \leq 1$ for $x \in \IR$, this series is absolutely and locally uniformly convergent in $\tau$ and $s$ in its domain of definition just like \eqref{eq:depth_one_eisenstein}. It yields a holomorphic function in~$s$ 
for $\re(s) > \frac{1}{4}$.
The analytic continuations of \smash{$\wh{f}$} and $g$ are given in Propositions \ref{prop:depth_one_eisenstein_analytic_continuation} and \ref{prop:depth_one_Eichler_eisenstein_analytic_continuation}, respectively. 
We therefore turn to the analytic continuation of $\calF_{\vc{m},\vc{\mu}} (\t, \bar{\t};s)$ from its definition in \eqref{eq:dimension_two_part_Eisenstein} for $\re (s) > \frac{1}{4}$.

\section{Analytic and Arithmetic Constituents of the Coupled Eisenstein Series}
\label{sec:analytic_atirhmetic_parts_coupled_eisenstein}

Before proving the analytic continuation of $\calF_{\vc{m},\vc{\mu}} (\t, \bar{\t};s)$ in Section~\ref{sec:analytic_continuation_coupled_eisenstein}, we establish the relevant analytic and arithmetic properties.

\subsection{Properties of the two-dimensional analytic part}
\label{sec:analytic_part_TwoDim}

We first consider the two-dimensional analogue of $\calI_k$ that we need in the analysis of $\calF_{\vc{m},\vc{\mu}} (\t, \bar{\t};s)$.
For $\k_1,\k_2 \in \frac{1}{2} \IZ$, $t,\w \in \IR$, $v \in \IR^+$, and $s \in \IC$ with $2 \re (s) + \min \{\k_1,\k_2\} > 1$, we define 
\begin{equation}\label{eq:Kk_definition}
\calK_{\vc{\k}} (v,t,\w;s) := \pi i v^{2-\k_1-\k_2-2s} e^{2 \pi t v}
\int_{\IR^2} \frac{\sgn (y-x) e^{-2\pi i (t -\w) v x- 2\pi i \w v y}}{(1-ix)^{\k_1} \lp 1+x^2\rp^s (1-iy)^{\k_2} \lp 1+y^2\rp^s} dx d y,
\end{equation}
which is continuous in $(v,t,\w,s)$ and holomorphic in $s$, for the ranges assumed. 
Our goal in this section is to generalize Lemma~\ref{lem:analytic_part_calI_properties} for $\calI_\k$ to $\calK_{\vc{\k}}$.
We first bound $\calK_{\vc{\k}}$.

\begin{prop}\label{prop:analytic_part_calK_properties}
Let $0 < \e < \frac{1}{8}$ be fixed, $\k_1, \k_2 \in \frac{1}{2} \IN$ with $\k_1,\k_2 \geq \frac{3}{2}$, $t, \w \in\IR$, and $v\in\IR^+$. Then, for $s \in \IC$ with $\frac{1-\min \{\k_1, \k_2\}}{2} + \e \leq \re (s) \leq \frac{1}{\e}$ and $|\im (s)| \leq \frac{1}{\e}$, we have
\begin{equation*}
\lp \calK_{\vc{\k}} (v,t,\w;s) 
- \frac{\d_{\w \not\in [0,1)}}{\lfloor \w \rfloor}  \calI_{\k_1+\k_2} (v,t;2s) \rp
v^{2s+\k_1+\k_2 -2} e^{-2\pi t v} 
\ll_{\vc{\k}, \e} \lp 1 + \frac{1}{v^2} \rp \frac{e^{- \pi |t| v}}{1+\w^2} .
\end{equation*}
\end{prop}
\begin{proof}
Integrating by parts twice, we obtain (for $x \in \IR$)
\begin{equation}\label{eq:details_analytic_part_calK_properties_intParts}
\int_{-\infty}^\infty \frac{\sgn (y-x) e^{-2\pi i \w v y}}{(1-iy)^{\k_2} \lp 1+y^2\rp^s} d y
=
\frac{\d_{\w \not\in [0,1)}}{\pi i \lfloor \w \rfloor v}
\frac{e^{-2\pi i \w v x}}{(1-ix)^{\k_2} \lp 1+x^2\rp^s}
+ e^{-2\pi i \w v x} \phi_{\k_2} (v,\w; x,s),
\end{equation}
where
\begin{align}\label{eq:details_analytic_part_calK_properties_intParts2}
&\phi_{\k_2} (v,\w; x,s) := \frac{i \d_{\w \not\in [0,1)}}{2 \pi^2 \lfloor \w \rfloor^2 v^2} 
\frac{1}{(1-ix)^{\k_2} \lp 1+x^2\rp^s}
\lp 2 \pi \{\w \} v - \frac{\k_2+s}{1-ix} + \frac{s}{1+ix}  \rp
\\
& 
+ \int_{-\infty}^\infty \frac{\sgn (y) e^{-2\pi i \w v y}}{(1-i(x+y))^{\k_2+s} (1+i(x+y))^{s}} 
\Bigg(
\frac{\d_{\w \not\in [0,1)}}{4 \pi^2 \lfloor \w \rfloor^2 v^2} 
\lp 2 \pi \{\w \} v  - \frac{\k_2+s}{1-i(x+y)}  + \frac{s}{1+i(x+y)}\rp^2
\notag \\  & \hspace{5.3cm}
+\frac{\d_{\w \not\in [0,1)}}{4 \pi^2 \lfloor \w \rfloor^2 v^2} 
\lp \frac{\k_2+s}{(1-i(x+y))^2} + \frac{s}{(1+i(x+y))^2} \rp
+ \d_{\w \in [0,1)} 
\Bigg) d y .
\notag 
\end{align}
Plugging this into \eqref{eq:Kk_definition}, we find
\begin{equation*}
\calK_{\vc{\k}} (v,t,\w;s) 
=
\frac{\d_{\w \not\in [0,1)}}{\lfloor \w \rfloor}  \calI_{\k_1+\k_2} (v,t;2s)
+ \pi i v^{2-\k_1-\k_2-2s} e^{2 \pi t v}
\int_{-\infty}^\infty \frac{e^{-2\pi i t v x} \phi_{\k_2} (v,\w; x,s) }{(1-ix)^{\k_1} \lp 1+x^2\rp^s} dx .
\end{equation*}
So the quantity to be bounded is
\begin{equation}\label{eq:details_calX_integral_expression}
\pi i  \int_{-\infty}^\infty 
\frac{e^{-2\pi i t v x} \phi_{\k_2} (v,\w; x,s) }{(1-ix)^{\k_1} \lp 1+x^2\rp^s} dx .
\end{equation}
For the values of $\k_2,v,\w,s$ assumed in the proposition, 
the function $z \mapsto \phi_{\k_2} (v,\w; z,s)$ defined by \eqref{eq:details_analytic_part_calK_properties_intParts2} (letting $x \mapsto z \in \IC$) is holomorphic for $z \in \IC$ with $|\im (z)| < 1$.
Also, for any fixed $0 < \d < 1$, 
\begin{equation}\label{eq:integral_phi_kappa_bound}
\phi_{\k_2} (v,\w; z,s) \ll_{\k_2, \e, \d}
\lp 1 + \frac{1}{v^2} \rp \frac{1}{1+\w^2}
\quad \mbox{for } |\im (z)| \leq 1-\d .
\end{equation}
This follows by bounding \eqref{eq:details_analytic_part_calK_properties_intParts2} as
\begin{equation*}
\phi_{\k_2} (v,\w; z,s) \ll_{\k_2, \e, \d}
\frac{v+1}{v^2} \frac{1}{1+\w^2} 
+ 
\lp 1 + \frac{1}{v^2} \rp \frac{1}{1+\w^2}
\int_{-\infty}^\infty 
\frac{1}{\left| (1-i(z+y))^{\k_2+s} (1+i(z+y))^{s} \right|} dy
\end{equation*}
and then noting that the remaining integral is $\ll_{\k_2, \e, \d} 1$.
Using the holomorphy of $z \mapsto \phi_{\k_2} (v,\w; z,s)$ for $|\im (z)| < 1$ and the bound \eqref{eq:integral_phi_kappa_bound} with $\d=\frac{1}{2}$, we can now argue as in Lemma~\ref{lem:analytic_part_calI_properties}~(1) to shift the integration path in \eqref{eq:details_calX_integral_expression} by $-\frac{i}{2}$ if $t > 0$ and by $\frac{i}{2}$ if $t \leq 0$. This proves the claimed bound.
\end{proof}

We now determine the value of $\calK_{\vc{\k}} (v,t,\w;s)$ for $s=0$.

\begin{prop}\label{prop:analytic_part_calK_value_at_zero}
Let $\k_1, \k_2 \in \IN + \frac{1}{2}$, $t, \w \in\IR$, and $v\in\IR^+$. Then we have
\begin{equation*}
\calK_{\vc{\k}} (v,t,\w;0) \!=\!
\begin{cases}
- \frac{(2\pi)^{\k_1+\k_2}}{\GG (\k_1+\k_2)} \frac{t^{\k_1+\k_2-1}}{t-\w}
{}_2F_1 \lp 1,\k_1; \k_1+\k_2; \frac{t}{t-\w} \rp
&\!\!\mbox{if } t>0 \mbox{ and } \w \leq 0,
\vspace{1ex} \\ 
\frac{(2\pi)^{\k_1+\k_2}}{\GG (\k_1+\k_2)} \frac{t^{\k_1+\k_2-1}}{\w}
{}_2F_1 \lp 1,\k_2; \k_1+\k_2; \frac{t}{\w} \rp
&\!\!\mbox{if } t>0 \mbox{ and } \w \geq t,
\vspace{1ex} \\
- \frac{(\k_1+\k_2-1)(2\pi)^{\k_1+\k_2}}{(\k_2-1) \GG (\k_1+\k_2)} t^{\k_1+\k_2-2}
{}_2F_1 \lp 1,2-\k_1-\k_2; 2-\k_2; \frac{\w}{t} \rp
&\!\!\mbox{if } t>0 \mbox{ and } 0 < \w < t,
\vspace{1ex} \\
0 &\!\!\mbox{if } t \leq 0 .
\end{cases} 
\end{equation*} 
\end{prop}
\begin{proof}
We let $s=0$ in \eqref{eq:Kk_definition} and replace $\sgn (y-x)$ in the integral with (for $k \in \IN$)
\begin{equation*}
\sgn (y-x) = \frac{2}{\pi} \lim_{k \to \infty} \mathrm{Si}(2 \pi k v (y-x)) ,
\end{equation*}
where \smash{$\mathrm{Si} (x) := \int_0^x \frac{\sin (\th)}{\th} d \th$} is the \textit{sine integral}.
Since $|\mathrm{Si} (x)| \leq \mathrm{Si} (\pi)$ for all $x \in \IR$, by the dominated convergence theorem, we can commute the limit over $k$ with the integral and find 
\begin{equation}\label{eq:details_analytic_part_calK_value_at_zero1}
\calK_{\vc{\k}} (v,t,\w;0) 
= 
2 i v^{2-\k_1-\k_2} e^{2 \pi t v} 
\lim_{k \to \infty} 
\int_{\IR^2} \!
\frac{e^{-2\pi i (t -\w) v x- 2\pi i \w v y}}{(1-ix)^{\k_1} (1-iy)^{\k_2}} 
\int_0^{k} \frac{\sin (2 \pi v (y-x) \th)}{\th} d \th
dx d y.
\end{equation}
Since $|y-x| \leq 2 \max \{|x|,|y|\}$, we have
\begin{equation*}
\int_0^{k} \! \frac{|\sin (2 \pi v (y-x) \th)|}{\th} d \th
\leq
\int_0^{1+2 \pi k v |y-x| } \! \frac{| \sin (\th) |}{\th} d \th
\leq
\mathrm{Si}(1) + \log \lp 1+4 \pi k v |x| \rp + \log \lp 1+4 \pi k v |y| \rp .
\end{equation*}
So the iterated integral over $x,y,\th$ in \eqref{eq:details_analytic_part_calK_value_at_zero1} is absolutely convergent and we can interchange the integrals over $\th$ and $(x,y)$. Then we use \eqref{eq:Ik_definition} and
Lemma~\ref{lem:analytic_part_calI_properties}~(2) to evaluate 
\begin{multline}
\calK_{\vc{\k}} (v,t,\w;0) 
= 
\frac{(2\pi)^{\k_1+\k_2}}{\GG (\k_1) \GG (\k_2)}
\int_0^{\infty} \! \frac{1}{\th} 
\big( (t-\w+\th)^{\k_1-1} (\w-\th)^{\k_2-1} \d_{\w-t < \th < \w}
\\
- (t-\w-\th)^{\k_1-1} (\w+\th)^{\k_2-1} \d_{- \w < \th < t-\w} \big) d \th .
\label{eq:details_analytic_part_calK_value_at_zero4}
\end{multline}
Note that if $t \leq 0$, then the integrand is identically zero and we have $\calK_{\vc{\k}} (v,t,\w;0)=0$ as claimed.

Suppose next that $t > 0$ and $0 < \w < t$. For any $0 < \e < \min \{\w, t-\w\}$, we rewrite \eqref{eq:details_analytic_part_calK_value_at_zero4} as
\begin{multline*}
\calK_{\vc{\k}} (v,t,\w;0)
=
\frac{(2\pi)^{\k_1+\k_2}}{\GG (\k_1) \GG (\k_2)} \Bigg(
\int_0^{\e} \frac{(t-\w+\th)^{\k_1-1} (\w-\th)^{\k_2-1} - (t-\w-\th)^{\k_1-1} (\w+\th)^{\k_2-1}}{\th}  d \th 
\\
+ 
\lp \int_{\w-t}^{-\e} + \int_{\e}^{\w} \rp 
\frac{(t-\w+\th)^{\k_1-1} (\w-\th)^{\k_2-1}}{\th}  d \th 
\Bigg)  .
\end{multline*}
Since $\th \mapsto (t-\w+\th)^{\k_1-1} (\w-\th)^{\k_2-1}$ is differentiable at $\th = 0$ because $t-\w > 0$ and $\w > 0$, the first term vanishes as $\e \to 0^+$ and the second term yields a Cauchy principal value integral as
\begin{equation*}
\calK_{\vc{\k}} (v,t,\w;0) = 
\frac{(2\pi)^{\k_1+\k_2}}{\GG (\k_1) \GG (\k_2)}  
\mathrm{PV}\int_{\w-t}^{\w} \frac{(t-\w+\th)^{\k_1-1} (\w-\th)^{\k_2-1}}{\th}d\th .
\end{equation*}
The resulting integral is a Hilbert transform, which can be evaluated using \cite[Chapter XV (33)]{Bat}, while noting that $\cot (\pi \k_2) = 0$ because $\k_2 \in \IN + \frac{1}{2}$. This yields the proposition for the case $t > 0$ and $0 < \w < t$. The remaining cases with $t > 0$ (where $\w \leq 0$ or $\w \geq t$) follow similarly.
\end{proof}

\begin{rem}\label{rem:calK_zero_value_3_over_2}
If $\k_1 = \k_2 = \frac{3}{2}$, then we use
\begin{equation*}
{}_2F_1 \lp 1,\frac{3}{2}; 3; x \rp = \frac{4}{x^2} \lp 2-x-2 \sqrt{1-x} \rp
\andd
{}_2F_1 \lp 1,-1; \frac{1}{2}; x \rp = 1-2x
\end{equation*}
to specialize the results of Proposition \ref{prop:analytic_part_calK_value_at_zero} as follows:\footnote{Note that here and throughout we drop the subscript $\vc{\k}$ if $\k_1 = \k_2 = \frac{3}{2}$.}
\begin{equation}\label{eq:calK_3_2_values_at_zero}
\calK (v,t,\w;0) \!=\!
\begin{cases}
16 \pi^3 \lp 2 \w - t + 2\sqrt{\w (\w-t)} \rp
&\mbox{if } t>0 \mbox{ and } \w \leq 0,
\vspace{1ex} \\ 
16 \pi^3 \lp 2 \w - t - 2\sqrt{\w (\w-t)} \rp
&\mbox{if } t>0 \mbox{ and } \w \geq t,
\vspace{1ex} \\
16 \pi^3 \lp 2 \w - t \rp
&\mbox{if } t>0 \mbox{ and } 0 < \w < t,
\vspace{1ex} \\
0 &\mbox{if } t \leq 0 .
\end{cases} 
\end{equation} 
\end{rem}

We finally generalize Lemma~\ref{lem:analytic_part_calI_properties}~(3) and develop the part of $\calK_{\vc{\k}} (v,t,\w;s)$ that leads to iterated Eichler integrals of modular forms appearing in completions of depth two mock modular forms.

\begin{prop}\label{prop:analytic_part_calK_properties_negative}
Let $0 < \e < \frac{1}{8}$ be fixed, $\k_1, \k_2 \in \frac{1}{2} \IN$ with $\k_1,\k_2 \geq \frac{3}{2}$, $t, \w \in\IR$ with $\w, t-\w \leq 0$, and $v\in\IR^+$. 
Then for $s \in \IC$ with $\frac{1-\min \{\k_1, \k_2\}}{2} + \e \leq \re (s) \leq \frac{1}{\e}$ and $|\im (s)| \leq \frac{1}{\e}$, we have
\begin{equation*}
\frac{1}{s} \calK_{\vc{\k}} (v,t,\w;s) v^{2s+\k_1+\k_2 -2} e^{-2\pi t v} 
\ll_{\vc{\k}, \e} e^{- \pi |t| v} .
\end{equation*}
Furthermore, we have
\begin{equation*}
\calK'_{\vc{\k}} (v,t,\w;0) 
= 4 \pi^2 v^{2-\k_1-\k_2} 
\lp  
\int_2^\infty \! \frac{e^{2 \pi \w v x}}{x^{\k_2}} 
\int_x^\infty \! \frac{e^{2 \pi (t-\w) v y}}{y^{\k_1}} dy  dx
-
\int_2^\infty \! \frac{e^{2 \pi (t-\w) v x}}{x^{\k_1}} 
\int_x^\infty \! \frac{e^{2 \pi \w v y}}{y^{\k_2}} dy  dx
\rp .
\end{equation*}
\end{prop}
\begin{proof}

Away from $s=0$, say $|s| > \frac{1}{8}$, as in the proof of Lemma~\ref{lem:analytic_part_calI_properties}~(3), the claimed bound follows as in Proposition~\ref{prop:analytic_part_calK_properties}. This is by shifting the integration path by $\frac{i}{2}$ (since $t \leq 0$) in
\begin{equation*}%\label{eq:details_analytic_part_calK_properties_negative1}
\calK_{\vc{\k}} (v,t,\w;s) v^{2s+\k_1+\k_2 -2} e^{-2\pi t v}  
= \pi i \int_{-\infty}^\infty 
\frac{e^{-2\pi i t v x} \psi_{\k_2} (v,\w; x,s) }{(1-ix)^{\k_1} \lp 1+x^2\rp^s} dx
\end{equation*}
with
\begin{equation*}%\label{eq:details_analytic_part_calK_properties_negative2}
\psi_{\k_2} (v,\w; x,s) := \int_{-\infty}^\infty \frac{\sgn (y) e^{-2\pi i \w v y}}{(1-i(x+y))^{\k_2+s} (1+i(x+y))^{s}} dy .
\end{equation*}
So assume $|s| \leq \frac{1}{8}$ from now on.
Writing $\sgn (y-x) = \d_{y>x} - \d_{x>y}$ in \eqref{eq:Kk_definition}, we get
\begin{equation}\label{eq:calK_calN_rewriting}
\calK_{\vc{\k}} (v,t,\w;s) 
=
\calN_{\vc{\k}} (v,t,\w;s) - \calN_{\k_2,\k_1} (v,t,t-\w;s) ,
\end{equation}
where
\begin{align}
\label{eq:calN_definition}
\calN_{\vc{\k}} (v,t,\w;s) 
&:=
\pi i v^{1-\k_1-s} e^{2 \pi (t-\w) v}
\int_{-\infty}^\infty
\frac{e^{-2\pi i (t -\w) v x}}{\lp 1-ix \rp^{\k_1+s} \lp 1+ix \rp^{s}}  
\mathcal{Y}_{\k_2} (v,\w;x,s) d x,
\\
\label{eq:calY_definition_complex}
\mathcal{Y}_{\k_2} (v,\w;z,s) 
&:=
v^{1-\k_2-s} e^{2 \pi \w v}
\int_z^\infty 
\frac{e^{-2\pi i \w v w}}{\lp 1-iw \rp^{\k_2+s} \lp 1+iw \rp^{s}} dw .
\end{align}
Here the latter is defined for $z \in \IC \setminus ((-i\infty,-i] \cup [i,i\infty))$ and with an integration path avoiding the principal branch cuts.
We first record some properties of $\mathcal{Y}_{\k_2} (v,\w;z,s)$.

\noindent (1) The function $z \mapsto \mathcal{Y}_{\k_2} (v,\w;z,s)$ is holomorphic on $\IC \setminus ((-i\infty,-i] \cup [i,i\infty))$. It is also holomorphic along the cuts $(-i \infty,-i)$ and $(i,i\infty)$ for the extension on both sheets continued from left and right of the cuts (because the integrand satisfies this property).

\noindent (2) As $|s| \leq \frac{1}{8}$, the growth of the integrand in  \eqref{eq:calY_definition_complex} at $w=i$ is at most like \smash{$|1+iw|^{-\frac{1}{8}}$}. So the integral is finite for $z=i$ as well and $\mathcal{Y}_{\k_2} (v,\w;z,s)$ extends continuously there (that is also so for the extension of $\mathcal{Y}_{\k_2}$ to $[i,i\infty)$ from either left or right of the cut).

\noindent (3) Since $\w \leq 0$, the integrand in \eqref{eq:calY_definition_complex} is $\ll_{\k_2,\e} |w|^{-1-2\e}$ for $|w| \geq 2$ (where $|1 \pm i w| \geq 1$).
By using integration paths as shown in Figure~\ref{fig:Yk_bound_proof_path} (where the circular arcs are centered at the origin and have radius $|z|$), we find
\begin{equation}\label{eq:calY_bound_upper_half_plane}
\mathcal{Y}_{\k_2} (v,\w;z,s) \ll_{\k_2,v,\w,s} 1
\quad \mbox{for any } \im (z) \geq 0.
\end{equation}
This also holds for the limits of $\mathcal{Y}_{\k_2} (v,\w;z,s)$ from the left or right of the branch-cut on $[i,i\infty)$.

\begin{figure}[h!]
\vspace{-10pt}
\centering
\includegraphics[scale=0.18]{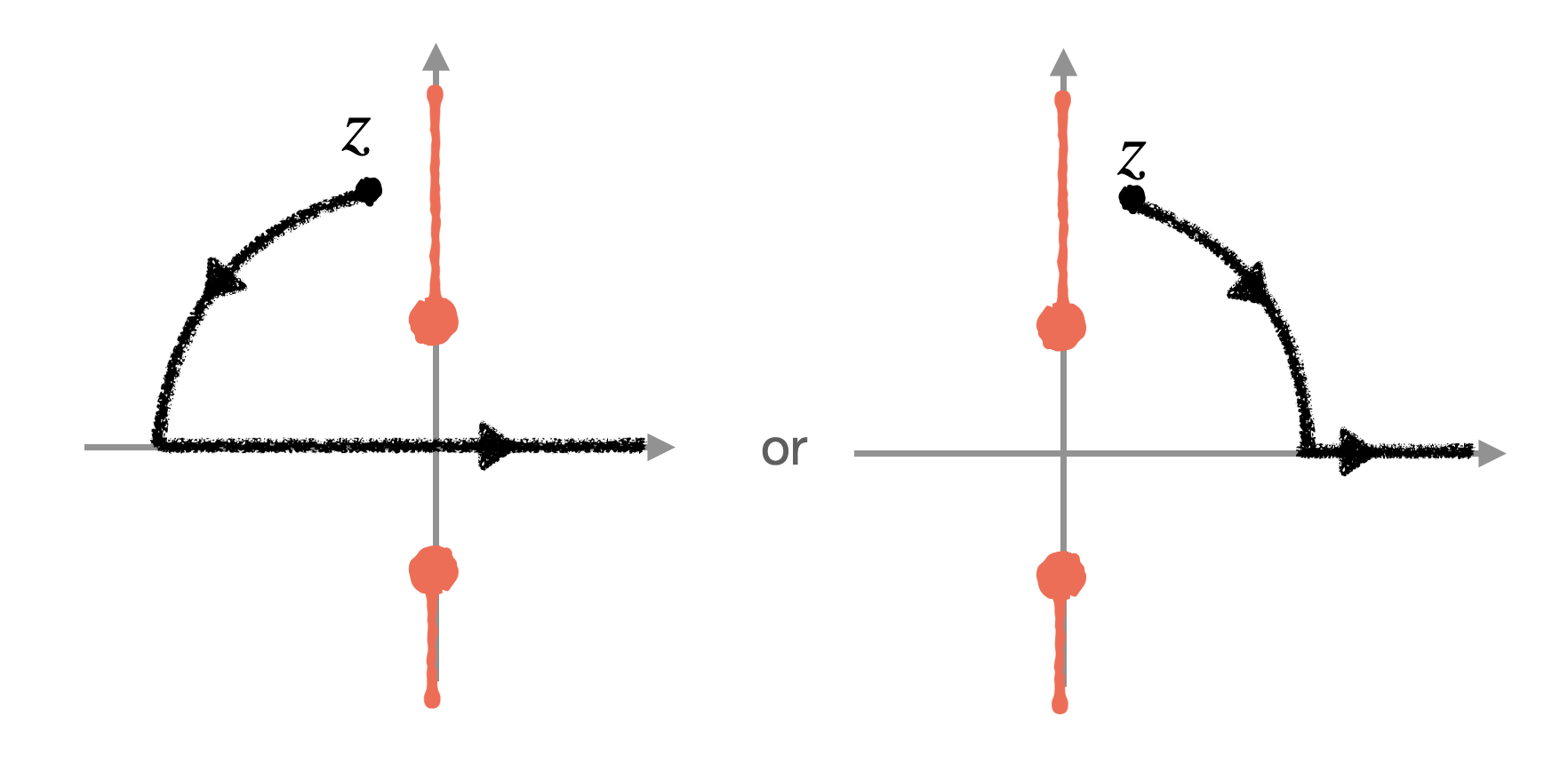}
\vspace{-20pt}
\caption{}
\label{fig:Yk_bound_proof_path}
\end{figure}

\noindent (4) Since the circular arc integrals above go to zero in the upper-half plane as the radius tends to infinity, we rewrite \eqref{eq:calY_definition_complex} as
\begin{equation*}
\mathcal{Y}_{\k_2} (v,\w;z,s) =
v^{1-\k_2-s} e^{2 \pi \w v}
\int_z^{i \infty}
\frac{e^{-2\pi i \w v w}}{\lp 1-iw \rp^{\k_2+s} \lp 1+iw \rp^{s}} dw 
\end{equation*}
with an integration path going to $i \infty$ from the right-side of the branch-cut.
Then defining
\begin{equation*}
\mathcal{Y}^\pm_{\k_2} (v,\w;z,s) 
:= \lim_{\d \to 0^+} \mathcal{Y}_{\k_2} (v,\w;z \pm \d,s),
\end{equation*}
for $x \geq 1$, we have
\begin{equation}\label{eq:calY_plus_minus_integrals}
\mathcal{Y}^\pm_{\k_2} (v,\w;i x,s) = 
v^{1-\k_2-s} e^{2 \pi \w v}
\int_{\CC_x^\pm}
\frac{e^{-2\pi i \w v w}}{\lp 1-iw \rp^{\k_2+s} \lp 1+iw \rp^{s}} dw .
\end{equation}
Here the integration paths $\CC_x^\pm$ are contained on the ray from $i$ to $i \infty$ with the limiting values of the integrand from left or right of the cut used as in Figure~\ref{fig:int_path_Cx_Plus_minus}.

\begin{figure}[h!]
\vspace{-10pt}
\centering
\includegraphics[scale=0.24]{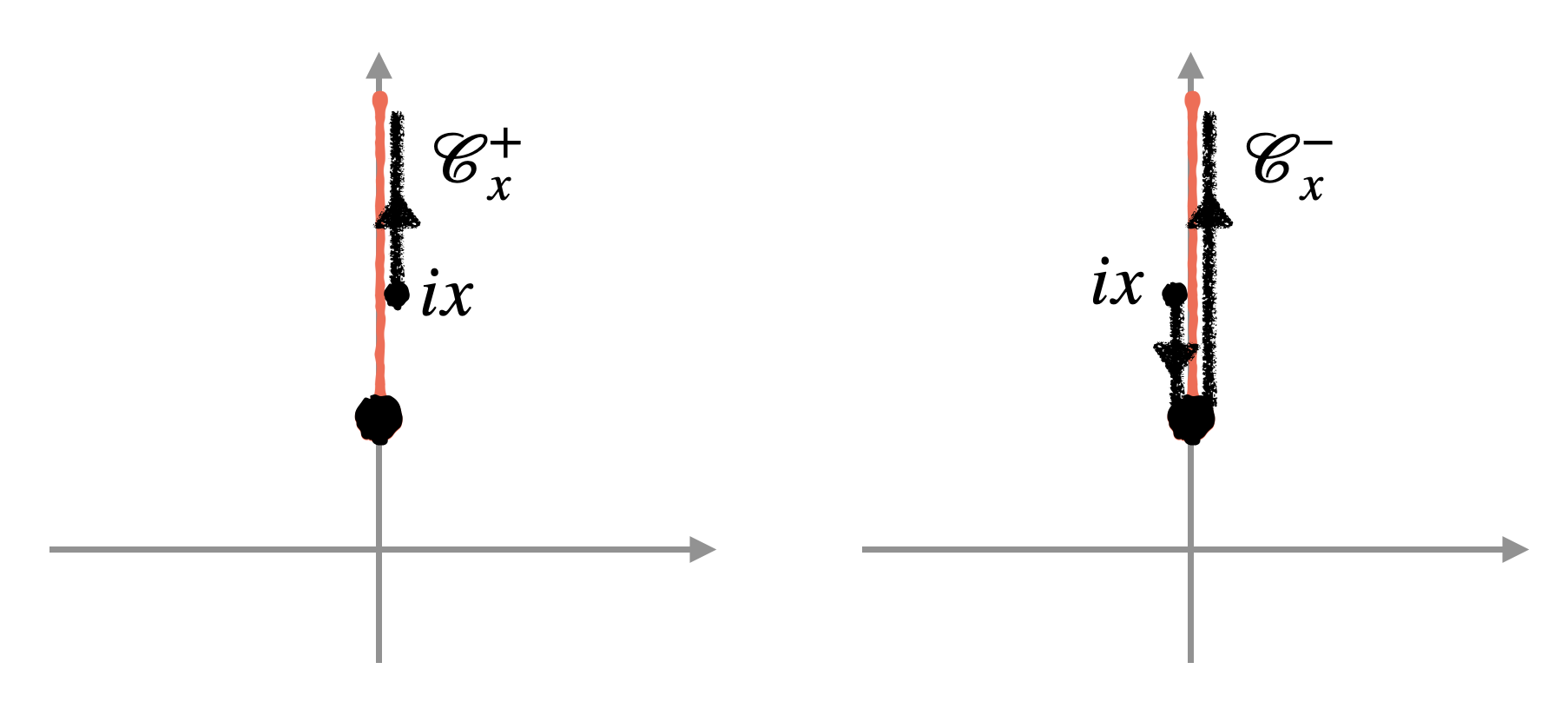}
\vspace{-20pt}
\caption{}
\label{fig:int_path_Cx_Plus_minus}
\end{figure}

\noindent (5) We can also compare the integrals in \eqref{eq:calY_plus_minus_integrals} to those taken on paths going to $i \infty$ from the left-side of the branch-cut. For example, for $x \geq 1$ we find
\begin{equation*}
e^{2 \pi i s} \mathcal{Y}^+_{\k_2} (v,\w;i x,s)
=
v^{1-\k_2-s} e^{2 \pi \w v}
\int_{\CD_x^-}
\frac{e^{-2\pi i \w v w}}{\lp 1-iw \rp^{\k_2+s} \lp 1+iw \rp^{s}} dw,
\end{equation*}
where the integration path $\CD_x^-$ is as shown in Figure~\ref{fig:int_path_Dx_minus}.

\begin{figure}[h!]
\vspace{-10pt}
\centering
\includegraphics[scale=0.24]{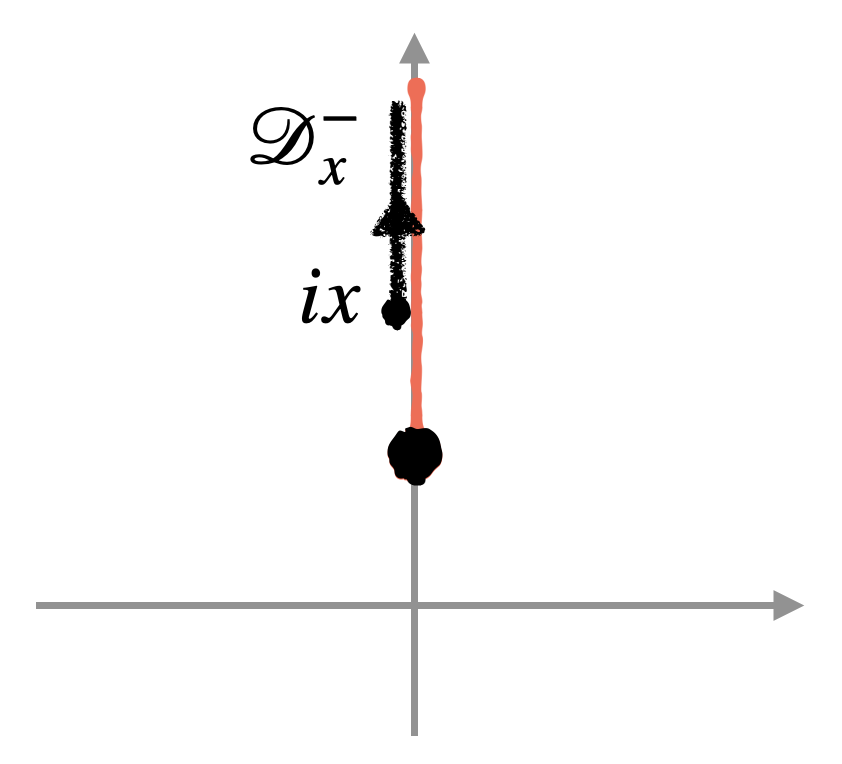}
\vspace{-10pt}
\caption{}
\vspace{-10pt}
\label{fig:int_path_Dx_minus}
\end{figure}

\noindent Note that $\int_{\CC_x^-} - \int_{\CD_x^-} = \int_{\CD}$, where $\CD$ is the integration path given in Figure~\ref{fig:main_int_path_D} to describe $\calI_\k$ as~\eqref{eq:Ik_expression_D_integration_path}. This yields the monodromy relation
\begin{equation}\label{eq:calY_monodromy}
\mathcal{Y}^-_{\k_2} (v,\w;i x,s) - e^{2 \pi i s} \mathcal{Y}^+_{\k_2} (v,\w;i x,s)
=
\calI_{\k_2} (v,\w;s) 
\quad \mbox{for } x \geq 1. 
\end{equation}

Returning to~\eqref{eq:calN_definition},
since $t-\w \leq 0$, we can deform the integration path in \eqref{eq:calN_definition} to~$\CD$, recalling \eqref{eq:calY_bound_upper_half_plane} along with the fact that $|s| \leq \frac{1}{8}$ for the growth at $i$.
This lets us rewrite
\begin{multline*}
\calN_{\vc{\k}} (v,t,\w;s)
=
- \pi v^{1-\k_1-s} e^{2 \pi (t-\w) v}
\\ \times
\int_1^\infty \frac{e^{2\pi (t -\w) v x}}{(x+1)^{\k_1+s} (x-1)^s}
\lp e^{-\pi i s} \mathcal{Y}^+_{\k_2} (v,\w;i x,s) - e^{\pi i s} \mathcal{Y}^-_{\k_2} (v,\w;i x,s) \rp dx .
\end{multline*}
Using \eqref{eq:calY_monodromy} to eliminate $\mathcal{Y}^-_{\k_2} (v,\w;i x,s)$ from this expression and noting \eqref{eq:Ik_expression_D_integration_path} to 
take the integral for the term with $\calI_{\k_2}$, we obtain
\begin{multline*}
\calN_{\vc{\k}} (v,t,\w;s) = 
2 \pi i e^{\pi i s} \sin (2 \pi s) v^{1-\k_1-s} e^{2 \pi (t-\w) v}
\int_1^\infty \frac{e^{2\pi (t -\w) v x}}{(x+1)^{\k_1+s} (x-1)^s}
\mathcal{Y}^+_{\k_2} (v,\w;i x,s)  dx
\\
+ \frac{\pi e^{\pi i s}}{2 \sin (\pi s)} 
\calI_{\k_1} (v,t-\w;s) \calI_{\k_2} (v,\w;s) .
\end{multline*}
Plugging this in \eqref{eq:calK_calN_rewriting} and recalling \eqref{eq:calY_plus_minus_integrals}, we finally find (for $|s| \leq \frac{1}{8}$ and $\w, t-\w \leq 0$)
\begin{multline*}
\frac{\calK_{\vc{\k}} (v,t,\w;s)}{s} = 
- 2\pi \frac{\sin (2 \pi s)}{s} v^{2-\k_1-\k_2-2s} e^{4 \pi t v}
\Bigg( \int_0^\infty \!\! \frac{e^{2\pi (t -\w) v x}}{(x+2)^{\k_1+s} x^s}
\int_x^\infty \!\! \frac{e^{2\pi \w v y}}{(y+2)^{\k_2+s} y^s} d y dx 
\\
- \int_0^\infty \!\! \frac{e^{2\pi \w v x}}{(x+2)^{\k_2+s} x^s}
\int_x^\infty \!\! \frac{e^{2\pi (t-\w) v y}}{(y+2)^{\k_1+s} y^s} d y dx\Bigg).
\end{multline*}
Letting $s \to 0$ immediately gives the claimed expression for $\calK'_{\vc{\k}} (v,t,\w;0)$.
The claimed bound also follows for $|s| \leq \frac{1}{8}$, since the involved integrals are $\ll 1$ there (note $t-\w,\w\leq 0$).
\end{proof}

\subsection{Bounding a convolution sum in arithmetic pieces}\label{sec:bound_arithmetic_convolution_s}

The arithmetic pieces of the coupled Eisenstein series turn out to be  bilinear combinations of arithmetic parts $\calL_{m,\mu}$ from \eqref{eq:arithemtic_part_defn}, which we discuss above for the weight $\frac{3}{2}$ Eisenstein series.
Our goal in this section is to bound a number of such bilinear combinations that we need for the analytic continuation in Section~\ref{sec:analytic_continuation_coupled_eisenstein}.

\begin{lem}\label{lem:arithmetic_bilinear_sing_reg_bounds}
Let $\e > 0$ be fixed, $m_1,m_2 \in \IN$ be odd, square-free, and $\mu_j \in \IZ/2m_j \IZ$ for~$j \in \{1,2\}$. 
Also let $k \in \IN$, $n \in \IZ_{\vc{m},\vc{\mu}}$, and $s \in \IC$.
Then for $\re (s) \geq -\frac{1}{4}$ we have, with $[k] := \{1,2,\ldots,k\}$,
\begin{align}
\label{eq:arithmetic_bilinear_sing_sing_bound}
&\sum_\pm \sum_{r \in \pm [k] + \musqm} \hspace{-0.5cm}
\pm \Lsing{m_1,\mu_1} (n-r;s) \Lsing{m_2,\mu_2} (r;s) 
\ll_{\vc{m},\e} 
\lp 1 + |n| \rp k^\e ,
\\
&\sum_\pm \sum_{r \in \pm [k] + \musqm} \hspace{-0.5cm}
\pm \lp 
\Lsing{m_1,\mu_1} (n-r;s) \Lreg{m_2,\mu_2} (r;s) 
+ \Lreg{m_1,\mu_1} (n-r;s) \Lsing{m_2,\mu_2} (r;s) 
\rp
\notag \\[-3ex] & \hspace{6cm}
\ll_{\vc{m},\e} 
\lp 1 + |\im (s)| \rp 
\lp 1 + |n| \rp
k^{\frac{1}{2}+\max \left\{ -\frac{2\re (s)}{3}, 0\right\} + \e}.
\label{eq:arithmetic_bilinear_reg_sing_bound}
\end{align}
\end{lem}
\begin{proof}
We first bound the combination with $\Lsing{} \Lsing{}$.
For $r \! \in \! \pm [k] + \musqmsmall$, 
we have $1+|r|^\e \ll_{\e} k^\e$ and $1+|n-r|^\e \ll_\e \lp 1+|n|^\e \rp k^\e$.
Then, by Lemma~\ref{lem:Lreg_Lsing_bound}, for $\re (s) \geq -\frac{1}{4}$, we find
\begin{equation*}
\sum_\pm \! \sum_{r \in \pm [k] + \musqm} \hspace{-0.7cm}
\pm \Lsing{m_1,\mu_1} (n-r;s) \Lsing{m_2,\mu_2} (r;s) 
\ll_{\vc{m},\e}
\lp 1+|n|^\e \rp k^{2\e} 
\sum_\pm \!
\sum_{r \in \pm [k] +\musqm} \hspace{-0.7cm}
\d_{-4m_1(n-r) = \square} \d_{-4m_2r = \square} .
\end{equation*}
If \smash{$\d_{-4m_1(n-r) = \square} \d_{-4m_2r = \square}$} is nonzero, there exist $\ell_1, \ell_2 \in \IN_0$ with $-4m_1(n-r) = \ell_1^2$ and $-4m_2r = \ell_2^2$. 
Then, because $\ell_2$ uniquely determines $r$, we estimate
\begin{equation*}
\sum_\pm \!
\sum_{r \in \pm [k] +\musqm} \hspace{-0.5cm}
\d_{-4m_1(n-r) = \square} \d_{-4m_2r = \square} 
\leq 
\# \left\{ \vc{\ell} \in \IZ^2 :\, -4m_1m_2n = m_2 \ell_1^2 + m_1 \ell_2^2 \right\}
\ll_{\vc{m}, \e} 1 + |n|^\e .
\end{equation*}
The last bound follows from the fact that the $\ell$-th Fourier coefficient of a weight one modular form $f$ on a congruence subgroup is $\ll_{f,\e} 1 + \ell^\e$ by \cite{DS}.
This yields \eqref{eq:arithmetic_bilinear_sing_sing_bound} after $\e \mapsto \frac{\e}{2}$ and relaxing the exponent of $|n|$ for the estimate $(1+|n|^\e) k^\e$, as we can assume $\e < 1$ without loss of generality.

To prove the second bound, we similarly use Lemma~\ref{lem:Lreg_Lsing_bound}, to obtain
\begin{multline*}
\sum_\pm  \sum_{r \in \pm [k] + \musqm} \hspace{-0.5cm}
\pm \lp 
\Lsing{m_1,\mu_1} (n-r;s) \Lreg{m_2,\mu_2} (r;s) 
+ \Lreg{m_1,\mu_1} (n-r;s) \Lsing{m_2,\mu_2} (r;s) 
\rp
\\
\ll_{\vc{m},\e} \!
\lp 1+ |\im (s)| \rp
\lp 1 + |n|^{\frac{1}{6} + 2\e} \rp 
k^{\max \left\{ -\frac{2\re (s)}{3}, 0\right\} + 2\e}
\sum_\pm \! \sum_{r \in \pm [k] + \musqm} \hspace{-0.6cm} 
\lp \d_{-4m_1(n-r) = \square}  + \d_{-4m_2r = \square} \rp .
\end{multline*}
Noting that in a sequence of $\ell \in \IN$ consecutive integers, there are $\leq 1 + \sqrt{\ell}$ squares, we can bound the remaining sums as $\ll_{\vc{m}} \sqrt{k}$.
Then \eqref{eq:arithmetic_bilinear_reg_sing_bound} follows as above by rescaling 
$\e \mapsto \frac{\e}{2}$ and relaxing the exponent of $|n|$ while assuming $\e < \frac{5}{6}$ without loss of generality.
\end{proof}

We next estimate the combination with $\Lreg{} \Lreg{}$, where the bound that follows from Lemma~\ref{lem:Lreg_Lsing_bound} turns out to be too weak in the $k$-aspect. 
To overcome this problem, we use convexity arguments to interpolate between 
the best possible (yet insufficient) bound at $\re (s) = -\frac{1}{4}$ following from Lemma~\ref{lem:Lreg_Lsing_bound} (via the Weyl bound \eqref{eq:Petrow_Young_bound}) 
and 
good estimates in the region of absolute convergence using cancellations between the summands.
We start with the latter step.

\begin{lem}\label{lem:arithmetic_bilinear_absolute_conv_k_bound}
Let $\e > 0$ be fixed, $m_1,m_2 \in \IN$, and $\mu_j \in \IZ/2m_j\IZ$ for $j \in \{1,2\}$. 
Also let $k \in \IN$, $n \in \IZ_{\vc{m},\vc{\mu}}$, and $s \in \IC$.
Then, for $\re (s) \geq \frac{1}{4} + \e$, we have
\begin{equation*}
\sum_\pm \sum_{r \in \pm [k] + \musqm} \hspace{-0.5cm}
\pm \calL_{m_1,\mu_1} (n-r;s) \calL_{m_2,\mu_2} (r;s) 
\ll_{\e}  1 .
\end{equation*}
\end{lem}
\begin{proof}
Denote the left-hand side in the lemma by $\ff_{\vc{m},\vc{\mu}} (n,k;s)$.
Using \eqref{eq:arithemtic_part_defn}, we rewrite it as
\begin{multline*}
\ff_{\vc{m},\vc{\mu}} (n,k;s) 
=
-2 \sum_{c_1,c_2 \geq 1} c_1^{-2s-\frac{3}{2}} c_2^{-2s-\frac{3}{2}}
\sum_{\substack{d_1 \pmod{c_1}^* \\ d_2 \pmod{c_2}^*}}  
\Psi_{m_1} (M_{c_1,d_1})_{0,\mu_1} \Psi_{m_2} (M_{c_2,d_2})_{0,\mu_2} 
\\[-0.4em]
\times
e^{2 \pi i n \frac{d_1}{c_1} + 2\pi i \musqmbig \lp \frac{d_2}{c_2} - \frac{d_1}{c_1} \rp }
\sum_{r=1}^k \sin  \lp 2\pi r \lp \frac{d_2}{c_2} - \frac{d_1}{c_1} \rp \rp .
\end{multline*}
Taking the sum over $r$ gives (using $\re (s) \geq \frac{1}{4} + \e$ and relaxing the $\gcd$ condition for $d_1,d_2$)
\begin{equation*}
| \ff_{\vc{m},\vc{\mu}} (n,k;s)  |
\leq 2 \sum_{c_1,c_2 \geq 1} c_1^{-2-2\e} c_2^{-2-2\e}
\sum_{\substack{d_1 \pmod{c_1} \\ d_2 \pmod{c_2}}} 
\frac{\d_{c_1 d_2 - c_2 d_1 \not\equiv 0 \pmod{c_1c_2}}}{
\left| \sin \lp \pi \frac{c_1 d_2 - c_2 d_1}{c_1 c_2}  \rp \right|} .
\vspace{-1ex}
\end{equation*}
Now define $\ell := \gcd (c_1,c_2)$, $\ccc_1 := \frac{c_1}{\ell}$, and $\ccc_2 := \frac{c_2}{\ell}$.
Note that for $d:= \ccc_1 d_2 - \ccc_2 d_1$, as $d_1, d_2$ runs modulo $\ccc_1 \ell$, resp.~$\ccc_2 \ell$, $d$ runs modulo $\ccc_1 \ccc_2 \ell$ exactly $\ell$ times.
Then
\begin{equation*}
\sum_{\substack{d_1 \pmod{c_1} \\ d_2 \pmod{c_2}}} 
\frac{\d_{c_1 d_2 - c_2 d_1 \not\equiv 0 \pmod{c_1c_2}}}{
\left| \sin \lp \pi \frac{c_1 d_2 - c_2 d_1}{c_1 c_2}  \rp \right|}
=
\ell \sum_{d=1}^{\ccc_1 \ccc_2 \ell-1} 
\frac{1}{\sin \lp \frac{\pi d}{\ccc_1 \ccc_2 \ell}  \rp}
\leq \ccc_1 \ccc_2 \ell^2 \log \lp \ccc_1 \ccc_2 \ell \rp 
\vspace{-1ex}
\end{equation*}
so that
\begin{equation*}
| \ff_{\vc{m},\vc{\mu}} (n,k;s)  |
\leq 2 \sum_{c_1,c_2 \geq 1} c_1^{-1-2\e} c_2^{-1-2\e}
\log \lp c_1 c_2 \rp
\ll_\e 1 .
\qedhere
\end{equation*}
\end{proof}

We now bound the combination with $\Lreg{} \Lreg{}$.

\begin{prop}\label{prop:arithmetic_bilinear_reg_reg_bound}
Let $\e > 0$ be fixed, $m_1,m_2 \!\in\! \IN$ be odd, square-free, and $\mu_j \!\in  \IZ/2m_j \IZ$ for~$j \!\in\!  \{1,2\}$. 
Also let $k \in \IN$, $n \in \IZ_{\vc{m},\vc{\mu}}$, and $s \in \IC$.
Then for $\re (s) \geq -\frac{1}{4}$, we have
\begin{equation*}
\sum_\pm \sum_{r \in \pm [k] + \musqm} \hspace{-0.5cm}
\pm \Lreg{m_1,\mu_1} (n-r;s) \Lreg{m_2,\mu_2} (r;s) 
\ll_{\vc{m},\e} 
\lp 1 + |\im (s)| \rp^2
\lp 1 + |n| \rp
k^{\frac{1}{2}+ \frac{5}{12} \max \left\{1-4\re (s), 0\right\} + \e} . 
\end{equation*}
\end{prop}
\begin{proof}
Denote the left-hand side of the claimed bound by $\hh_{\vc{m},\vc{\mu}} (n,k;s)$.
Using Lemma~\ref{lem:Lreg_Lsing_bound} as in the proof of Lemma~\ref{lem:arithmetic_bilinear_sing_reg_bounds} yields the weaker bound
\begin{equation}\label{eq:arithmetic_bilinear_reg_reg_bound_midstep1}
\hh_{\vc{m},\vc{\mu}} (n,k;s)
\ll_{\vc{m},\e}
\lp 1 + |\im (s)| \rp^2
\lp 1 + |n| \rp
k^{1 + \max \left\{ -\frac{4\re (s)}{3}, 0\right\}+ \e} 
\quad  \mbox{for } \re (s) \geq -\frac{1}{4} .
\end{equation}
To improve upon this, we use the decomposition \eqref{eq:Lreg_Lsing_decomp_definition} to write
\begin{align*}
&\hh_{\vc{m},\vc{\mu}} (n,k;s)
\!=\! 
2 \sqrt{m_1 m_2} \s_{-2s-1} (m_1) \s_{-2s-1} (m_2) \z^2 (4s+2)
\sum_\pm \!\! \sum_{r \in \pm [k] + \musqm}  \hspace{-0.7cm}
\pm \calL_{m_1,\mu_1} (n-r;s) \calL_{m_2,\mu_2} (r;s) 
\\[-1ex]
& \hspace{3cm}
- \frac{1}{4s} 
\sum_\pm \!\! \sum_{r \in \pm [k] + \musqm}   \hspace{-0.7cm}
\pm \lp \Lreg{m_1,\mu_1} (n-r;s) \Lsing{m_2,\mu_2} (r;s) 
+ \Lsing{m_1,\mu_1} (n-r;s) \Lreg{m_2,\mu_2} (r;s)  \rp
\\[-1ex]
& \hspace{4cm}
-\frac{1}{16s^2} 
\sum_\pm \!\! \sum_{r \in \pm [k] + \musqm}  \hspace{-0.7cm}
\pm \Lsing{m_1,\mu_1} (n-r;s) \Lsing{m_2,\mu_2} (r;s)  .
\end{align*}
Then employing Lemmas~\ref{lem:arithmetic_bilinear_sing_reg_bounds} and \ref{lem:arithmetic_bilinear_absolute_conv_k_bound} we obtain
(slightly weakening in the $s$-aspect)
\begin{equation}\label{eq:arithmetic_bilinear_reg_reg_bound_midstep2}
\hh_{\vc{m},\vc{\mu}} (n,k;s) \ll_{\vc{m},\e} 
\lp 1 + |\im (s)| \rp^2
\lp 1 + |n| \rp
k^{\frac{1}{2} + \e} 
\quad  \mbox{for } \re (s) \geq \frac{1}{4}+\e .
\end{equation}
For $-\frac{1}{4} \leq \re (s) \leq \frac{1}{4} + \e$, we use Theorem~\ref{thm:phragmen_lindelof} to interpolate from the bound \eqref{eq:arithmetic_bilinear_reg_reg_bound_midstep1} for $\re (s) \!=\! - \frac{1}{4}$ to \eqref{eq:arithmetic_bilinear_reg_reg_bound_midstep2} for $\re (s) \!=\! \frac{1}{4} + \e$ (noting $\hh_{\vc{m},\vc{\mu}} (n,k;s)  \ll_{n,k,\vc{m},\e} e^{|\im (s)|}$ in that range by \eqref{eq:arithmetic_bilinear_reg_reg_bound_midstep1}).
With \eqref{eq:arithmetic_bilinear_reg_reg_bound_midstep2}, this yields
the proposition (rescaling $\e$ as necessary).
\end{proof}

\section{Analytic Continuation of the Coupled Eisenstein Series}
\label{sec:analytic_continuation_coupled_eisenstein}

Using the results from Section~\ref{sec:analytic_atirhmetic_parts_coupled_eisenstein}, our goal now is the analytic continuation of 
$\calF_{\vc{m},\vc{\mu}} (\t, \bar{\t};s)$, defined in \eqref{eq:dimension_two_part_Eisenstein}, to a half-plane containing $s=0$. 
This leads to the proof of Theorem~\ref{thm:coupled_Eisenstein_depth_two_analytic_continuation}, given at the end of this section.
We start with the Fourier expansion of $\calF_{\vc{m},\vc{\mu}} (\t, \bar{\t};s)$ for $\Re (s) > \frac{1}{4}$.

\begin{prop}\label{prop:dimension_two_part_Eisenstein_fourier_expansion}
Let $m_1,m_2 \in \IN$, $\mu_j \in \IZ/2m_j\IZ$ for $j \in \{1,2\}$, $\t \in \IH$, and $s \in \IC$ with $\re (s) > \frac{1}{4}$. Then we have
\begin{equation*}
\calF_{\vc{m},\vc{\mu}} (\t, \bar{\t};s)
= s  \sum_{n \in \IZ_{\vc{m},\vc{\mu}}} e^{2 \pi i n \t} 
\psum_{k \in \IZ_{m_2,\mu_2}} 
\calK (v,n,k;s) \calL_{m_1,\mu_1} (n-k;s) \calL_{m_2,\mu_2} (k;s) .
\end{equation*}
\end{prop}
\begin{proof}
We start with \eqref{eq:dimension_two_part_Eisenstein} and
using \eqref{eq:weil_T_S_explicit} and \eqref{eq:weil_projective_representation}, we rewrite
\begin{equation}\label{eq:dimension_two_part_Eisenstein_fourier_expansion_step1}
\calF_{\vc{m},\vc{\mu}} (\t, \bar{\t};s)
\hspace{-0.05cm}
=
\hspace{-0.05cm}
\pi is v^{2s} \hspace{-0.25cm} 
\sum_{c_1,c_2 \geq 1} \hspace{-0.2cm} c_1^{-2s-\frac{3}{2}} c_2^{-2s-\frac{3}{2}}
\hspace{-0.33cm}
\sum_{\substack{d_1 \pmod{c_1}^* \\ d_2 \pmod{c_2}^*}} \hspace{-0.5cm}
I_{\vc{m},\vc{\mu},\vc{c},\vc{d}} (\t, \bar{\t}; s) 
\hspace{-0.25cm} \prod_{j \in \{1,2\}} \hspace{-0.25cm}
\Psi_{m_j} \lp M_{c_j,d_j} \rp_{0,\mu_j} 
\!
e^{-2\pi i \frac{d_j}{c_j} \frac{\mu_j^2}{4m_j}},
\end{equation}
where
\begin{equation*}
I_{\vc{m},\vc{\mu},\vc{c},\vc{d}} (\t, \bar{\t}; s) 
:=
\sum_{\vc{k} \in \IZ^2 + \lp \frac{d_1}{c_1},\frac{d_2}{c_2} \rp}
\frac{
\sgn (k_2-k_1)
e^{2 \pi i k_1 \frac{\mu_1^2}{4m_1} + 2 \pi i k_2 \frac{\mu_2^2}{4m_2}}
}{
|\t+k_1|^{2s} (\t+k_1)^{\frac{3}{2}}
|\t+k_2|^{2s} (\t+k_2)^{\frac{3}{2}}
} .
\end{equation*}
Letting $k_2 \mapsto k_2+k_1$ and applying Poisson summation on the sum in $k_1$ then yields
\begin{equation}\label{eq:dimension_two_part_Eisenstein_fourier_expansion_step2}
I_{\vc{m},\vc{\mu},\vc{c},\vc{d}} (\t, \bar{\t}; s) = 
i v^{-4s-2}
e^{2 \pi i \frac{d_1}{c_1} \lp \frac{\mu_1^2}{4m_1}+\frac{\mu_2^2}{4m_2} \rp}
\hspace{-0.5cm}
\sum_{k_2 \in \IZ + \frac{d_2}{c_2} - \frac{d_1}{c_1}} \hspace{-0.6cm}
\sgn (k_2) e^{2 \pi i k_2 \frac{\mu_2^2}{4m_2}} \!\!
\sum_{n \in \IZ_{\vc{m},\vc{\mu}}} \hspace{-0.3cm}
\calN_{n,k_2,v} (s) e^{2 \pi i n \lp u + \frac{d_1}{c_1} \rp},
\end{equation}
where
\begin{equation}\label{eq:calN_defintion_step}
\calN_{n,k_2,v} (s) := 
\int_\IR \frac{e^{-2\pi i n v x}}{
\lp 1-ix \rp^{\frac{3}{2}+s} \lp 1+ix \rp^s
\lp 1-i \lp x + \frac{k_2}{v} \rp \rp^{\frac{3}{2}+s} \lp 1+i \lp x + \frac{k_2}{v} \rp  \rp^s
}
dx  .
\end{equation}
We shift the integration path for $\calN_{n,k_2,v} (s)$ by $-\frac{i}{2}$ if $n>0$ and by $\frac{i}{2}$ if $n \leq 0$.
We then note that $(1+a^2) (1+b^2) \geq \frac{1}{2} (1+(a-b)^2)$ for any $a,b \in \IR$ to find
\begin{equation*}
\begin{rcases}
\vphantom{1^{1^{1^1}}} 
\left| \frac{1}{2} \pm i x \right|
\left|  \frac{1}{2} \pm i \lp x + \frac{k_2}{v} \rp \right|
\\
\left| \frac{3}{2} \pm i x \right|
\left|  \frac{3}{2} \pm i \lp x + \frac{k_2}{v} \rp \right|
\end{rcases}
\geq 
\frac{1}{8} \lp 1 + x^2 \rp^{\frac{1}{4}} 
\lp 1 + \frac{k_2^2}{v^2} \rp^{\frac{1}{4}} .
\end{equation*}
This yields the bound
\begin{equation*}
\calN_{n,k_2,v} (s) \ll_s \lp 1 + \frac{k_2^2}{v^2} \rp^{-\frac{3}{8}- \frac{\re (s)}{2}} e^{-\pi |n| v}.
\end{equation*}
Since $\re (s) > \frac{1}{4}$, this shows that the sum over $(k_2,n)$ in \eqref{eq:dimension_two_part_Eisenstein_fourier_expansion_step2} is absolutely convergent and we can interchange the two sums.
Using $(1+a^2) (1+b^2) \geq \frac{1}{2} (1+(a-b)^2)$, we can similarly interchange the sum over $k_2$ with the integral over $x$ in \eqref{eq:calN_defintion_step}.
The resulting sum over $k_2$ obeys the requirements of Theorem~\ref{thm:poisson_summation} and thus using Poisson summation there yields
\begin{multline}\label{eq:dimension_two_part_Eisenstein_fourier_expansion_step3}
I_{\vc{m},\vc{\mu},\vc{c},\vc{d}} (\t, \bar{\t}; s) = 
i v^{-4s-1}
e^{2 \pi i  \lp \frac{d_1}{c_1} \frac{\mu_1^2}{4m_1}+\frac{d_2}{c_2} \frac{\mu_2^2}{4m_2} \rp}
\sum_{n \in \IZ_{\vc{m},\vc{\mu}}} 
e^{2 \pi i n \lp u + \frac{d_1}{c_1} \rp}
\\ \times
\int_\IR 
\frac{e^{-2\pi i n v x}}{\lp 1-ix \rp^{\frac{3}{2}} \lp 1+x^2 \rp^s} 
\ \ 
\psum_{k \in \IZ_{m_2,\mu_2}}
e^{2 \pi i k \lp \frac{d_2}{c_2}-\frac{d_1}{c_1} \rp}
\int_\IR 
\frac{\sgn (y-x) e^{-2\pi i k v (y-x)}}{\lp 1-iy \rp^{\frac{3}{2}} \lp 1+y^2 \rp^s}  
dy dx. 
\end{multline}
Using the decomposition \eqref{eq:details_analytic_part_calK_properties_intParts}
and the bound \eqref{eq:integral_phi_kappa_bound}, we can remove the symmetric summation for the sum over $k$ and interchange the integral in $x$ and the sum in $k$ for the term with $\phi_{\frac{3}{2}}$. This yields
\begin{multline*}
\pi i v^{-2s-1} e^{2 \pi n v} \int_\IR 
\frac{e^{-2\pi i n v x}}{\lp 1-ix \rp^{\frac{3}{2}} \lp 1+x^2 \rp^s} 
\ 
\psum_{k \in \IZ_{m_2,\mu_2}} 
e^{2 \pi i k \lp \frac{d_2}{c_2}-\frac{d_1}{c_1} \rp}
\int_\IR 
\frac{\sgn (y-x) e^{-2\pi i k v (y-x)}}{\lp 1-iy \rp^{\frac{3}{2}} \lp 1+y^2 \rp^s}  
dy dx
\\
=
\calI_3 (v,n;2s) \psum_{\substack{k \in \IZ_{m_2,\mu_2} \\ k \not\in [0,1)}} 
\frac{e^{2 \pi i k \lp \frac{d_2}{c_2}-\frac{d_1}{c_1} \rp}}{\lfloor k \rfloor}
+
\sum_{k \in \IZ_{m_2,\mu_2}} \!\!\!\!
e^{2 \pi i k \lp \frac{d_2}{c_2}-\frac{d_1}{c_1} \rp} 
\lp \calK (v,n,k;s) - \frac{\d_{k \not\in [0,1)}}{\lfloor k \rfloor} \calI_3 (v,n;2s)  \rp .
\end{multline*}
We next plug this into \eqref{eq:dimension_two_part_Eisenstein_fourier_expansion_step3}, which we then use in \eqref{eq:dimension_two_part_Eisenstein_fourier_expansion_step1}.
By Proposition \ref{prop:analytic_part_calK_properties}, the sum over $(\vc{c}, \vc{d}, n, k)$ for the contribution of the second term is absolutely convergent because $\re (s) > \frac{1}{4}$.
So we can split off this term and interchange the sums to obtain
\begin{equation*}%\label{eq:calF_calF1_calF2_decomposition}
\calF_{\vc{m},\vc{\mu}} (\t, \bar{\t};s) =
\calF^{[1]}_{\vc{m},\vc{\mu}} (\t, \bar{\t};s) +
\calF^{[2]}_{\vc{m},\vc{\mu}} (\t, \bar{\t};s),
\end{equation*}
where
\begin{align}\label{eq:calF1_definition}
\calF^{[1]}_{\vc{m},\vc{\mu}} (\t, \bar{\t};s) 
&:= 
i s 
\sum_{c_1,c_2 \geq 1} 
c_1^{-2s-\frac{3}{2}} c_2^{-2s-\frac{3}{2}}
\sum_{\substack{d_1 \pmod{c_1}^* \\ d_2 \pmod{c_2}^*}} 
\Psi_{m_1} (M_{c_1,d_1})_{0,\mu_1} 
\Psi_{m_2} (M_{c_2,d_2})_{0,\mu_2}
\\ & \hspace{5cm} \times
\sum_{n \in \IZ_{\vc{m},\vc{\mu}}} 
\calI_3 (v,n;2s) e^{2 \pi i n \lp \t + \frac{d_1}{c_1} \rp}
\psum_{\substack{k \in \IZ_{m_2,\mu_2} \\ k \not\in [0,1)}} 
\frac{e^{2 \pi i k \lp \frac{d_2}{c_2}-\frac{d_1}{c_1} \rp}}{\lfloor k \rfloor},
\notag\\
\calF^{[2]}_{\vc{m},\vc{\mu}} (\t, \bar{\t};s)
&:= s \!\! \sum_{n \in \IZ_{\vc{m},\vc{\mu}}} \!\!\!\! e^{2 \pi i n \t} 
\!\!\! \sum_{k \in \IZ_{m_2,\mu_2}} \!\!\!\!
\lp \calK (v,n,k;s) 
-  \frac{\d_{k \not\in [0,1)}}{\lfloor k \rfloor} \calI_3 (v,n;2s) \! \rp \!
\calL_{m_1,\mu_1} (n-k;s) \calL_{m_2,\mu_2} (k;s).
\notag
\end{align}
Now we focus on $\calF^{[1]}_{\vc{m},\vc{\mu}} (\t, \bar{\t};s)$. Using summation by parts for the sum over $k$, we obtain
\begin{equation*}
\psum_{\substack{k \in \IZ_{m_2,\mu_2} \\ k \not\in [0,1)}} 
\frac{e^{2 \pi i k \lp \frac{d_2}{c_2}-\frac{d_1}{c_1} \rp}}{\lfloor k \rfloor}
=
2 i e^{2\pi i \musqmbig \lp \frac{d_2}{c_2}-\frac{d_1}{c_1} \rp}
\sum_{k \geq 1} \frac{1}{k(k+1)} 
\sum_{r=1}^k \sin \lp 2 \pi r \lp \frac{d_2}{c_2}-\frac{d_1}{c_1} \rp \rp.
\end{equation*}
Plugging this into \eqref{eq:calF1_definition} and using Lemma~\ref{lem:analytic_part_calI_properties}~(1), to get the bound
\begin{multline*}
\Bigg| c_1^{-2s-\frac{3}{2}} c_2^{-2s-\frac{3}{2}} 
\calI_3 (v,n;2s) e^{2 \pi i n \lp \t + \frac{d_1}{c_1} \rp}
\frac{1}{k(k+1)} 
\sum_{r=1}^k \sin \lp 2 \pi r \lp \frac{d_2}{c_2}-\frac{d_1}{c_1} \rp \rp
\Bigg|
\\
\ll_{s,v}
c_1^{-2\re(s)-\frac{3}{2}} c_2^{-2\re(s)-\frac{3}{2}} 
\frac{e^{-\pi |n| v}}{k(k+1)} 
\frac{\d_{c_1 d_2 - c_2 d_1 \not\equiv 0 \pmod{c_1c_2}}}{
\left| \sin \lp \pi \frac{c_1 d_2 - c_2 d_1}{c_1 c_2}  \rp \right|},
\end{multline*}
we find that the resulting joint sum over $(\vc{c},\vc{d},n,k)$ is absolutely convergent for $\re (s) > \frac{1}{4}$.
In particular, we can interchange the sums over $\vc{c},\vc{d},n,k$, to rewrite
\begin{multline*}
\calF^{[1]}_{\vc{m},\vc{\mu}} (\t, \bar{\t};s) = 
-2 s 
\sum_{n \in \IZ_{\vc{m},\vc{\mu}}}  
\calI_3 (v,n;2s) e^{2 \pi i n \t}
\sum_{k \geq 1} \frac{1}{k(k+1)} 
\sum_{c_1,c_2 \geq 1} 
c_1^{-2s-\frac{3}{2}} c_2^{-2s-\frac{3}{2}}
\\ \times
\sum_{\substack{d_1 \pmod{c_1}^* \\ d_2 \pmod{c_2}^*}} \hspace{-0.4cm}
\Psi_{m_1} (M_{c_1,d_1})_{0,\mu_1} 
\Psi_{m_2} (M_{c_2,d_2})_{0,\mu_2}
e^{2 \pi i n \frac{d_1}{c_1}}
e^{2\pi i \musqmbig \lp \frac{d_2}{c_2}-\frac{d_1}{c_1} \rp}
\sum_{r=1}^k \sin \lp 2 \pi r \lp \frac{d_2}{c_2}-\frac{d_1}{c_1} \rp \rp .
\end{multline*}
Commuting the finite sum over $r$ with the sums over $\vc{c},\vc{d}$ and using \eqref{eq:arithemtic_part_defn}, we evaluate
\begin{equation*}
\calF^{[1]}_{\vc{m},\vc{\mu}} (\t, \bar{\t};s)
=  s \sum_{n \in \IZ_{\vc{m},\vc{\mu}}} \!\!\!\! 
\calI_3 (v,n;2s) e^{2 \pi i n \t}
\sum_{k \geq 1} \frac{1}{k(k+1)}
\sum_\pm \sum_{r \in \pm [k] + \musqm} \hspace{-0.4cm}
\pm \calL_{m_1,\mu_1} (n-r;s) \calL_{m_2,\mu_2} (r;s) .
\end{equation*}
Reversing the summation by parts step above, we use the bound in Lemma \ref{lem:arithmetic_bilinear_absolute_conv_k_bound}, to get
\begin{equation*}
\psum_{k \in \IZ_{m_2,\mu_2}} 
\!\!\! \frac{\d_{k \not\in [0,1)}}{\lfloor k \rfloor} 
\calL_{m_1,\mu_1} (n-k;s) \calL_{m_2,\mu_2} (k;s)
\!=\!
\sum_{k \geq 1} \! \frac{1}{k(k+1)} \!
\sum_\pm \! \sum_{r \in \pm [k] + \musqm} \hspace{-0.75cm}
\pm \calL_{m_1,\mu_1} (n-r;s) \calL_{m_2,\mu_2} (r;s) .
\end{equation*}
Plugging this into $\calF^{[1]}_{\vc{m},\vc{\mu}} (\t, \bar{\t};s)$ and then combining with $\calF^{[2]}_{\vc{m},\vc{\mu}} (\t, \bar{\t};s)$ gives the proposition.
\end{proof}

Next we give the analytic continuation of $\calF_{\vc{m},\vc{\mu}} (\t, \bar{\t};s)$.

\begin{prop}\label{prop:dimension_two_part_Eisenstein_analytic_continuation}
Let $m_1,m_2 \!\in \IN$ be odd, square-free, $\mu_j \!\in \IZ/2m_j\IZ$ for $j \!\in\! \{1,2\}$, $\t \in \IH$, and $s \in \IC$.
Then $s \mapsto\! \calF_{\vc{m},\vc{\mu}} (\t, \bar{\t};s)$ analytically continues to $\re (s)\! > \! - \frac{1}{20}$ and $\calF_{\vc{m},\vc{\mu}} (\t, \bar{\t}) := \calF_{\vc{m},\vc{\mu}} (\t, \bar{\t};0)$ equals
\begin{multline*}
\calF_{\vc{m},\vc{\mu}} (\t, \bar{\t}) = 
H_{\vc{m},\vc{\mu}} (\t) 
+ C_{m_1} C_{m_2} 
\int_{-\bar{\t}}^{i \infty} 
\frac{\TH_{m_2, \mu_2} (w_2)}{(-i(w_2+\t))^{\frac{3}{2}}} 
\int_{w_2}^{i \infty} 
\frac{\TH_{m_1, \mu_1} (w_1)}{(-i(w_1+\t))^{\frac{3}{2}}} 
d w_1 dw_2 
\\
- C_{m_1} C_{m_2} 
\int_{-\bar{\t}}^{i \infty} 
\frac{\TH_{m_1, \mu_1} (w_1)}{(-i(w_1+\t))^{\frac{3}{2}}} 
\int_{w_1}^{i \infty} 
\frac{\TH_{m_2, \mu_2} (w_2)}{(-i(w_2+\t))^{\frac{3}{2}}} 
d w_2 dw_1 .
\end{multline*}
\end{prop}
\begin{proof}
First we assume that $\re (s) > \frac{1}{4}$. Using \eqref{eq:Lreg_Lsing_decomp_definition} and Proposition \ref{prop:dimension_two_part_Eisenstein_fourier_expansion}, while defining
\begin{equation*}%\label{eq:gamma_m_definition}
\g_{\vc{m}} (s) := \frac{1}{2 \sqrt{m_1 m_2} \s_{-2s-1} (m_1) \s_{-2s-1} (m_2) \z^2 (4s+2)},
\end{equation*}
we rewrite
\begin{multline*}
\calF_{\vc{m},\vc{\mu}} (\t, \bar{\t};s)
= s  \g_{\vc{m}} (s) 
\sum_{n \in \IZ_{\vc{m},\vc{\mu}}} \hspace{-0.3cm} e^{2 \pi i n \t} 
\psum_{k \in \IZ_{m_2,\mu_2}} \hspace{-0.2cm}
\calK (v,n,k;s) 
\lp \Lreg{m_1,\mu_1} (n-k;s) + \frac{1}{4s} \Lsing{m_1,\mu_1} (n-k;s) \rp
\\[-2ex] \times
\lp \Lreg{m_2,\mu_2} (k;s) + \frac{1}{4s} \Lsing{m_2,\mu_2} (k;s) \rp .
\end{multline*}
Using the bounds in Lemma~\ref{lem:analytic_part_calI_properties}~(1), Lemma~\ref{lem:Lreg_Lsing_bound}, and Proposition~\ref{prop:analytic_part_calK_properties} for $\re (s) \!>\! \frac{1}{4}$, we  split this as
\begin{align}
\notag
&\calF_{\vc{m},\vc{\mu}} (\t, \bar{\t};s)
=
s  \g_{\vc{m}} (s) 
\sum_{n \in \IZ_{\vc{m},\vc{\mu}}}  \sum_{k \geq 1}
\frac{\calI_3 (v,n;2s)e^{2 \pi i n \t}}{k(k+1)}
\sum_\pm \!\! \sum_{r \in \pm [k] + \musqm} \hspace{-0.5cm}
\pm \Lreg{m_1,\mu_1} (n-r;s) \Lreg{m_2,\mu_2} (r;s)
\notag \\[-0.75ex]
& \quad +
s  \g_{\vc{m}} (s) 
\sum_{n \in \IZ_{\vc{m},\vc{\mu}}} \sum_{k \in \IZ_{m_2,\mu_2}}  \hspace{-0.3cm} 
\lp \calK (v,n,k;s) 
- \frac{\d_{k \not\in [0,1)}}{\lfloor k \rfloor}  \calI_3 (v,n;2s) \rp \!
\Lreg{m_1,\mu_1} (n-k;s) \Lreg{m_2,\mu_2} (k;s)
e^{2 \pi i n \t}
\notag \\[-0.75ex] \notag 
& \quad  +
\frac{\g_{\vc{m}} (s)}{4}
\sum_{n \in \IZ_{\vc{m},\vc{\mu}}} \sum_{k \in \IZ_{m_2,\mu_2}}  \hspace{-0.3cm} 
\calK (v,n,k;s) 
\Lreg{m_1,\mu_1} (n-k;s) \Lsing{m_2,\mu_2} (k;s)
e^{2 \pi i n \t}
\\[-0.75ex] \notag 
& \quad +
\frac{\g_{\vc{m}} (s)}{4}
\sum_{n \in \IZ_{\vc{m},\vc{\mu}}} \sum_{k \in \IZ_{m_2,\mu_2}}  \hspace{-0.3cm} 
\calK (v,n,k;s) 
\Lsing{m_1,\mu_1} (n-k;s) \Lreg{m_2,\mu_2} (k;s)
e^{2 \pi i n \t}
\\[-0.75ex]
& \quad  +
\frac{\g_{\vc{m}} (s)}{16}
\sum_{n \in \IZ_{\vc{m},\vc{\mu}}}
\sum_{\substack{k \in \IZ_{m_2,\mu_2} \\ n \leq k \leq 0}} \hspace{-0.2cm}
\frac{\calK (v,n,k;s)}{s}
\Lsing{m_1,\mu_1} (n-k;s) \Lsing{m_2,\mu_2} (k;s) e^{2 \pi i n \t}.
\label{eq:calF_rewriting_to_be_continued}
\end{align}
Here we apply summation by parts while noting Proposition \ref{prop:arithmetic_bilinear_reg_reg_bound} for the first term, note the absolute convergence of the last four terms to remove the symmetric summation, and 
use that the contributions from $\mathcal{L}^{\mathrm{sing}}\mathcal{L}^{\mathrm{sing}}$ are only nonzero for $k \leq 0$ and $n-k \leq 0$ by \eqref{eq:Lsing_definition} for the last term.

Now note that $s$ and $\g_{\vc{m}} (s)$ along with all the summands in these sums in $(n,k)$ are holomorphic for $\re (s) > -\frac{1}{20}$.
We next prove that these sums in $(n,k)$ are absolutely and locally uniformly convergent for $\re (s) > -\frac{1}{20}$, which in turn proves the claimed analytic continuation. 
Thus we fix an arbitrary $0 < \e < \frac{1}{20}$ and assume
\begin{equation*}
-\frac{1}{20} + \e \leq \re (s) \leq \frac{1}{\e}
\andd
| \im (s) | \leq \frac{1}{\e} .
\end{equation*}
In this range, the summand of the first sum  in \eqref{eq:calF_rewriting_to_be_continued} is bounded for $n \in \IZ_{\vc{m},\vc{\mu}}$ and $k \in \IN$ as
\begin{equation*}
\frac{\calI_3 (v,n;2s) e^{2 \pi i n \t}}{k(k+1)}
\sum_\pm \!\! \sum_{r \in \pm [k] + \musqm} \hspace{-0.6cm}
\pm \Lreg{m_1,\mu_1} (n-r;s) \Lreg{m_2,\mu_2} (r;s)
\ll_{\vc{m},v,\e} 
\lp 1 + |n| \rp 
e^{-\pi |n| v}
k^{-1-\frac{2\e}{3}} 
\end{equation*}
using Lemma~\ref{lem:analytic_part_calI_properties}~(1) and Proposition~\ref{prop:arithmetic_bilinear_reg_reg_bound}. This uniform bound in $s$ is summable over $n \in \IZ_{\vc{m},\vc{\mu}}$ and $k \in \IN$  and the first sum converges to a holomorphic function for $\re (s) > -\frac{1}{20}$.
For the fifth sum in \eqref{eq:calF_rewriting_to_be_continued}, we estimate for $n \in \IZ_{\vc{m},\vc{\mu}}$ and $k \in \IZ_{m_2,\mu_2}$ with $k,n-k \leq 0$ as
\begin{equation*}
\frac{\calK (v,n,k;s)}{s} 
\Lsing{m_1,\mu_1} (n-k;s) \Lsing{m_2,\mu_2} (k;s)
e^{2 \pi i n \t}
\hspace{-0.1cm} \ll_{\vc{m},v,\e} \hspace{-0.07cm}
e^{-\pi |n| v} 
\lp 1 \!+\! |n-k| \!+\! |k| \rp \! \d_{-4m_1(n-k)=\square}  \d_{-4m_2k=\square}
\end{equation*}
using Proposition \ref{prop:analytic_part_calK_properties_negative} and Lemma \ref{lem:Lreg_Lsing_bound} (while weakening the exponents of $|n-k|$ and $|k|$).
This bound is similarly summable for the relevant range of $(n,k)$.
The remaining three sums in \eqref{eq:calF_rewriting_to_be_continued} are treated analogously with Lemma~\ref{lem:analytic_part_calI_properties}~(1), Lemma~\ref{lem:Lreg_Lsing_bound}, and Proposition~\ref{prop:analytic_part_calK_properties}, as above.

With the analytic continuation proved, we compute $\calF_{\vc{m},\vc{\mu}} (\t, \bar{\t})$ by setting $s=0$ in \eqref{eq:calF_rewriting_to_be_continued}, to find
\begin{align*}
\notag
\calF_{\vc{m},\vc{\mu}} (\t, \bar{\t})
&=
\frac{\g_{\vc{m}} (0)}{4}
\sum_{n \in \IZ_{\vc{m},\vc{\mu}}^{>0}} \hspace{-0.2cm}
e^{2 \pi i n \t}
\sum_{k \in \IZ_{m_2,\mu_2}^{\leqslant 0}} \hspace{-0.2cm}
\calK (v,n,k;0) 
\Lreg{m_1,\mu_1} (n-k;0) \Lsing{m_2,\mu_2} (k;0)
\\ \notag
& \quad +
\frac{\g_{\vc{m}} (0)}{4}
\sum_{n \in \IZ_{\vc{m},\vc{\mu}}^{>0}} \hspace{-0.2cm}
e^{2 \pi i n \t}
\sum_{k \in \IZ_{m_1,\mu_1}^{\leqslant 0}} \hspace{-0.2cm}
\calK (v,n,n-k;0) 
\Lreg{m_2,\mu_2} (n-k;0) \Lsing{m_1,\mu_1} (k;0)
\\ 
& \quad  +
\frac{\g_{\vc{m}} (0)}{16}
\sum_{n \in \IZ_{m_1,\mu_1}^{\leqslant 0}} 
\sum_{k \in \IZ_{m_2,\mu_2}^{\leqslant 0}} \hspace{-0.3cm}
\calK' (v,n+k,k;0) \Lsing{m_1,\mu_1} (n;0) \Lsing{m_2,\mu_2} (k;0) 
e^{2 \pi i (n+k) \t},
\end{align*}
where we let $k \mapsto n-k$ for the second line and $n \mapsto n+k$ for the third line, while noting that $\calK (v,n,k;0) = 0$ for $n \leq 0$ by  Proposition~\ref{prop:analytic_part_calK_value_at_zero} and $\Lsing{m_j,\mu_j} (r;0) = 0$ for $r>0$ and $j \!\in\! \{1,2\}$ by \eqref{eq:Lsing_definition}.
Using the expressions
\begin{equation*}
\frac{\g_{\vc{m}} (0)}{4}
= \frac{9 \sqrt{m_1 m_2}}{2\pi^4 \s(m_1) \s(m_2)} 
\andd
\frac{\g_{\vc{m}} (0)}{16}
=
\frac{9 \sqrt{m_1 m_2}}{8 \pi^4 \s(m_1) \s(m_2)}
= - \frac{C_{m_1} C_{m_2}}{4 \pi^2}
\end{equation*}
along with \eqref{eq:calK_3_2_values_at_zero} for $\calK (v,n,k;0)$ and $\calK (v,n,n-k;0)$, Proposition \ref{prop:analytic_part_calK_properties_negative} for $\calK' (v,n+k,k;0)$, and Lemma~\ref{lem:arithmetic_part_at_origin} for the values of $\mathcal{L}^{\mathrm{reg/sing}}$, we find the claimed expression for $\calF_{\vc{m},\vc{\mu}} (\t, \bar{\t})$ in the proposition. Here we recall the Fourier coefficients $\b_{\vc{m},\vc{\mu}} (n)$ of $H_{\vc{m},\vc{\mu}}$ given in \eqref{eq:depth_two_holomorphic_fourier_coef_definition}.
\end{proof}

We now have everything needed to prove Theorem~\ref{thm:coupled_Eisenstein_depth_two_analytic_continuation}.

\begin{proof}[Proof of Theorem~\ref{thm:coupled_Eisenstein_depth_two_analytic_continuation}]
Starting with the rewriting of $\wh{H}_{\vc{m},\vc{\mu}} (\t, \bar{\t})$ in \eqref{eq:hatH_rewriting2_g_f_F} for $\re(s) > \frac{1}{4}$ and recalling the continuation of $g_{m_j,\mu_j}$, \smash{$\wh{f}_{m_j,\mu_j}$}, and $\calF_{\vc{m},\vc{\mu}}$ to $\re (s) > - \frac{1}{20}$ from Propositions
\ref{prop:depth_one_Eichler_eisenstein_analytic_continuation},
\ref{prop:depth_one_eisenstein_analytic_continuation}, and~\ref{prop:dimension_two_part_Eisenstein_analytic_continuation}, respectively, the analytic continuation of $\wh{H}_{\vc{m},\vc{\mu}} (\t, \bar{\t};s)$ to $\re (s) > - \frac{1}{20}$ immediately follows.
These results also give the value at $s=0$ as claimed.
The derivative in $\bar{\t}$ immediately follows and the modular transformation follows from the analytic continuation of \eqref{eq:whH_modular_transformation} to $s=0$. 
\end{proof}

\section{An Example from Vafa--Witten invariants}
\label{sec:vafa_witten_example}

In this section we prove Corollary~\ref{cor:h_3_mu_coupled_eisenstein_expression} and relate our discussion to the depth two mock modular form $h_{3,\mu}$ coming from Vafa--Witten invariants.
Recall from \eqref{eq:h_3_0_1_fourier_expansion_vafa_witten}, \eqref{eq:depth_two_example}, and \eqref{eq:depth_two_example_modular_transformation} that $h_{3,\mu}$ with $\mu\!\in\!\IZ/3\IZ$ is a depth two mock modular form of weight $3$ and multiplier $\Psi^*_{\! A_2}$.
Using Proposition~\ref{prop:depth_one_eisenstein_analytic_continuation} with \eqref{eq:Weil_three_decomposition} and \eqref{eq:hatf_modular_transformation}, we may also form a trivial depth two mock modular form with the same weight and multiplier, combining multiples of two depth one mock modular forms as
\begin{equation*}
\sum_{\a \in \IZ/2\IZ} f_{1,\a} (\t) f_{3,2\mu+3\a} (\t) 
\quad \mbox{for } \mu \in \IZ/3\IZ.
\end{equation*}
Here we recall the details on $f_{1,\a}$ given in \eqref{eq:f1_fourier_coef_hurwitz} and \eqref{eq:f1_modular_completion}.
For $f_{3,\ell}$ with $\ell \in \IZ/6\IZ$, we specialize Proposition~\ref{prop:depth_one_eisenstein_analytic_continuation} to $m=3$, obtaining its Fourier expansion as
\begin{align*}%\label{eq:f3_fourier_coef_hurwitz}
f_{3,0} (\t) &= -3 \sum_{n\geq 0} \lp H(12n) + 3 \d_{3 \mid 4n} H\lp \frac{4n}{3} \rp \rp q^n,
\\[-0.75ex] \notag
f_{3,\pm 1} (\t) &= -3 \sum_{n\geq 0}  H(12n+11) q^{n+\frac{11}{12}},
\hspace{1.2cm}
f_{3,\pm 2} (\t) = -3 \sum_{n\geq 0}  H(12n+8) q^{n+\frac{2}{3}},
\\[-0.75ex] \notag
f_{3,3} (\t) &= -3 \sum_{n\geq 0} \lp H(12n+3) + 3 \d_{3 \mid 4n+1} H \lp \frac{4n+1}{3} \rp \rp q^{n+\frac{1}{4}}.
\end{align*}
The leading Fourier coefficients are
\begin{equation*}
\mat{
f_{3,0} (\t) \\
f_{3,\pm 1} (\t) \\
f_{3,\pm 2} (\t) \\
f_{3,3} (\t)
}
=
\mat{
1 - 4q - 6q^2 - 12q^3 - 10q^4 - 12q^5 - 18q^6 - 12q^7 - 18q^8 + \ldots \\
-3q^{\frac{11}{12}} - 9q^{\frac{23}{12}} - 6q^{\frac{35}{12}} - 15q^{\frac{47}{12}} - 9q^{\frac{59}{12}} - 21q^{\frac{71}{12}} - 9q^{\frac{83}{12}} + \ldots \\
-3q^{\frac23} - 6q^{\frac53} - 9q^{\frac83} - 12q^{\frac{11}3} - 12q^{\frac{14}3} - 12q^{\frac{17}3} - 18q^{\frac{20}3} - 18q^{\frac{23}3} + \ldots \\
-q^{\frac14} - 6q^{\frac54} - 7q^{\frac94} - 12q^{\frac{13}4} - 6q^{\frac{17}4} - 
 24q^{\frac{21}4} - 7q^{\frac{25}4} - 18q^{\frac{29}4} + \ldots
} .
\end{equation*}
Noting $\TH_{3, \ell} = 2 \vth_{3, \ell}$, its modular completion is 
\begin{equation*}
\wh{f}_{3,\ell} (\t, \bar{\tau}) = f_{3,\ell} (\t)
+ \frac{3i\sqrt{3}}{2 \pi \sqrt{2}} \int_{-\bar{\t}}^{i \infty} 
\frac{\vth_{3, \ell} (w)}{(-i(w+\t))^{\frac{3}{2}}} dw .
\end{equation*}
Together with \eqref{eq:f1_fourier_coef_hurwitz_examples}, this gives
\begin{equation}\label{eq:f1_f3_leading_fourier_coefs}
\Bigg( \displaystyle \sum_{\a \in \{0,1\}} f_{1,\a} (\t)  f_{3,2\mu+3\a} (\t) \Bigg)_{\! \mu \in \{0,\pm 1\}}
\\
=
\mat{
1 - 6 q + 42 q^2 + 168 q^3 + 408 q^4 + 744 q^5 + \ldots \\
-3 q^{\frac23} + 24 q^{\frac53} + 135 q^{\frac83} + 330 q^{\frac{11}3} + 630 q^{\frac{14}3}  + \ldots 
}  .
\end{equation}

Theorem \ref{thm:coupled_Eisenstein_depth_two_analytic_continuation} allows us to go beyond such trivial depth two mock modular forms, using coupled Eisenstein series. Specifically, for $\vc{m} = (1,3)$ we obtain a depth two mock modular form $H_{(1,3), (\a,2\mu+3\g)}$ with $\a,\g \in \IZ/2\IZ$, $\mu \in \IZ/3\IZ$, whose modular completion satisfies (also by~\eqref{eq:Weil_three_decomposition})
\begin{equation*}
%\label{eq:wh_H_example_depth_two_modular_transformation}
\wh{H}_{(1,3), (\a,2\mu+3\g)} 
\lp \frac{a\t+b}{c \t+d}, \frac{a\bar{\t}+b}{c\bar{\t}+d} \rp
\!=\! (c \t + d)^3 \hspace{-0.4cm}
\sum_{\substack{\b,\l \in \IZ/2\IZ \\ \nu \in \IZ/3\IZ}} \hspace{-0.3cm}
\Psi^*_1 (M)_{\a,\b} \, \Psi_{1}(M)_{\g,\l} \, \Psi^*_{\! A_2} (M)_{\mu,\nu}
\wh{H}_{(1,3), (\b,2\nu+3\l)} (\t, \bar{\t}) 
\end{equation*}
for $M = \pmat{a & b \\ c&d} \in \SL_2(\IZ)$. Moreover, $H_{(1,3), (\a,2\mu+3\g)}$ has the shadow
\begin{equation*}
(2 v)^{\frac{3}{2}} \frac{\del}{\del \bar{\t}}
\wh{H}_{(1,3), (\a,2\mu+3\g)}  (\t, \bar{\t}) 
=
\frac{3i\sqrt{3}}{2 \pi \sqrt{2}} 
\wh{f}_{1,\a} (\t,\bar{\t}) \vth_{3, 2\mu+3\g} (-\bar{\t})
- \frac{3i}{\pi \sqrt{2}}
\wh{f}_{3,2\mu+3\g} (\t,\bar{\t}) \vth_{1, \a} (-\bar{\t}) .
\end{equation*}
The shadow of $f_{1,\a} f_{3,2\mu+3\g}$ has the same form, but with the minus sign replaced by a plus sign.

Combining these facts with \eqref{eq:depth_two_example} then shows that 
\begin{equation*}
\calT_\mu (\t) :=  \sum_{\a \in \IZ/2\IZ} 
\lp H_{(1,3), (\a,2\mu+3\a)} (\t) + f_{1,\a} (\t) f_{3,2\mu+3\a} (\t) \rp
- 8 h_{3,\mu} (\t)
\quad \mbox{for } \mu \in \IZ/3\IZ
\end{equation*}
has a vanishing shadow. 
By \eqref{eq:h_3_0_1_fourier_expansion_vafa_witten}, \eqref{eq:f1_f3_leading_fourier_coefs}, and the fact that 
$H_{(1,3), (\a,2\mu+3\a)}$ has nonzero Fourier coefficients only for positive exponents by \eqref{eq:depth_two_holomorphic_definition}, we have
\begin{equation}\label{eq:calT_fourier_expansion}
\calT_0 (\t) = \frac{1}{9} + O(q)
\andd
\calT_{\pm 1} (\t) = O \lp q^{\frac{2}{3}} \rp .
\end{equation}
In particular, $\calT_\mu$ is a holomorphic vector-valued modular form with weight $3$ and multiplier~$\Psi^*_{\! A_2}$.
Note that the vector space of such modular forms is one-dimensional\footnote{
This can be checked with the dimension formula given in \cite{Bor} or \cite{Bru}.
} 
and is spanned by the theta function $\Theta^{[E_6]}_{\mu}$ for the lattice $E_6$ and its cosets (see \eqref{eq:E6_theta_function}).
Comparing with \eqref{eq:calT_fourier_expansion}, we find $\Theta^{[E_6]}_{\mu}  = 9 \calT_\mu$ and thus
\begin{equation}\label{eq:h_3_mu_decomposition_Eisenstein}
h_{3,\mu} (\t) 
= \frac{1}{8} \sum_{\a \in \IZ/2\IZ} 
\lp H_{(1,3), (\a,2\mu+3\a)} (\t) + f_{1,\a} (\t) f_{3,2\mu+3\a} (\t) \rp
- \frac{1}{72} \Theta^{[E_6]}_{\mu} (\t) .
\end{equation}
Using this relation between $h_{3,\mu}$ and the functions studied in this paper, we prove Corollary~\ref{cor:h_3_mu_coupled_eisenstein_expression}.

\begin{proof}[Proof of Corollary~\ref{cor:h_3_mu_coupled_eisenstein_expression}]
For $\re (s) > \frac{1}{4}$, we split off the contribution of the three terms from the last factor of \eqref{eq:h_3_mu_coupled_eisenstein_expression} 
using \eqref{eq:depth_one_eisenstein} and \eqref{eq:coupled_Eisenstein_definition}
and rewrite
\begin{multline}\label{eq:h_3_mu_coupled_eisenstein_expression_rewriting_details1}
\wh{h}_{3,\mu} (\t, \bar{\t};s)
=
\frac{1}{8} \sum_{\a \in \IZ/2\IZ} 
\lp
\wh{H}_{(1,3),(\a,2\mu +3\a)} (\t, \bar{\t};s)
+
\wh{f}_{1,\a} (\t, \bar{\t}; s) \wh{f}_{3,2\mu+3\a} (\t, \bar{\t}; s) 
\rp
\\[-1.5ex]
-\frac{1}{72}
\frac{v^{2s}}{2}
\sum_{(c,d) \in \LL}
\frac{\Psi_{\! A_2} (M_{c,d})_{0,\mu}}{|c\t+d|^{4s} \lp c\t+d \rp^{3}}.
\end{multline}
Here we employ \eqref{eq:Weil_three_decomposition} and the unitarity of the Weil multiplier $\Psi_m$ to get the last term.
By Theorem~\ref{thm:coupled_Eisenstein_depth_two_analytic_continuation} and Proposition~\ref{prop:depth_one_eisenstein_analytic_continuation}, respectively, the first two terms analytically continue to $\re (s) > -\frac{1}{20}$. For $s=0$ they give
\begin{equation*}
\frac{1}{8} \sum_{\a \in \IZ/2\IZ} 
\lp \wh{H}_{(1,3), (\a,2\mu+3\a)} (\t, \bar{\t}) 
+ \wh{f}_{1,\a} (\t, \bar{\t}) \wh{f}_{3,2\mu+3\a} (\t, \bar{\t}) \rp .
\end{equation*}
By \eqref{eq:h_3_mu_decomposition_Eisenstein}, the claimed analytic continuation and value at $s=0$ for $\wh{h}_{3,\mu} (\t, \bar{\t};s)$ follows if the third term of \eqref{eq:h_3_mu_coupled_eisenstein_expression_rewriting_details1} continues as well and yields $- \frac{1}{72} \Theta^{[E_6]}_{\mu} (\t)$ at $s=0$.
With the term 
\begin{equation}\label{eq:weight_three_eisenstein_series}
\frac{v^{2s}}{2}
\sum_{(c,d) \in \LL}
\frac{\Psi_{\! A_2} (M_{c,d})_{0,\mu}}{|c\t+d|^{4s} \lp c\t+d \rp^{3}},
\end{equation}
we have an ordinary vector-valued Eisenstein series (with a spectral parameter $2s$) transforming with weight $3$ and the dual of the Weil representation for $A_2$.
It already converges to a holomorphic function of $s$ for $\re(s) > -\frac{1}{4}$.
So \eqref{eq:weight_three_eisenstein_series} at $s=0$ yields a weight $3$ holomorphic modular form with multiplier~$\Psi^*_{\! A_2}$ and with constant Fourier coefficient $1$ for $\mu = 0$.
Since the vector space of such modular forms is one-dimensional as discussed above \eqref{eq:h_3_mu_decomposition_Eisenstein}, this identifies the modular form as $\Theta^{[E_6]}_{\mu} (\t)$ and thereby completes the proof of the analytic continuation and the value at $s=0$.

The Fourier expansion for $h_{3,\mu}$, then follows from \eqref{eq:h_3_mu_decomposition_Eisenstein}
along with the Fourier coefficients of $H_{\vc{m},\vc{\mu}}$ given in \eqref{eq:depth_two_holomorphic_fourier_coef_definition} and those of $f_{m, \mu}$ given in Proposition~\ref{prop:depth_one_eisenstein_analytic_continuation}.
\end{proof}

\appendix

\section{Details on $N_{m,\mu} (n, k)$}\label{sec:solving_quadratic_equations} 
Here we give details on $N_{m,\mu} (n, k)$. We restrict $k$ to prime powers, as~$N_{m,\mu} (n, k)$ is multiplicative in $k$ by the Chinese remainder theorem. We start with odd primes and $n \neq 0$.

\begin{lem}\label{lem:quadratic_solution_count3}
Let $r \in \IN_0$, $m \in \IN$, $\mu \in \IZ/2m\IZ$, $n \in \IZ_{m,\mu}$ with $n \neq 0$, and $p$ be an odd prime. 
Then, writing $m = p^\a m_0$ and $-4mn = p^\ell n_0$ with $\a, \ell \in \IN_0$ and $p \nmid m_0, n_0 \in \IZ$, the following hold:\footnote{Here and throughout $\left(\frac{a}{b} \right)$ denotes the Kronecker symbol.}
\begin{enumerate}[leftmargin=*]
\item[\rm(1)] If $p^\a \mid \mu$, then we have $\ell - \a \geq 0$ and
\begin{equation*}
N_{m, \mu} (n, p^r) =
\begin{cases}
p^r 
\quad &\mbox{if } r \leq \a, \ell-\a, \\
p^{\left\lfloor \frac{r+\a}{2} \right\rfloor}
\quad &\mbox{if } \a < r \leq \ell-\a, \\
0
\quad &\mbox{if } \ell-\a < \a, r~\text{or }
\a \leq \ell-\a < r \mbox{ with } \ell \mbox{ odd}, \\
 \lp 1 + \lp \frac{n_0}{p} \rp \rp p^{\frac{\ell}{2}}
\quad &\mbox{if } \a \leq \ell-\a < r \mbox{ with } \ell \mbox{ even}.
\end{cases}
\end{equation*}

\item[\rm(2)] If $\mu = p^\b \mu_0$ with $0 \leq \b < \a$ and $p \nmid \mu_0$, then we have
\begin{equation*}
N_{m, \mu} (n, p^r) =
\begin{cases}
p^\b \ddiv{n+\frac{\mu^2}{4m}}{p^\b}
\quad &\mbox{if } \b < r, \\
p^r \ddiv{n+\frac{\mu^2}{4m}}{p^r}
\quad &\mbox{if } r \leq \b .
\end{cases}
\end{equation*}
\end{enumerate}
\end{lem}
\begin{proof}
For (1), if $p^\a \!\mid\! \mu$, then $-4mn \!\in\! 4 m \IZ + \mu^2$ and so $\ell \!\geq\! \a$ as claimed. Shifting $\nu \!\mapsto\! \nu + \frac{\mu}{p^\a}$ 
in~\eqref{eq:number_congruence_solutions}, 
\begin{equation*}
N_{m,\mu} (n, p^r) = \left| \left\{\nu \pmod{p^r} : \ 
p^\a \nu^2 \equiv p^{\ell-\a} n_0 \pmod{p^r} \right\} \right| .
\end{equation*}
We then use Hensel's lemma to lift solutions between prime powers if $\a \leq \ell-\a < r$.
For (2), we note that the congruence that we need to solve becomes 
\begin{equation*}
p^\a \nu^2 - 2 \mu_0  p^\b \nu + 4 m_0 \lp n + \frac{\mu^2}{4m} \rp \equiv 0 \pmod{p^r},
\end{equation*}
which we study case-by-case for $\b < \a \leq r$, $\b < r < \a$, and $r \leq \b < \a$.
\end{proof}

The case $p=2$ is similar and gives the following result.

\begin{lem}\label{lem:quadratic_solution_count2}
Let $r \in \IN_0$, $m \in \IN$, $\mu \in \IZ/2m\IZ$, and $n \in \IZ_{m,\mu}$ with $n \neq 0$. Writing~$m = 2^\a m_0$ and $-4mn = 2^\ell n_0$ with $\a, \ell \in \IN_0$ and $2 \nmid m_0, n_0 \in \IZ$, the following hold.
\begin{enumerate}[leftmargin=*]
\item[\rm(1)] If $2^{\a + 1} \mid \mu$, then we have $\ell - \a - 2 \geq 0$ and
\begin{equation*}
N_{m, \mu} (n, 2^r) =
\begin{cases}
2^r 
\quad &\mbox{if } r \leq \a \mbox{ and } r \leq \ell-\a-2, \\
2^{\left\lfloor \frac{r+\a}{2} \right\rfloor}
\quad &\mbox{if } \a < r \leq \ell-\a-2, \\
0
\quad &\mbox{if } \ell-\a-2  < \a,r \mbox{ or } \a \leq \ell-\a-2 < r \mbox{ with } \ell \mbox{ odd},  \\
2^{\frac{\ell}{2}-1}
\quad &\mbox{if } \a \leq \ell-\a-2 \mbox{ with } \ell \mbox{ even} \mbox{ and } r=\ell-\a-1, \\
2^{\frac{\ell}{2}} \d_{n_0 \equiv 1 \pmod{4}}
\quad &\mbox{if } \a \leq \ell-\a-2 \mbox{ with } \ell \mbox{ even} \mbox{ and } r=\ell-\a, \\ 
2^{\frac{\ell}{2}+1} \d_{n_0 \equiv 1 \pmod{8}}
\quad &\mbox{if } \a \leq \ell-\a-2 \mbox{ with } \ell \mbox{ even} \mbox{ and } r \geq \ell-\a + 1.
\end{cases}
\end{equation*}

\item[\rm(2)] If $\mu = 2^\b \mu_0$ with $0 \leq \b \leq \a$ and $2 \nmid \mu_0$, then we have 
\begin{equation*}
N_{m, \mu} (n, 2^r) =
\begin{cases}
2^\b \ddiv{n+\frac{\mu^2}{4m}}{2^\b}
\quad &\mbox{if } \b < \a \mbox{ and } \b < r, \\
2^{\b+1} \ddiv{n+\frac{\mu^2}{4m}}{2^{\b+1}}
\quad &\mbox{if } \b = \a < r, \\
2^r \ddiv{n+\frac{\mu^2}{4m}}{2^r}
\quad &\mbox{if } r \leq \b \leq \a . \\
\end{cases}
\end{equation*}
\end{enumerate}
\end{lem}

Finally, we consider the case $n=0$.

\begin{lem}\label{lem:quadratic_solution_count_n_zero}
Let $r \in \IN_0$, $m \in \IN$, $\mu \in \IZ/2m\IZ$, and $p$ be a prime.
Writing $m = p^\a m_0$ with $\a \in \IN_0$ and $p \nmid m_0 \in \IN$, the following hold:
\begin{enumerate}[leftmargin=*]
\item[\rm(1)] If $0 \!\in\! \IZ_{m,\mu}\!$ (equiv.~$4m \! \mid \! \mu^2$), then either $2 p^\a \!\mid\! \mu$ or $\mu = 2 p^\b \mu_1$ with $0 \leq 2\b-\a < \b < \a$ and~$p \nmid \mu_1$.

\item[\rm(2)] If $4m \mid \mu^2$ and $2 p^\a \mid \mu$, then
\begin{equation*}
N_{m, \mu} (0, p^r) =
\begin{cases}
p^r \quad &\mbox{if } r \leq \a, \\
p^{\left\lfloor \frac{r+\a}{2} \right\rfloor}
\quad &\mbox{if } r > \a. 
\end{cases}
\end{equation*}

\item[\rm(3)] If $4m \mid \mu^2$ and $\mu = 2 p^\b \mu_1$ with $0 \leq 2\b-\a < \b < \a$ and $p \nmid \mu_1$, then 
\begin{equation*}
N_{m, \mu} (0, p^r) =
\begin{cases}
p^r \quad &\mbox{if } r \leq 2\b-\a, \\
0
\quad &\mbox{if } r > 2\b-\a. 
\end{cases}
\end{equation*}
\end{enumerate}
\end{lem}

\section{Numerical Checks}\label{sec:numerical_checks}

In this section, we numerically check the identity given in Corollary~\ref{cor:h_3_mu_coupled_eisenstein_expression} for the coefficients $c_\mu (n)$ of the depth two mock modular form $h_{3,\mu}$ from Vafa--Witten invariants. Thus for $\mu \in \IZ/3\IZ$, $n \in \IN_0 + \epsilon_\mu$, and $N \in \IN$, we define the truncated sum \smash{$c^{[N]}_\mu (n)$} using the definition of $c_\mu (n)$ in \eqref{eq:c_mu_n_definition_from_corollary}, by restricting the first sum to $|\ell| \leq 2N$ and the second to $|\ell| \leq N$.
Both the first and second terms converge as $N \to \infty$. We choose the cutoff values $N$ and $2N$ for faster numerical convergence. We leave a detailed study of these properties for future work.
Figure~\ref{fig:coef_c_mu_50_149q3} displays \smash{$c_{0}^{[N]} (50)$} and \smash{$c_{1}^{[N]} (\frac{149}{3})$}.

\begin{figure}[H]
 \vspace{-8pt}
  \centering
    \includegraphics[scale=0.32]{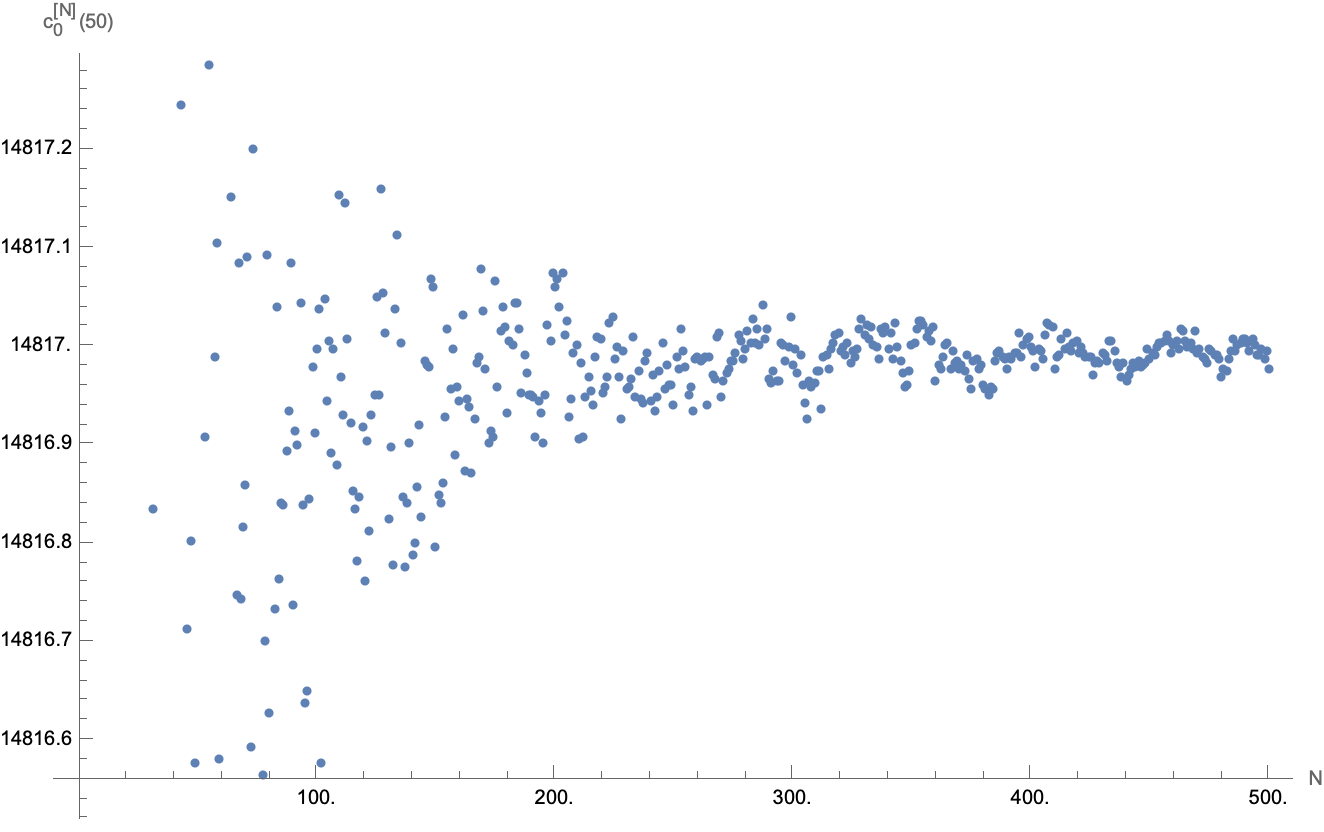}
    \includegraphics[scale=0.32]{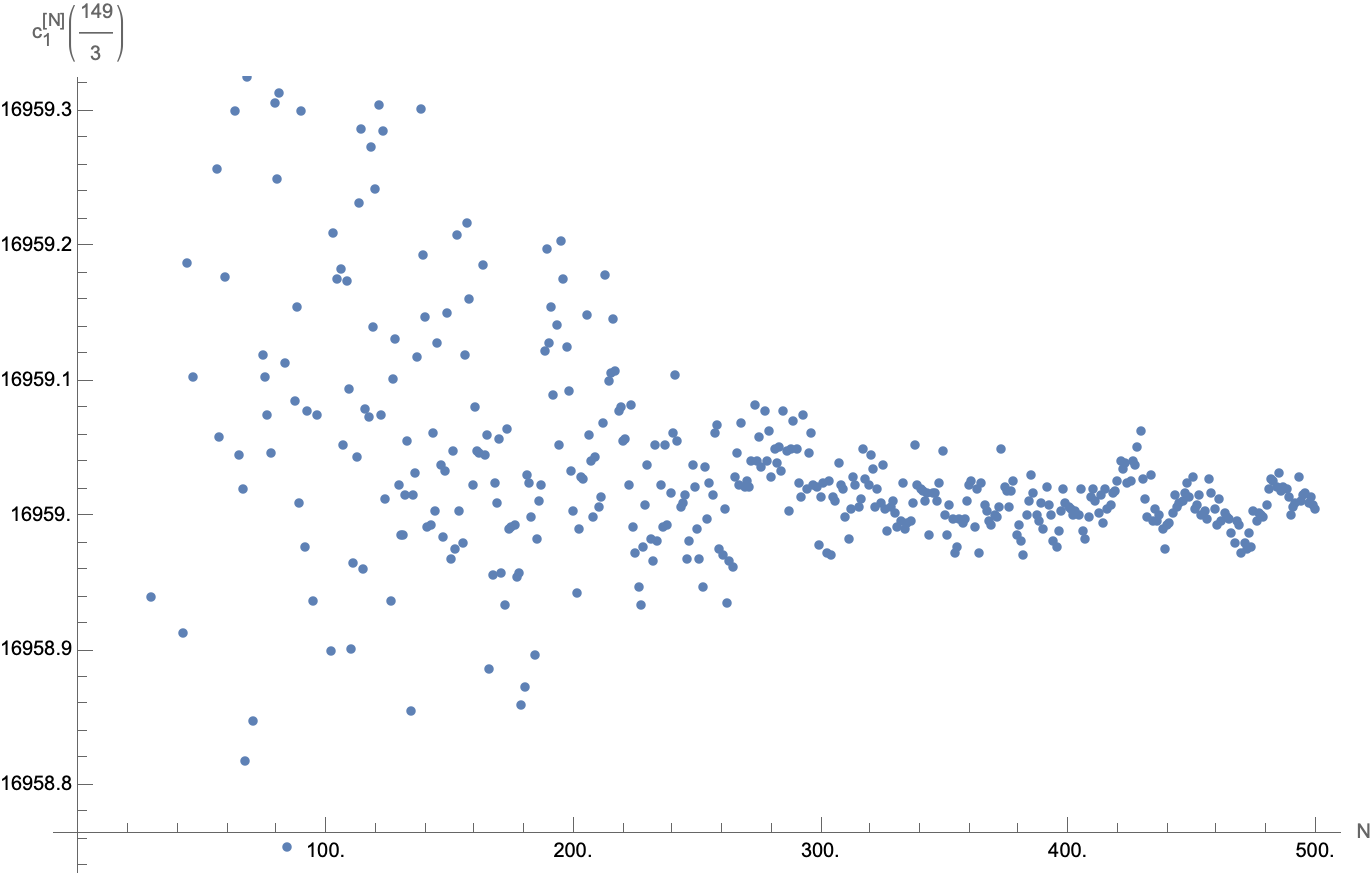}
     \vspace{-10pt}
    \caption{}
    \label{fig:coef_c_mu_50_149q3}
    \vspace{-15pt}
\end{figure}

Compare these with $c_{0} (50) = 14817$ and $c_{1} (\frac{149}{3}) = 16959$ (see Table 2 of \cite{CM}). Indeed,
\begin{equation*}
c_{0}^{[2500]} (50) = 14816.99974\ldots
\andd
c_{1}^{[2500]} \lp \frac{149}{3} \rp = 16959.00048\ldots .
\end{equation*}

\end{document}